\documentclass[a4paper,11pt]{article}

\usepackage{amsfonts,amsmath,amssymb,amsthm,enumerate,tikz-cd}
\usepackage{times}
\usepackage[margin=1in]{geometry}

\newtheorem{theorem}{Theorem}
\newtheorem{lemma}[theorem]{Lemma}
\newtheorem{proposition}[theorem]{Proposition}
\newtheorem{corollary}[theorem]{Corollary}
\newtheorem{theorema}{Theorem}
\newtheorem{corollarya}[theorema]{Corollary}

\theoremstyle{definition}
\newtheorem{remark}[theorem]{Remark}
\newtheorem{definition}[theorem]{Definition}
\newtheorem{notation}[theorem]{Notation}

\newcommand{\acos}{\mskip2mu{\rm acos}\mskip1mu}
\newcommand{\rmi}{\mskip2mu{\rm i}\mskip1mu}

\newcommand{\Order}{\mathop{\rm O}\nolimits}
\newcommand{\Sign}{\mathop{\rm sign}\nolimits}
\newcommand{\rmd}{\,{\rm d}}
\newcommand{\rme}{{\rm e}}
\newcommand{\Interior}{\mathop {\rm Int}\nolimits}
\newcommand{\Volume}{\mathop {\rm V}\nolimits}

\newcommand{\stretches}{\rightsquigarrow}

\def\J{\mathcal{J}}
\def\Cset{\mathbb{C}}
\def\Nset{\mathbb{N}}
\def\Rset{\mathbb{R}}
\def\Tset{\mathbb{T}}
\def\Zset{\mathbb{Z}}

\def\nf{\mathfrak{n}}
\def\pf{\mathfrak{p}}
\def\qf{\mathfrak{q}}
\def\tf{\mathfrak{t}}
\def\Pf{\mathfrak{P}}
\def\Qf{\mathfrak{Q}}
\def\Tf{\mathfrak{T}}
\def\nv{\boldsymbol{n}}
\def\ov{\boldsymbol{o}}
\def\pv{\boldsymbol{p}}
\def\qv{\boldsymbol{q}}
\def\tv{\boldsymbol{t}}
\def\uv{\boldsymbol{u}}
\def\vv{\boldsymbol{v}}
\def\wv{\boldsymbol{w}}
\def\xv{\boldsymbol{x}}
\def\yv{\boldsymbol{y}}
\def\Qv{\boldsymbol{Q}}

\title{Chaotic properties of billiards in circular
       polygons\footnote{To appear
       in \emph{Communications in Mathematical Physics}}}

\author{Andrew Clarke\footnote{Universitat
        Polit\`ecnica de Catalunya, Barcelona, Spain
        (\texttt{andrew.michael.clarke@upc.edu})}
        \and
        Rafael Ram{\'\i}rez-Ros\footnote{Universitat
        Polit\`ecnica de Catalunya, Barcelona, Spain
        (\texttt{rafael.ramirez@upc.edu})}}

\date{July 31, 2024}

\begin{document}

\maketitle

\abstract{We study billiards in domains enclosed by circular polygons.
These are closed $C^1$ strictly convex curves
formed by finitely many circular arcs.
We prove the existence of a set in phase space,
corresponding to generic sliding trajectories close enough
to the boundary of the domain,
in which the return billiard dynamics is semiconjugate to
a transitive subshift on infinitely many symbols that contains
the full $N$-shift as a topological factor for any $N \in \Nset$,
so it has infinite topological entropy.
We prove the existence of uncountably many asymptotic generic sliding
trajectories approaching the boundary with optimal uniform linear speed,
give an explicit exponentially big (in $q$) lower bound on
the number of $q$-periodic trajectories as $q \to \infty$,
and present an unusual property of the length spectrum.
Our proofs are entirely analytical.}

\medskip

\noindent
\textbf{Keywords:}
Billiards, circular polygons, chaos, symbolic dynamics, periodic trajectories,
length spectrum

\section{Introduction}

A \emph{billiard problem} concerns the motion of a particle inside
the domain bounded by a closed plane curve $\Gamma$
(or the domain bounded by a hypersurface of some higher-dimensional
 Euclidean space).
The motion in the interior of the domain is along straight lines,
with elastic collisions at the boundary according to the optical
law of reflection: the angles of incidence and reflection are equal.
These dynamical systems were first introduced by
Birkhoff~\cite{birkhoff1922surface}.
See~\cite{KozlovTreschev1991,Tabachnikov1995,ChernovMarkarian2006}
for a general description.

In the case of dispersing billiards (i.e. when the boundary is a
union of concave components),
the dynamics is chaotic~\cite{sinai1970dynamical};
indeed, such billiards exhibit ergodicity, the Bernoulli property,
sensitive dependence on initial conditions, and so forth.
In fact, it was believed for some years that billiards without any
dispersing walls could not display chaos.

Thus it came as a surprise when Bunimovich, in his famous paper,
detailed a proof that the billiard in a stadium exhibits
the Bernoulli property~\cite{bunimovich1979ergodic}.
The boundary of the stadium billiard consists of two
straight parallel lines connected at either end by semicircles.
Stadia are $C^1$ and convex, but not $C^2$ or strictly convex.
We study the class of $C^1$ \emph{strictly} convex billiards bounded
by finitely many circular arcs.
No billiard in this class satisfies the celebrated
\emph{B-condition}---that is, the condition that all circular arcs
can be completed to a disk within the billiard domain---,
which is the hallmark for the \emph{defocusing mechanism}
in billiards whose focusing boundaries are all circular
arcs~\cite[Section 8.3]{ChernovMarkarian2006}.
In spite of this, we observe several chaotic phenomena.

Denote by $\Gamma$ a $C^1$, closed, strictly convex curve in $\Rset^2$.
The phase space of the billiard inside
the domain bounded by $\Gamma$ is
the 2-dimensional cylinder $\mathcal{M} = \Tset \times [0, \pi]$;
the angular component of the cylinder is a parameter on
the curve $\Gamma$,
and the height component is the angle of incidence/reflection.
We denote by $f:\mathcal{M} \to \mathcal{M}$ the billiard map
(i.e. the collision map of the billiard flow with the
 boundary~$\Gamma$ of the domain;
 see Section~\ref{sec:Fundamental} for a precise definition).
We say that the curve $\Gamma$ is a \emph{circular polygon}
if it is a union of a finite number of circular arcs,
concatenated in such a way that $\Gamma$ is strictly convex,
and at the points where two circular arcs with different radii meet,
the tangents agree, so that $\Gamma$ is $C^1$, but not $C^2$.
We do not consider circles as circular polygons.
A \emph{circular $k$-gon} is a circular polygon
with exactly $k$ circular arcs.
Note that a circular polygon cannot have fewer than 4 circular arcs,
so if $\Gamma$ is a circular $k$-gon, then $k \ge 4$
(see Lemma~\ref{lemma_geq4arcs} in Section~\ref{sec:CircularPolygons}
 for a proof of this fact).

In this paper we are interested in the \emph{sliding} trajectories
that do not skip any arc in any of their infinite turns
around $\Gamma$.
These are trajectories close to
the boundary $\partial \mathcal M = \Tset \times \{0,\pi\}$.
That is, trajectories where the angle of incidence/reflection
is close to $0$ or $\pi$ (see Definition~\ref{def:GenericSliding}).

In what follows we give heuristic statements of our main results.

\begin{theorema}\label{thm:SymbolicDynamicsIntro}
If $\Gamma$ is a circular polygon,
then there is a set $\mathcal{J} \subset \mathcal{M}$
accumulating on $\partial \mathcal{M}$ such that the return map
$F:\mathcal{J} \to \mathcal{J}$ of $f$ to $\mathcal{J}$
is topologically semiconjugate to a transitive subshift on
infinitely many symbols that contains the full $N$-shift as
a topological factor for any $N \in \Nset$,
so it has infinite topological entropy.
\end{theorema}

See Proposition~\ref{prop:ShiftMapProp} in Section~\ref{sec:Symbols} and
Theorem~\ref{thm:SymbolicDynamics} in Section~\ref{sec:ChaoticMotions}
for a precise formulation of this theorem.
Be aware that the map with infinite entropy is the return map $F$,
not the billiard map $f$.

The \emph{final sliding motions} are the possible qualitative behaviors
that a sliding billiard trajectory posses as the number of impacts
tends to infinity, forward or backward.
Every forward counter-clockwise sliding billiard trajectory
$(\varphi_n,\theta_n) = f^n(\varphi,\theta)$,
where $\varphi_n$ are the angles of impact and
$\theta_n$ are the angles of incidence/reflection,
belongs to exactly one of the following three classes:
\begin{itemize}
\item
\emph{Forward bounded} ($\mathcal{B}_0^+$):
$\inf_{n \ge 0} \theta_n > 0$;
\item
\emph{Forward oscillatory} ($\mathcal{O}_0^+$):
$0  = \liminf_{n \to +\infty} \theta_n <
\limsup_{n \to +\infty} \theta_n$; and
\item
\emph{Forward asymptotic} ($\mathcal{A}_0^+$):
$\lim_{n \to +\infty} \theta_n = 0$.
\end{itemize}
This classification also applies for backward counter-clockwise
sliding trajectories when $n \le -1$ and $n \to -\infty$,
in which case we write a superindex $-$ instead of $+$ in each of
the classes: $\mathcal{B}_0^-$, $\mathcal{O}_0^-$ and $\mathcal{A}_0^-$.
And it also applies to (backward or forward) clockwise sliding
trajectories,
in which case we replace $\theta_n$ with $|\theta_n - \pi|$ in
the definitions above and we write a subindex $\pi$ instead of $0$ in each
of the classes: $\mathcal{B}_\pi^\pm$,
$\mathcal{O}_\pi^\pm$ and $\mathcal{A}_\pi^\pm$.

Terminologies \emph{bounded} and \emph{oscillatory}
are borrowed from Celestial Mechanics.
See, for instance, \cite{Guardia_etal2022}.
In our billiard setting, bounded means bounded away
from  $\theta = \pi$ in the clockwise case and
bounded away from $\theta = 0$ in the counter-clockwise case.
That is, a sliding billiard trajectory is bounded when it does
not approach $\partial \mathcal{M}$.

The following corollary is an immediate consequence of
Theorem~\ref{thm:SymbolicDynamicsIntro},
see Section~\ref{sec:ChaoticMotions}.

\begin{corollarya}
\label{cor:PossibleBehaviors}
If $\Gamma$ is a circular polygon, then
$\mathcal{X}_\lambda^- \cap \mathcal{Y}_\lambda^+ \neq \emptyset$
for
$\mathcal{X},\mathcal{Y} \in \{\mathcal{B}, \mathcal{O}, \mathcal{A}\}$
and $\lambda \in \{0,\pi\}$.
\end{corollarya}

From now on, we focus on counter-clockwise sliding trajectories.
Corollary~\ref{cor:PossibleBehaviors} does not provide
data regarding the maximal \emph{speed} of diffusion for asymptotic
trajectories.
Which is the faster way in which $\theta_n \to 0$
for asymptotic sliding trajectories?
The answer is provided in the following theorem.

\begin{theorema}\label{thm:BoundaryIntro}
If $\Gamma$ is a circular polygon,
then there are uncountably many asymptotic generic sliding billiard
trajectories that approach the boundary with uniform linear speed.
That is, there are constants $0 < a < b$ such that if
$\{ \theta_n \}_{n \in \Zset}$ is
the corresponding sequence of angles of incidence/reflection of any of
these uncountably many asymptotic generic sliding billiard trajectories,
then
\[
a |n| \le 1/\theta_n \le b |n|, \qquad \forall |n| \gg 1.
\]
Linear speed is optimal.
That is, there is no billiard trajectory such that
\[
\lim_{n \to +\infty} n \theta_n = 0.
\]
\end{theorema}

See Theorem~\ref{thm:Boundary} in Section~\ref{sec:OptimalSpeed},
where we also get uncountably many one-parameter families (paths)
of forward asymptotic generic sliding billiard trajectories, for a
more detailed version of Theorem~\ref{thm:BoundaryIntro}.
The definition of \emph{generic} billiard trajectories is a bit
technical, see Definition~\ref{def:GenericSliding}
and Remark~\ref{rem:GenericTrajectories}.
The term \emph{uniform} means that the constants $0 < a < b$
do not depend on the billiard trajectory.
The term \emph{linear} means that $1/\theta_n$ is bounded between
two positive multiples of $|n|$.
\emph{Optimality} comes as no surprise since
$\sum_{n \ge 0} \theta_n = +\infty$
for any billiard trajectory in any circular polygon---or, for that matter,
in any strictly convex billiard table whose billiard flow is defined
for all time~\cite{halpern1977strange}.
Optimality is proved in Proposition~\ref{prop:OptimalSpeed}
in Section~\ref{sec:OptimalSpeed}.

There are two key insights (see Section~\ref{sec:Fundamental})
behind this theorem.
First, when we iterate $f$ along one of the circular arcs
of the circular polygon $\Gamma$, the angle of reflection $\theta$
is constant, so $\theta$ can drop only when the trajectory crosses
the singularities between consecutive circular arcs.
Second, the maximal drop corresponds to multiplying $\theta$
by a uniform (in $\theta$) factor that is smaller than one.
As we must iterate the map many (order $ \theta^{-1}$) times
to slide fully along each circular arc,
we cannot approach the boundary with a faster than linear speed.

As Theorem~\ref{thm:SymbolicDynamicsIntro} gives us only a
topological \emph{semi}conjugacy to symbolic dynamics,
it does not immediately provide us with the abundance of
periodic orbits that the shift map possesses.
However our techniques enable us to find many
periodic sliding billiard trajectories.
We state in the following theorem that the number of such trajectories
in circular polygons grows exponentially with respect to the period.
In contrast, Katok~\cite{Katok1987} showed that the numbers
of isolated periodic billiard trajectories and
of parallel periodic billiard trajectories grow
subexponentially in any (linear) polygon.

Given any integers $1 \le p < q$,
let $\Pi(p,q)$ be the set of $(p,q)$-periodic billiard trajectories
in $\Gamma$.
That is, the set of periodic trajectories that close
after $p$ turns around $\Gamma$ and $q$ impacts in $\Gamma$,
so they have rotation number $p/q$.
Let $\Pi(q) = \cup_{1 \le p < q} \Pi(p,q)$
be the set of periodic billiard trajectories with period $q$.
The symbol~$\#$ denotes the \emph{cardinality} of a set.

\begin{theorema}\label{thm:PeriodicIntro}
If $\Gamma$ is a circular $k$-gon and $p \in \Nset$, there are constants
$c_\star(p), M_\star, h_\star > 0$ such that
\begin{enumerate}[(a)]
\item
$\# \Pi (p,q) \ge c_\star(p) q^{kp-1} + \Order(q^{kp-2})$
as $q \to \infty$ for any fixed $p \in \Nset$; and
\item
$\# \Pi (q) \ge M_\star \rme^{h_\star q}/q$
as $q \to + \infty$.
\end{enumerate}
\end{theorema}

We give explicit expressions for $c_\star(p)$,
$M_\star$ and $h_\star$ in Section~\ref{sec:PeriodicTrajectories}.
The optimal value of $c_\star(p)$ is equal to the volume
of a certain $(kp-1)$-dimensional compact convex polytope with
an explicitly known half-space representation,
see Proposition~\ref{prop:Optimalc}.
We do not give optimal values of $M_\star$ and $h_\star$.
The relation between the optimal value of $h_\star$ and the
topological entropy of the billiard map $f$ is an open problem.
We acknowledge that some periodic trajectories in $\Pi(p,q)$
may have period less than $q$ when $\gcd(p,q) \neq 1$,
but they are a minority,
so the previous lower bounds capture the growth rate
of the number of periodic trajectories with rotation number $p/q$
and minimal period $q$ even when $p$ and $q$ are not coprime.

If the circular polygon has some symmetry,
we can perform the corresponding natural reduction to count the
number of symmetric sliding periodic trajectories,
but then the exponent $kp-1$ in the first lower bound would be smaller
because there are fewer reduced arcs than original arcs.
The exponent $h_\star$ would be smaller too.
See~\cite{CasasRamirezRos2012,Goncalves_etal2024} for examples
of symmetric periodic trajectories in other billiards.
The first reference deals with axial symmetries.
The second one deals with rotational symmetries.

Let $|\Gamma|$ be the length of $\Gamma$.
If $g = \{ z_0,\ldots,z_{q-1}\} \subset \Gamma$ is a
$q$-periodic billiard trajectory,
let $L(g) = |z_1 - z_0| + \cdots + |z_{q-1} - z_0|$
be its \emph{length}.
If $(g_q)_q$ is any sequence such that $g_q \in \Pi(1,q)$,
then $\lim_{q \to +\infty} L(g_q) = |\Gamma|$.
There are so many generic sliding $(1,q)$-periodic billiard trajectories
inside circular polygons that we can find sequences
$(g_q)_q$ such that the differences $L(g_q) - |\Gamma|$
have rather different asymptotic behaviors as $q \to +\infty$.

\begin{theorema}
\label{thm:AsymptoticConstantIntro}
If $\Gamma$ is a circular polygon,
then there are constants $c_- < c_+ < 0$
such that for any fixed $c \in [c_-,c_+]$ there exists a
sequence $(g_q)_q$, with $g_q \in \Pi(1,q)$, such that
\[
L(g_q) =
|\Gamma| + c/q^2 + \Order(1/q^3),\quad
\mbox{as $q \to +\infty$.}
\]
Consequently, there exists a
sequence $(h_q)_q$, with $h_q \in \Pi(1,q)$, such that
\[
c_- = \liminf _{q \to +\infty} \big( (L(h_q) - |\Gamma|)q^2 \big) <
\limsup _{q \to +\infty} \big( (L(h_q) - |\Gamma|)q^2 \big) = c_+,\quad
\mbox{as $q \to +\infty$.}
\]
Besides,
$c_- \le -\pi^2 |\Gamma|/6$ and
$c_+ = -\frac{1}{24}
\left[ \int_{\Gamma} \kappa^{2/3}(s) \rmd s \right]^3$,
where $\kappa(s)$ is the curvature of $\Gamma$ as a function
of an arc-length parameter $s \in [0,|\Gamma|)$.
\end{theorema}

Let us put these results into perspective by comparing them
with the observed behavior in sufficiently smooth
(say $C^6$) and strictly convex billiards,
which for the purpose of this discussion we refer to
as \emph{Birkhoff billiards}.
Lazutkin's theorem (together with a refinement due to Douady)
implies that Birkhoff billiards possess a family of
caustics\footnote{A closed curve $\gamma$ contained in
the interior of the region bounded by $\Gamma$ is called
a \emph{caustic} if it has the following property:
if one segment of a billiard trajectory is tangent to $\gamma$,
then so is every segment of that trajectory.}
accumulating on the
boundary~\cite{douady1982applications,lazutkin1973existence}.
These caustics divide the phase space into invariant regions,
and therefore guarantee a certain \emph{regularity} of
the dynamics near the boundary,
in the sense that the conclusion of Theorem~\ref{thm:SymbolicDynamicsIntro}
never holds for Birkhoff billiards.
Not only does the conclusion of Theorem~\ref{thm:BoundaryIntro}
not hold for Birkhoff billiards,
but in such systems there are no trajectories approaching the boundary
asymptotically as the orbits remain in invariant regions bounded
by the caustics.
As for Theorem~\ref{thm:PeriodicIntro},
a well-known result of Birkhoff~\cite{birkhoff1922surface}
implies that Birkhoff billiards have $\# \Pi(p,q) \ge 2$
for each coprime pair $p, q$ such that $1 \leq p < q$.
This lower bound turns out to be sharp,
in the sense that for any such pair $p,q$,
there exist Birkhoff billiards with exactly two
geometrically distinct periodic orbits of
rotation number $p/q$~\cite{pinto2012billiards};
a simple example is that the billiard in a non-circular ellipse
has two periodic orbits of rotation number $1/2$,
corresponding to the two axes of symmetry.
It follows that the conclusion of
Theorem~\ref{thm:PeriodicIntro} does not hold in general
for Birkhoff billiards.
Finally,
as for Theorem~\ref{thm:AsymptoticConstantIntro},
a well-known result of Marvizi-Melrose~\cite{marvizi1982spectral}
implies that if $(g_q)_q$, with $g_q \in \Pi(1,q)$,
is any sequence of periodic billiard trajectories in
a Birkhoff billiard $\Gamma$, then
\[
L(g_q) =
|\Gamma| + c_+/q^2 + \Order(1/q^4),\quad
\mbox{as $q \to +\infty$,}
\]
where
$c_+ = -\frac{1}{24}
\left[\int_{\Gamma} \kappa^{2/3}(s) \rmd s \right]^3$.
Hence, $(1,q)$-periodic billiard trajectories
in circular polygons are \emph{asymptotically shorter}
than the ones in Birkhoff billiards.

An interesting question in general that has been considered to a
significant extent in the literature is:
what happens to the caustics of Lazutkin's theorem,
and thus the conclusions of Theorems~\ref{thm:SymbolicDynamicsIntro}
and~\ref{thm:BoundaryIntro},
if we loosen the definition of a Birkhoff billiard?
Without altering the basic definition of the billiard map $f$,
there are three ways that we can generalise Birkhoff billiards:
\eqref{item_loose1}~by relaxing the strict convexity hypothesis,
\eqref{item_loose2}~by relaxing the smoothness hypothesis,
or \eqref{item_loose3}~by increasing the dimension of
the ambient Euclidean space.

\begin{enumerate}[(i)]
\item
\label{item_loose1}
Mather proved that if the boundary is convex and $C^r$ for $r \geq 2$,
but has at least one point of zero curvature,
then there are no caustics and there exist trajectories
which come arbitrarily close to being positively tangent to the boundary
and also come arbitrarily close to being negatively tangent to the
boundary~\cite{mather1982glancing}.
Although this result is about finite segments of billiard trajectories,
there are also infinite trajectories tending to the boundary
both forward and backward in time in such
billiards: $\mathcal{A}_0^\pm \cap \mathcal{A_\pi^\mp} \neq \emptyset$, see~\cite{mather1991variational}.

\item
\label{item_loose2}
Despite six continuous derivatives being
the stated smoothness requirement for Lazutkin's
theorem~\cite{douady1982applications,lazutkin1973existence},
there is some uncertainty regarding what happens for $C^5$ boundaries,
and in fact it is generally believed that 4 continuous derivatives
should suffice.
Halpern constructed billiard tables that are strictly convex
and $C^1$ but not $C^2$ such that the billiard particle
experiences an infinite number of collisions in finite
time~\cite{halpern1977strange}; that is to say,
the billiard flow is incomplete.
This construction does not apply to our case,
as our billiard boundaries have only a finite number of singularities
(points where the boundary is only $C^1$ and not $C^2$),
whereas Halpern's billiards have infinitely many.

The case of boundaries that are strictly convex and $C^1$
but not $C^2$ and have only a finite number (one, for example)
of singularities was first considered by
Hubacher~\cite{hubacher1987instability},
who proved that such billiards have no caustics in a neighborhood
of the boundary. This result opens the door for our analysis.

\item
\label{item_loose3}
It has been known since the works of Berger and Gruber that
in the case of strictly convex and sufficiently smooth billiards
in higher dimension (i.e. the billiard boundary is a codimension~1
submanifold of $\mathbb{R}^d$ where $d \geq 3$),
only ellipsoids have caustics~\cite{berger1995seules,gruber1995only}.
However Gruber also observed that in this case,
even in the absence of caustics,
the Liouville measure of the set of trajectories approaching
the boundary asymptotically is zero~\cite{gruber1990convex}.
The question of existence of such trajectories was thus left open.

It was proved in~\cite{clarke2022arnold}
(combined with results of~\cite{clarke2019generic})
that generic strictly convex analytic billiards in $\Rset^3$
(and `many' such billiards in $\Rset^d$ for $d \geq 4$)
have trajectories approaching the boundary asymptotically.
It is believed that the meagre set of analytic strictly convex
billiard boundaries in $\Rset^d$, $d \geq 3$, for which these
trajectories do not exist consists entirely of ellipsoids,
but the perturbative methods of~\cite{clarke2022arnold}
do not immediately extend to such a result.
\end{enumerate}

Billiards in circular polygons have been studied numerically in the
literature~\cite{benettin1978numerical, dullin1996two,
hayli1986experiences, henon1983benettin, makino2019bifurcation}.
In the paper~\cite{balint2011chaos} the authors use
numerical simulations and semi-rigorous arguments to study
billiards in a 2-parameter family of circular polygons.
They conjecture that, for certain values of the parameters,
the billiard is ergodic.
In addition they provide heuristic arguments in favor of
this conjecture.

A related problem is the lemon-shaped billiard,
which is known to display chaos~\cite{bunimovich2016another,
chen2013ergodicity, jin2020hyperbolicity}.
These billiards are strictly convex but not $C^1$,
so the billiard map is well-defined only on a proper subset of
the phase space.

The \emph{elliptic flowers} recently introduced by
Bunimovich~\cite{Bunimovich2022} are closed $C^0$ curves
formed by finitely many pieces of ellipses.
\emph{Elliptic polygons} are elliptic flowers
that are $C^1$ and strictly convex,
so they are a natural generalisation of circular polygons.
One can obtain a 1-parameter family of elliptic polygons
with the string construction from any convex (linear) polygon.
The \emph{string construction} consists of wrapping
an inelastic string around the polygon and tracing
a curve around it by keeping the string taut.
Billiards in elliptic polygons can be studied with
the techniques presented here for circular polygons.
We believe that all results previously stated in this introduction,
with the possible exception of the inequality $c_- \le -\pi^2|\Gamma|/6$
given in Theorem~\ref{thm:AsymptoticConstantIntro},
hold for generic elliptic polygons.
However, there are elliptic polygons that are globally~$C^2$,
and not just~$C^1$, the \emph{hexagonal string billiard} first studied by
Fetter~\cite{fetter2012numerical}
being the most celebrated example.
We do not know how to deal with $C^2$ elliptic polygons
because jumps in the curvature of the boundary are a key
ingredient in our approach to get chaotic billiard dynamics.
Fetter suggested that the hexagonal string billiard could be integrable,
in which case it would be a counterexample to the Birkhoff
conjecture\footnote{It is well-known that billiards in ellipses
are integrable. The so-called \emph{Birkhoff conjecture} says that
elliptical billiards are in fact the only integrable billiards.
This conjecture, in its full generality, remains open.}.
However, such integrability was numerically put in doubt
in \cite{bialy2022numerical}.

In what follows we describe the main ideas of our proofs.
Let $\Gamma$ be a circular $k$-gon.
It is well-known that the angle of incidence/reflection is a
constant of motion for billiards in circles.
Therefore, for the billiard in $\Gamma$, the angle of incidence/reflection
can change only when we pass from one circular arc to another,
and not when the billiard has consecutive impacts on
the same circular arc.
The main tool that we use to prove our theorems is what we call
the \emph{fundamental lemma} (Lemma~\ref{lem:Fundamental} below),
which describes how trajectories move up and down
after passing from one circular arc to the next.
The phase space~$\mathcal{M}$ of the billiard map~$f$ is a cylinder,
with coordinates $(\varphi,\theta)$
where $\varphi \in \Tset$ is a parameter on the boundary $\Gamma$,
and where $\theta \in [0, \pi]$ is the angle of incidence/reflection.
We consider two vertical segments $\mathcal{L}_j$ and
$\mathcal{L}_{j+1}$ in~$\mathcal{M}$
corresponding to consecutive singularities of $\Gamma$,
and sufficiently small values of $\theta$.
The index~$j$ that labels the singularities is defined modulo $k$.
The triangular region $\mathcal{D}_j$ bounded by
$\mathcal{L}_j$ and $f(\mathcal{L}_j)$ is a fundamental domain
of the billiard map~$f$;
that is, a set with the property that sliding trajectories have
exactly one point in $\mathcal{D}_j$ on each turn around $\Gamma$.
Consider now the sequence of backward iterates
$\big\{ f^{-n}(\mathcal{L}_{j+1}) \big\}$ of $\mathcal{L}_{j+1}$.
This sequence of slanted segments divides the fundamental domain
$\mathcal{D}_j$ into infinitely many quadrilaterals,
which we call \emph{fundamental quadrilaterals}.
The fundamental lemma describes which fundamental quadrilaterals
in $\mathcal{D}_{j+1}$ we can visit if we start in a given
fundamental quadrilateral in $\mathcal{D}_j$.

In order to prove Theorem~\ref{thm:SymbolicDynamicsIntro},
we apply the fundamental lemma iteratively to describe how
trajectories visit different fundamental quadrilaterals consecutively
in each of the $k$ fundamental domains $\mathcal{D}_j$ in $\mathcal{M}$.
A particular coding of possible sequences of $k$ fundamental
quadrilaterals that trajectories can visit gives us our symbols.
We then use a method due to
Papini and Zanolin~\cite{papini2004fixed,papini2004periodic}
(extended to higher dimensions by Pireddu and
Zanolin~\cite{pireddu2009fixed,pireddu2007cutting,pireddu2008chaotic})
to prove that the billiard dynamics is semiconjugate to
a shift map on the sequence space of this set of symbols;
this method is called \emph{stretching along the paths}.
Observe that we could equally have used the method of
correctly aligned windows~\cite{arioli2001symbolic, gidea2004covering, 
zgliczynski2004covering}, or the crossing number
method~\cite{kennedy2001topological};
note however that the latter would not have provided us with
the large amount of periodic orbits that the other two methods do.
We note that, although Theorem~\ref{thm:SymbolicDynamicsIntro}
provides us with a topological semiconjugacy to symbolic dynamics,
we expect that this could be improved to a full conjugacy by
using other methods.

Once the proof of Theorem~\ref{thm:SymbolicDynamicsIntro} is completed,
Theorems~\ref{thm:BoundaryIntro}, \ref{thm:PeriodicIntro},
and~\ref{thm:AsymptoticConstantIntro} are proved by
combining additional arguments with the symbolic dynamics
we have constructed.
With respect to the symbolic dynamics,
we choose a coding of the fundamental quadrilaterals visited
by a trajectory that corresponds to $\theta$
tending to $0$ in the fastest way possible.
We then prove that the corresponding billiard trajectories satisfy
the conclusion of Theorem~\ref{thm:BoundaryIntro}.
As for Theorem~\ref{thm:PeriodicIntro},
the method of stretching along the paths guarantees
the existence of a periodic billiard trajectory for
every periodic sequence of symbols.
Consequently, the proof of Theorem~\ref{thm:PeriodicIntro}
amounts to counting the number of sequences of symbols that
are periodic with period $p$ (because each symbol describes
one full turn around the table; see Section~\ref{sec:Symbols} for details)
such that the corresponding periodic sliding billiard trajectories
after $p$ turns around the table have rotation number $p/q$.
It turns out that this reduces to counting the number of
integer points whose coordinates sum $q$ in a certain
$kp$-dimensional convex polytope.
We do this by proving that the given convex polytope contains a
hypercube with sides of a certain length,
and finally by counting the number of integer points whose
coordinates sum $q$ in that hypercube.

The structure of this paper is as follows.
In Section~\ref{sec:CircularPolygons} we describe
the salient features of circular polygons.
We summarise the \emph{stretching along the paths} method in
Section~\ref{sec:Stretching}.
In Section~\ref{sec:Fundamental} we define the billiard map and
dicuss the billiard dynamics in circular polygons,
before giving the definition of fundamental quadrilaterals,
as well as the statement and proof of the fundamental lemma.
Symbols are described in Section~\ref{sec:Symbols}.
Chaotic dynamics,
and thus the proofs of Theorem~\ref{thm:SymbolicDynamicsIntro} and
Corollary~\ref{cor:PossibleBehaviors},
is established in Section~\ref{sec:ChaoticMotions},
whereas Section~\ref{sec:OptimalSpeed} contains the proof of
Theorem~\ref{thm:BoundaryIntro}.
In Section~\ref{sec:PeriodicTrajectories}, we count the periodic orbits,
thus proving Theorem~\ref{thm:PeriodicIntro}.
Finally, Theorem~\ref{thm:AsymptoticConstantIntro} is proved
in Section~\ref{sec:LengthSpectrum}.
Some technical proofs are relegated to appendices.

\section{Circular polygons}\label{sec:CircularPolygons}

In this section we define our relevant curves,
construct their suitable parametrisations,
and introduce notations that will be extensively used in the rest
of the paper.

A \emph{piecewise-circular curve} (or \emph{PC curve} for short) is
given by a finite sequence of circular arcs in
the Euclidean plane $\Rset^2$,
with the endpoint of one arc coinciding with the beginning point
of the next.
PC curves have been studied by several authors.
See~\cite{banchoff1994geometry,BurelleKirk2021} and
the references therein.
Lunes and lemons (two arcs),
yin-yang curves, arbelos and PC cardioids (three arcs),
salinons, Moss's eggs and pseudo-ellipses (four arcs),
and Reuleaux polygons (arbitrary number of arcs)
are celebrated examples of simple closed PC
curves~\cite{Dixon1987,AlsinaNelsen2011}.
A simple closed PC curve is a PC curve not crossing itself such that
the endpoint of its last arc coincides with the beginning point
of its first arc.

All simple closed PC curves are Jordan curves,
so we could study the billiard dynamics in any domain enclosed by
a simple closed PC curve.
However, such domains are too general for our purposes.
We will only deal with strictly convex domains without corners or cusps.
Strict convexity is useful, because then any ordered pair of
points on the boundary defines a unique billiard trajectory.
Absence of cusps and corners implies that the corresponding billiard map
is a global homeomorphism in the phase space $\mathcal{M}$,
see Section~\ref{sec:Fundamental}.

Therefore, we will only consider \emph{circular polygons},
defined as follows.

\begin{definition}
A \emph{circular $k$-gon} is a simple closed strictly convex curve
in $\Rset^2$ formed by the concatenation of $k > 1$ circular arcs,
in such a way that the curve is $C^1$, but not $C^2$,
at the intersection points of any two consecutive circular arcs.
The \emph{nodes} of a circular polygon are the intersection points
of each pair of consecutive circular arcs.
\end{definition}

Reuleaux polygons, lemons, lunes, yin-yang curves, arbelos, salinons,
PC cardioids are not circular polygons,
but pseudo-ellipses and Moss's eggs
(described later on; see also~\cite[Section~1.1]{Dixon1987}) are.
We explicitly ask that consecutive arcs always have different radii,
so the curvature has jump discontinuities at all nodes.
We do not consider circumferences as circular polygons
since circular billiards are completely integrable.

Let $\Gamma$ be a circular $k$-gon with arcs
$\Gamma_1,\ldots,\Gamma_k$,
listed in the order in which they are concatenated,
moving in a counter-clockwise direction.
Each arc $\Gamma_j$ is completely determined by
its \emph{center} $O_j$, its \emph{radius} $r_j > 0$ and
its \emph{angular range} $[a_j,b_j] \subset \Tset$.
Then $\delta_j = b_j - a_j$ is the \emph{central angle}
of $\Gamma_j$.
Using the standard identification $\Rset^2 \simeq \Cset$,
let
\[
A_j = O_j + r_j \rme^{\rmi a_j}, \qquad B_j = O_j + r_j \rme^{\rmi b_j}
\]
be the two nodes of arc $\Gamma_j$.
We denote by
\begin{equation}
\label{eq:SetOfNodes}
\Gamma_\star =
\{ A_1,\ldots,A_k \} = \{ B_1,\ldots,B_k \}
\end{equation}
the \emph{set of nodes} of $\Gamma$.

\begin{notation}
The index $j$ that labels the arcs of any circular $k$-gon is
defined modulo $k$.
Hence, $\Gamma_j = \Gamma_{j \mod k}$,
$r_j = r_{j \mod k}$, $a_j = a_{j \mod k}$ and so forth.
In particular, $\Gamma_{k+1} = \Gamma_1$.
\end{notation}

\begin{definition}
\label{def:PolarParametrization}
The \emph{polar parametrisation} of $\Gamma$
is the counter-clockwise parametrisation
\[
z : \Tset \to \Gamma \subset \Rset^2 \simeq \Cset,\qquad
z(\varphi) = O_j + r_j \rme^{\rmi\varphi}, \qquad
\forall \varphi \in [a_j,b_j].
\]
The points $a_1,\ldots,a_k$ are the \emph{singularities} of $\Gamma$.
\end{definition}

This parametrisation is well-defined because, by definition,
$B_j = A_{j+1}$ (the endpoint of any arc coincides with
the beginning point of the next), and $b_j = a_{j+1}$
(two consecutive arcs have the same oriented tangent line at their
 intersecting node).
From now on,
the reader should keep in mind that singularities
$a_1,\ldots,a_k$ are always ordered in such a way that
\begin{equation}
\label{eq:AngularRanges}
a_1 < b_1 = a_2 < b_2 = a_3 < \cdots <
b_{k-1} = a_k < b_k = a_1 + 2\pi.
\end{equation}

As far as we know,
all the billiards in circular polygons that have been studied
in the past correspond to cases with \emph{exactly} four
arcs~\cite{balint2011chaos, benettin1978numerical, dullin1996two,
 hayli1986experiences, henon1983benettin, makino2019bifurcation}.
It turns out that this is the simplest case,
in the context of the next lemma.

\begin{lemma}
\label{lemma_geq4arcs}
Let $\Gamma$ be a circular $k$-gon with radii $r_j > 0$,
singularities $a_j \in \Tset$ (or $b_j = a_{j+1}$)
and central angles $\delta_j = b_j - a_j \in (0,2\pi)$.
Set $w_j = \rme^{\rmi b_j} - \rme^{\rmi a_j} \in \Cset$.
Then $\Gamma$ has at least four arcs: $k \ge 4$, and
\begin{equation}
\label{eq:NecessarySufficient}
\sum_{j=1}^k \delta_j = 2\pi, \qquad
\sum_{j=1}^k r_j w_j = 0.
\end{equation}
\end{lemma}

\begin{proof}
Clearly,
$\sum_{j=1}^k \delta_j = \sum_{j=1}^k (b_j-a_j) = b_k - a_1 = 2\pi$.
It is known that a bounded measurable function
$\rho \in L(\Tset)$ is the radius of curvature of a closed curve
if and only if
\begin{equation}
\label{eq:ClosedCondition}
\int_{a_1}^{b_k} \rho(\varphi) \rme^{\rmi \varphi} \rmd \varphi =
\int_0^{2\pi} \rho(\varphi) \rme^{\rmi \varphi} \rmd \varphi = 0.
\end{equation}
Since the radius of curvature of $\Gamma$
is the piecewise constant function
$\rho_{|(a_j,b_j)} \equiv r_j$,
the general condition~(\ref{eq:ClosedCondition}) becomes
$-\rmi \sum_{j=1}^k r_j w_j = 0$.
Note that $\sum_{j=1}^k w_j = 0$, $w_j \neq 0$ for all $j$,
and $\dim_\Rset [w_1,\ldots,w_k] = 2$ when $k \ge 3$.
If $\Gamma$ has just two arcs: $k = 2$, then
\[
r_1 w_1 + r_2 w_2 = 0, \qquad w_1 + w_2 = 0, \qquad w_1,w_2 \neq 0.
\]
This implies that $r_1 = r_2$ and
contradicts our assumption about radii of consecutive arcs.
If $\Gamma$ has just three arcs: $k = 3$, then
\[
r_1 w_1 + r_2 w_2 + r_3 w_3 = 0,\qquad
w_1 + w_2 + w_3 = 0,\qquad \dim_\Rset [w_1,w_2,w_3] = 2.
\]
This implies that $r_1 = r_2 = r_3$ and we reach the same
contradiction.
\end{proof}

Necessary conditions~\eqref{eq:AngularRanges}
and~\eqref{eq:NecessarySufficient} are sufficient ones too.
To be precise, if the radii $r_j > 0$,
the angular ranges $[a_j,b_j] \subset \Tset$ and
the central angles $\delta_j = b_j - a_j \in (0,2\pi)$
satisfy~\eqref{eq:AngularRanges} and~\eqref{eq:NecessarySufficient},
then there exists a $2$-parameter family of circular $k$-gons
sharing all those elements.
Moreover,
all circular $k$-gons in this family are the same modulo translations.
Let us prove this claim.
Once we put the center $O_1$ at an arbitrary location,
all other centers are recursively determined by imposing
that $A_{j+1} = B_j$, which implies (since $b_j = a_{j+1}$) that
\[
O_{j+1} = O_j + (r_j - r_{j+1})\rme^{\rmi b_j}, \qquad j=1,\ldots,k-1.
\]
The obtained PC curve $\Gamma = z(\Tset)$,
where $z(\varphi)$ is the polar parametrisation introduced in
Definition~\ref{def:PolarParametrization},
is closed by~\eqref{eq:ClosedCondition} and
it is $C^1$ and strictly convex by construction.
Hence, $\Gamma$ is a circular $k$-gon.
This means that any circular $k$-gon is completely determined
once we know its first center $O_1$, its first singularity $a_1$,
its radii $r_1,\ldots,r_k$,
and its central angles $\delta_1,\ldots,\delta_k$.

The above discussion shows that circular $k$-gons form,
modulo translations and rotations, a $(2k-3)$-parameter family.
More precisely, if we set $O_1 = (0,0)$ and $a_1 = 0$
by means of a translation and a rotation,
then parameters $r_1,\ldots,r_k,\delta_1,\ldots,\delta_k$
are restricted by~\eqref{eq:NecessarySufficient},
which has codimension three.
If, in addition,
we normalise somehow (in the literature one can find many
different choices) our circular $k$-gons with a scaling,
we get that they form a $(2k-4)$-parameter family
modulo similarities.
The reader can find a complete geometric description,
modulo similarities, of the four-parameter family of
(convex and nonconvex, symmetric and nonsymmetric)
closed $C^1$ PC curves with four arcs in~\cite{hayli1986experiences},
whose goal was to exhibit numerically the richness of the billiard
dynamics in those $C^1$ PC curves.

For brevity, we only give a few simple examples of
symmetric and non-symmetric circular polygons with four and six arcs.
We skip many details.

\begin{figure}[t]
\includegraphics*[width=78mm]{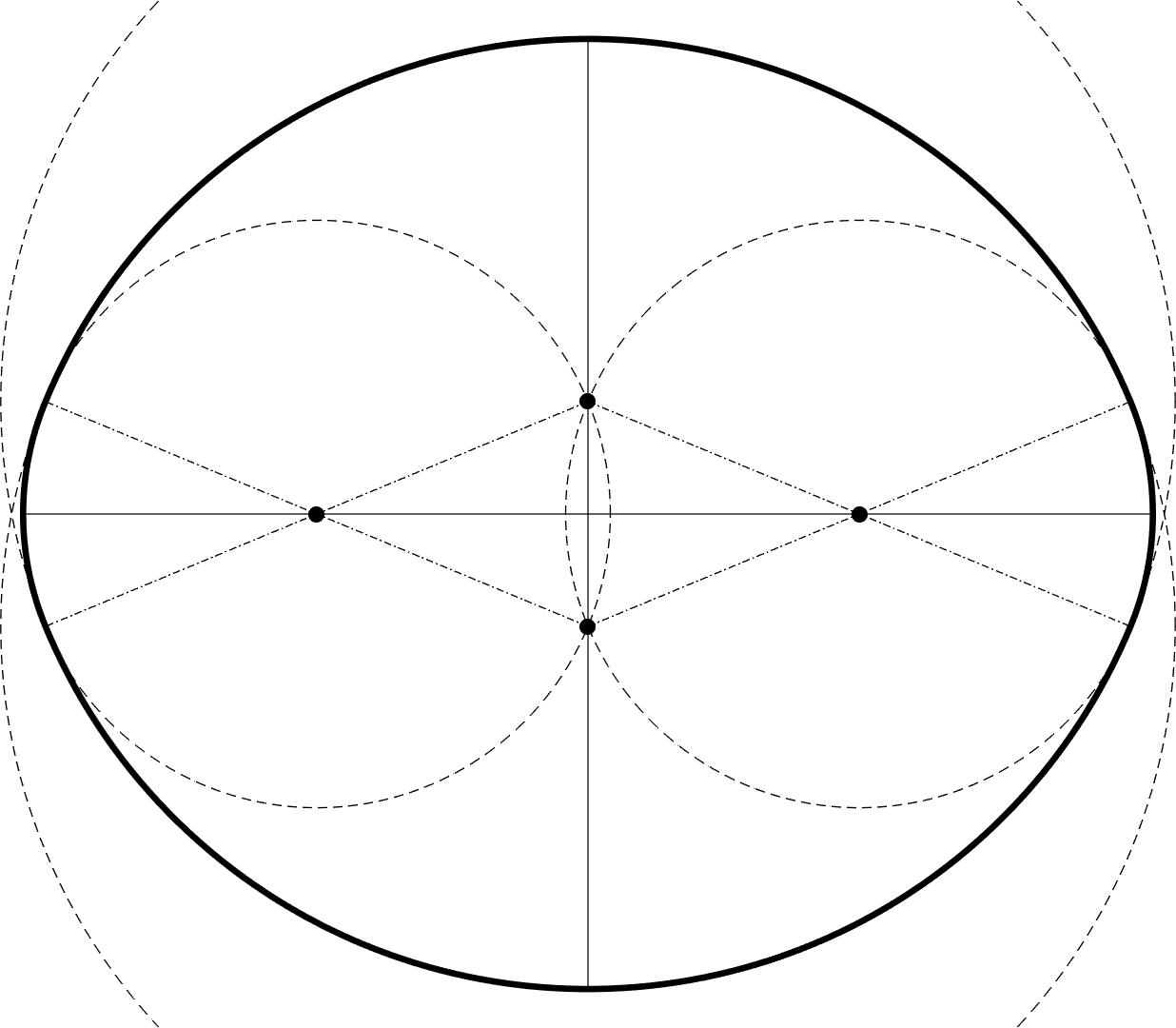}\hfill
\includegraphics*[width=78mm]{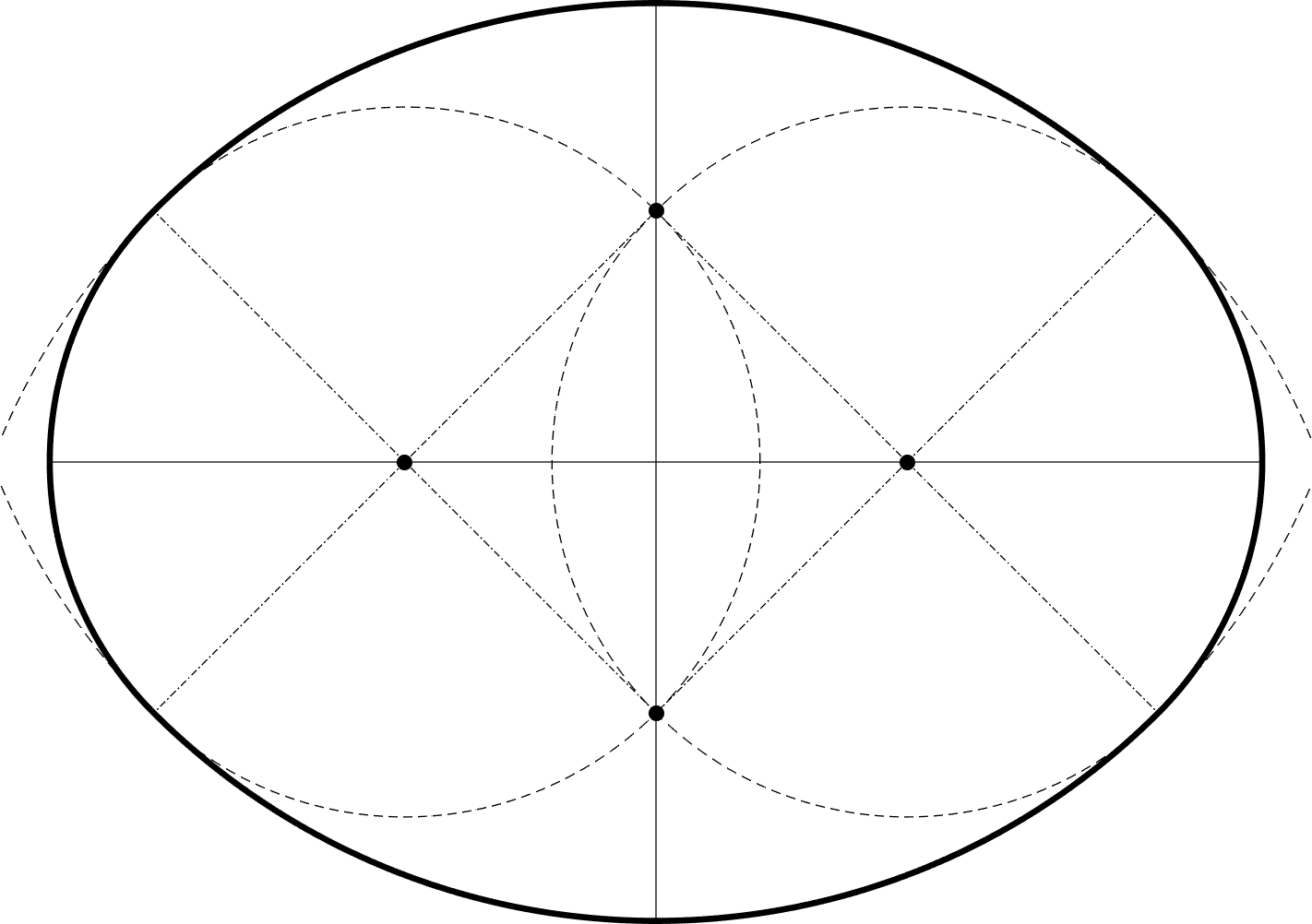}
\caption{\small Left: Pseudo-ellipse $E_{\pi/4,1,2}$.
Right: Squared pseudo-ellipse $E_{\pi/2,1,2}$.
Pseudo-ellipses are represented with thick lines,
their pairs of symmetry lines with thin continuous lines,
their centers $O_j$ with solid dots,
the circumferences of radii $r_j$ centered at $O_j$ with dashed thin lines,
their angular ranges $[a_j,b_j]$ with dash-dotted thin lines, and
their nodes are the intersections of the thick and dash-dotted thin
lines.}
\label{fig:PseudoEllipses}
\end{figure}

\emph{Pseudo-ellipses} are the simplest examples.
We may define them as the circular $4$-gons with a
$\Zset_2 \times \Zset_2$-symmetry.
They form, modulo translations and rotations, a three-parameter family.
The radii and central angles of any pseudo ellipse have the form
\[
r_1 = r_3 = r,\qquad
r_2 = r_4 = R,\qquad
\delta_1 = \delta_3 = \alpha,\qquad
\delta_2 = \delta_4 = \pi - \alpha,
\]
for some free parameters $\alpha \in (0,\pi)$, and $r,R > 0$.
We will assume that $0 < r < R$ for convenience.
We will denote by $E_{\alpha,r,R}$ the corresponding
pseudo-ellipse.
Given any pseudo-ellipse $E_{\alpha,r,R}$,
its centers form a \emph{rhombus} (4~equal sides) and
its nodes form a \emph{rectangle} (4~equal angles).
If $\alpha = \pi/2$, then $\delta_1=\delta_2=\delta_3=\delta_4=\pi/2$
and we say that $E_{\pi/2,r,R}$ is
a \emph{squared pseudo-ellipse}.
The term \emph{squared} comes from the fact that the centers of such
pseudo-ellipses form a square.
See Figure~\ref{fig:PseudoEllipses}.
The nodes of a squared pseudo-ellipse still form a rectangle, not a square.
On the contrary, the celebrated Benettin-Strelcyn ovals,
whose billiard dynamics was numerically studied
in~\cite{benettin1978numerical,henon1983benettin,makino2019bifurcation},
are pseudo-ellipses whose nodes form a square,
but whose centers only form a rhombus.
Later on,
the extent of chaos in billiards associated to general pseudo-ellipses
was numerically studied in~\cite{dullin1996two}.

Another celebrated example of a circular $4$-gon
is \emph{Moss's egg}~\cite[Section 1.1]{Dixon1987},
whose radii and central angles have the form
\[
r_1 = r, \quad
r_2 = 2r = r_4,\quad
r_3 = (2-\sqrt{2})r,\quad
\delta_1 = \pi,\quad
\delta_2 = \pi/4 = \delta_4,\quad
\delta_3 = \pi/2,
\]
for some free parameter $r>0$, called the radius of the egg.
All Moss's eggs are congruent modulo similarities.
They have a $\Zset_2$-symmetry, so their nodes form an
\emph{isosceles trapezoid} (2 pairs of consecutive equal angles) and
its centers form a \emph{kite} (2 pairs of adjacent equal-length sides).
In fact, this kite is somewhat degenerate since it is, in fact, a triangle.
See Figure~\ref{fig:OtherExamples}.
Billiards in a 2-parameter family of circular $4$-gons with
$\Zset_2$-symmetry, but not containing Moss's egg,
were considered in~\cite{balint2011chaos}.
The heuristic analysis of sliding trajectories contained
in Section~4.5 of that paper is closely related to our study.

\begin{figure}[t]
\includegraphics*[width=65mm]{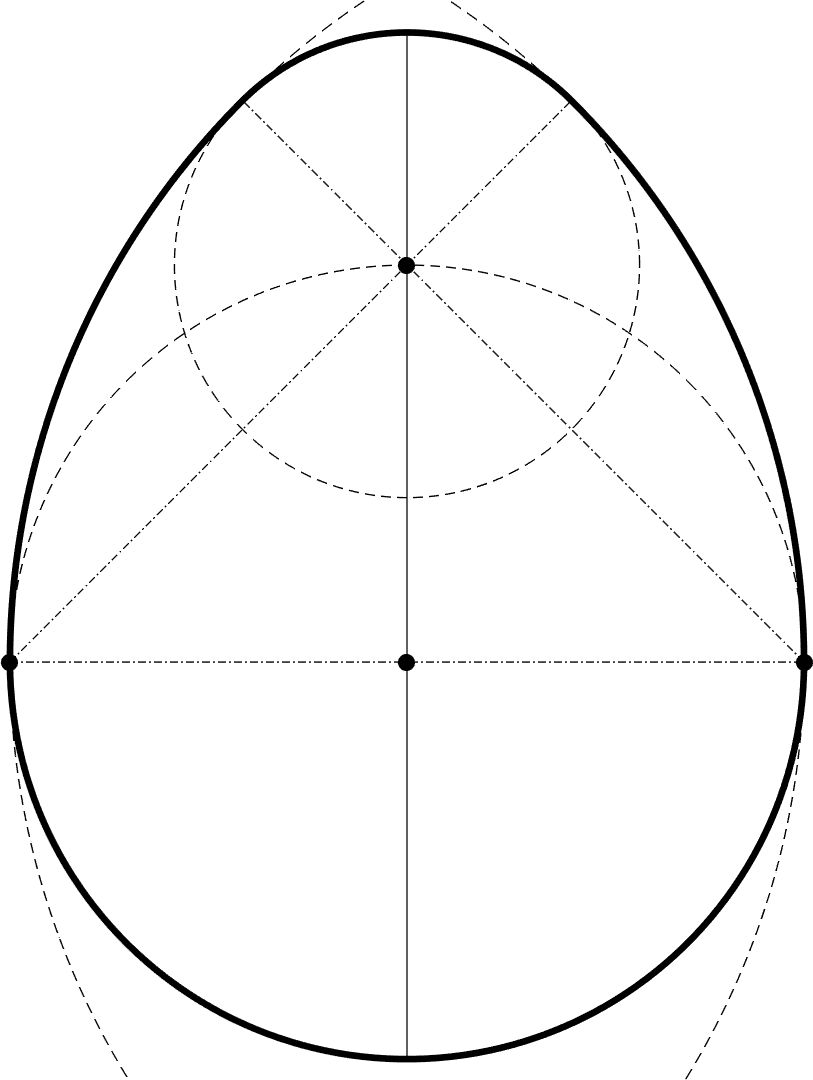}\hfill
\includegraphics*[width=85mm]{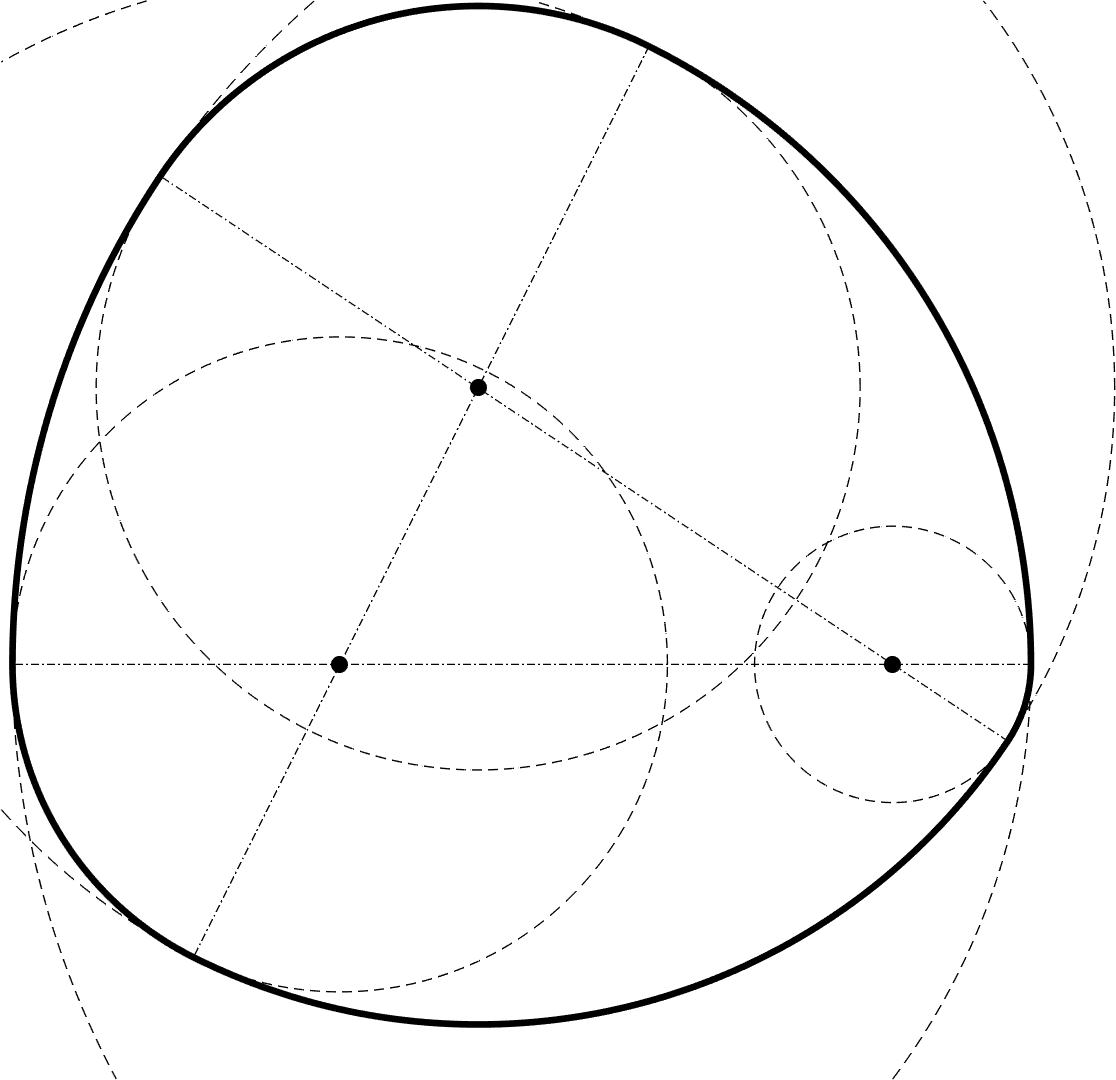}
\caption{\small Left: Moss's egg.
Right: A nonsymmetric circular $6$-gon with centers
$O_1 = O_4 = (3,-1)$, $O_2 = O_5 = (-1,-1)$ and $O_3 = O_6 = (0,1)$,
which form a triangle.
Circular polygons are represented with thick lines,
the symmetry line of Moss's egg with a thin continuous line,
their centers $O_j$ with solid dots,
the circumferences of radii $r_j$ centered at $O_j$
with dashed thin lines,
their angular ranges $[a_j,b_j]$ with dash-dotted thin lines,
and their nodes are the intersections of the thick and
dash-dotted thin lines.}
\label{fig:OtherExamples}
\end{figure}

Next, we describe a way to construct some circular $6$-gons.
Fix a triangle $\triangle ABC$ with vertexes $A$, $B$ and $C$
ordered in the \emph{clockwise} direction.
Let $\alpha$, $\beta$ and $\gamma$ be its internal angles.
Let $a$, $b$ and $c$ be the lengths of its sides,
following the standard convention.
That is, $a$ refers to the side opposed to vertex $A$ and so forth.
Then we look for circular $6$-gons with centers
$O_1 = O_4 = A$, $O_2 = O_5 = B$, $O_3 = O_6 = C$ and central angles
$\delta_1 = \delta_4 = \alpha$, $\delta_2 = \delta_5 = \beta$ and
$\delta_3 = \delta_6 = \gamma$.
In this setting, all radii are determined by the choice of the first one.
Namely, we can take
\[
r_1 = r,\quad
r_2 = r + c,\quad
r_3 = r + c - a,\quad
r_4 = r + c - a + b,\quad
r_5 = r + b - a,\quad
r_6 = r + b,
\]
for any $r > \max\{0,a-c,a-b\}$.
Therefore, we obtain a one-parameter family of parallel
circular $6$-gons, parameterised by the first radius $r_1 = r$.
See Figure~\ref{fig:OtherExamples} for a non-symmetric example
with $A = (3,-1)$, $B=(-1,-1)$, $C=(0,1)$ and $r=1$.

One can draw circular polygons with many arcs by applying
similar constructions,
but that challenge is beyond the scope of this paper.
The interested reader can look for inspiration in the nice
construction of elliptic
flowers due to Bunimovich~\cite{Bunimovich2022}.

To end this section,
we emphasise that all our theorems are general.
They can be applied to any circular polygon.
Thus, we do not need to deal with concrete circular polygons.

\section{The `stretching along the paths' method}
\label{sec:Stretching}

In this short section,
we present the main ideas of the
\emph{stretching along the paths} method developed by
Papini and Zanolin~\cite{papini2004fixed,papini2004periodic},
and extended by Pireddu and
Zanolin~\cite{pireddu2007cutting,pireddu2008chaotic,pireddu2009fixed}.
The reader acquainted with the method can take note of
the notation introduced in Definition~\ref{def:Streching}
and skip the rest of this section.

This method is a technical tool to establish the existence of \emph{topological chaos};
that is, chaotic dynamics in continuous maps.
We present a simplified version of the method because we work
in the two-dimensional annulus $\mathcal{M} = \Tset \times [0,\pi]$
and our maps are \emph{homeomorphisms} on $\mathcal{M}$.
We also change some terminology because our maps stretch
along \emph{vertical} paths, instead of \emph{horizontal} paths.
The reader interested in more general statements about
higher dimensions,
finding fixed and periodic points in smaller compact sets,
the study of crossing numbers, non-invertible maps,
and maps not defined in the whole space $\mathcal{M}$,
is referred to the original references.

Let $\mathcal{M} = \Tset \times [0,\pi]$.
By a \emph{continuum} we mean a compact connected subset
of $\mathcal{M}$.
\emph{Paths} and \emph{arcs} are the continuous and the homeomorphic
images of the unit interval $[0,1]$, respectively.
Most definitions below are expressed in terms of paths,
but we could also use arcs or continua,
see~\cite[Table 3.1]{papini2004fixed}.
\emph{Cells} are the homeomorphic image of the unit square $[0,1]^2$
so they are simply connected and compact.
The Jordan-Shoenflies theorem implies that any simply connected compact
subset of $\mathcal{M}$ bounded by a Jordan curve is a cell.
Figure~\ref{fig:Cells} provides a visual guide for the
following definitions.

\begin{definition}\label{def:OrientedCell}
An \emph{oriented cell} $\widetilde{\mathcal{Q}}$
is a cell $\mathcal{Q} \subset \mathcal{M}$
where we have chosen four different points
$\widetilde{\mathcal{Q}}_{\rm bl}$ (base-left),
$\widetilde{\mathcal{Q}}_{\rm br}$ (base-right),
$\widetilde{\mathcal{Q}}_{\rm tr}$ (top-right) and
$\widetilde{\mathcal{Q}}_{\rm tl}$ (top-left) over the boundary
$\partial \mathcal{Q}$ in a counter-clockwise order.
The \emph{base side} of $\widetilde{\mathcal{Q}}$ is the arc
$\widetilde{\mathcal{Q}}_{\rm b} \subset \partial \mathcal{Q}$ that
goes from $\widetilde{\mathcal{Q}}_{\rm bl}$
to $\widetilde{\mathcal{Q}}_{\rm br}$ in the counter-clockwise
orientation.
Similarly, $\widetilde{\mathcal{Q}}_{\rm l}$,
$\widetilde{\mathcal{Q}}_{\rm r}$ and $\widetilde{\mathcal{Q}}_{\rm t}$
are the \emph{left}, \emph{right} and \emph{top sides}
of $\widetilde{\mathcal{Q}}$.
Finally,
$\widetilde{\mathcal{Q}}_{\rm h} =
 \widetilde{\mathcal{Q}}_{\rm b} \cup \widetilde{\mathcal{Q}}_{\rm t}$ and
$\widetilde{\mathcal{Q}}_{\rm v} =
 \widetilde{\mathcal{Q}}_{\rm l} \cup \widetilde{\mathcal{Q}}_{\rm r}$
are the \emph{horizontal} and \emph{vertical sides} of
$\widetilde{\mathcal{Q}}$.
\end{definition}

All our cells will have line segments as vertical sides,
some being even quadrilaterals.

\begin{figure}[t]
\includegraphics*[width=80mm]{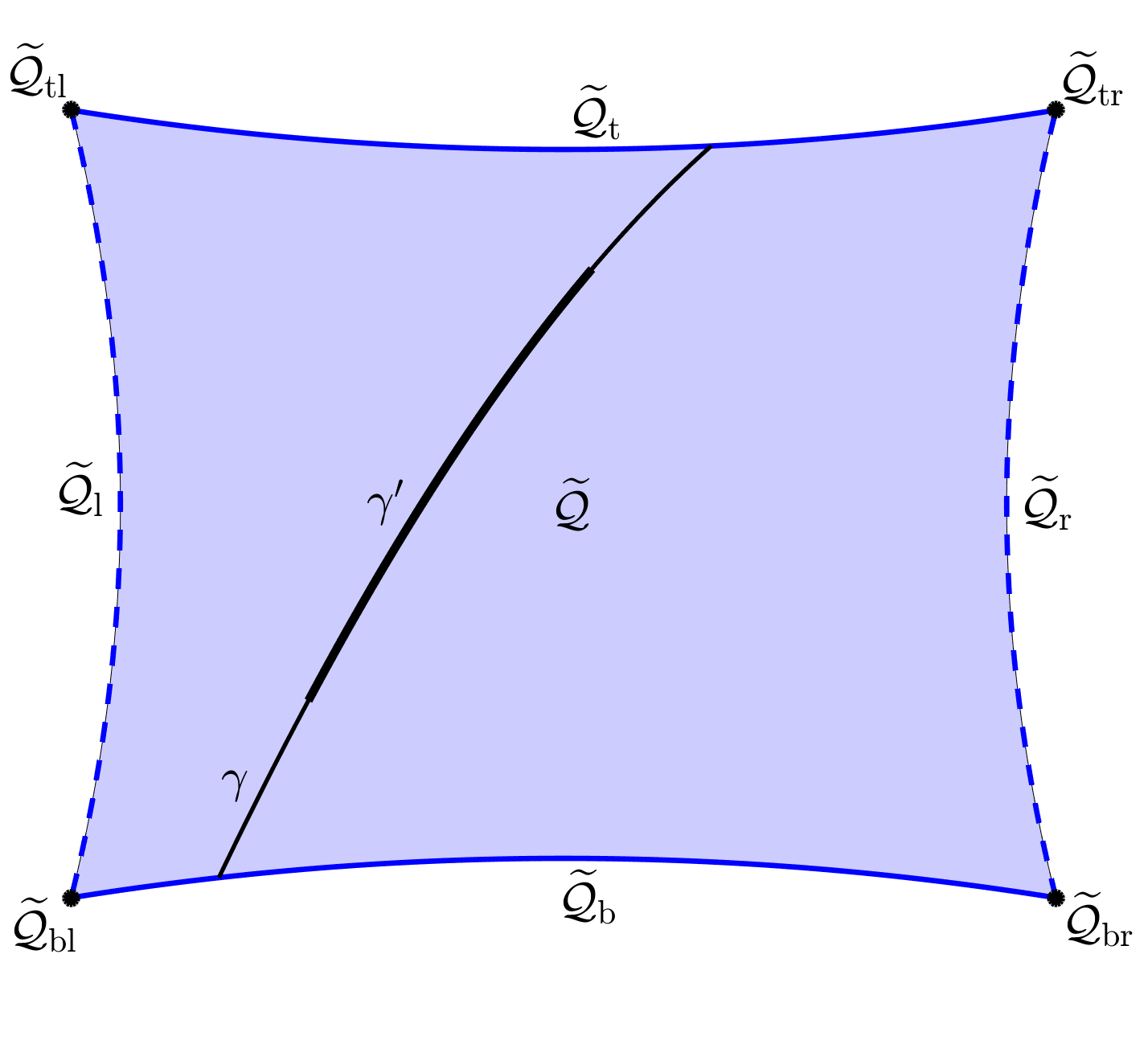}\hfill
\includegraphics*[width=70mm]{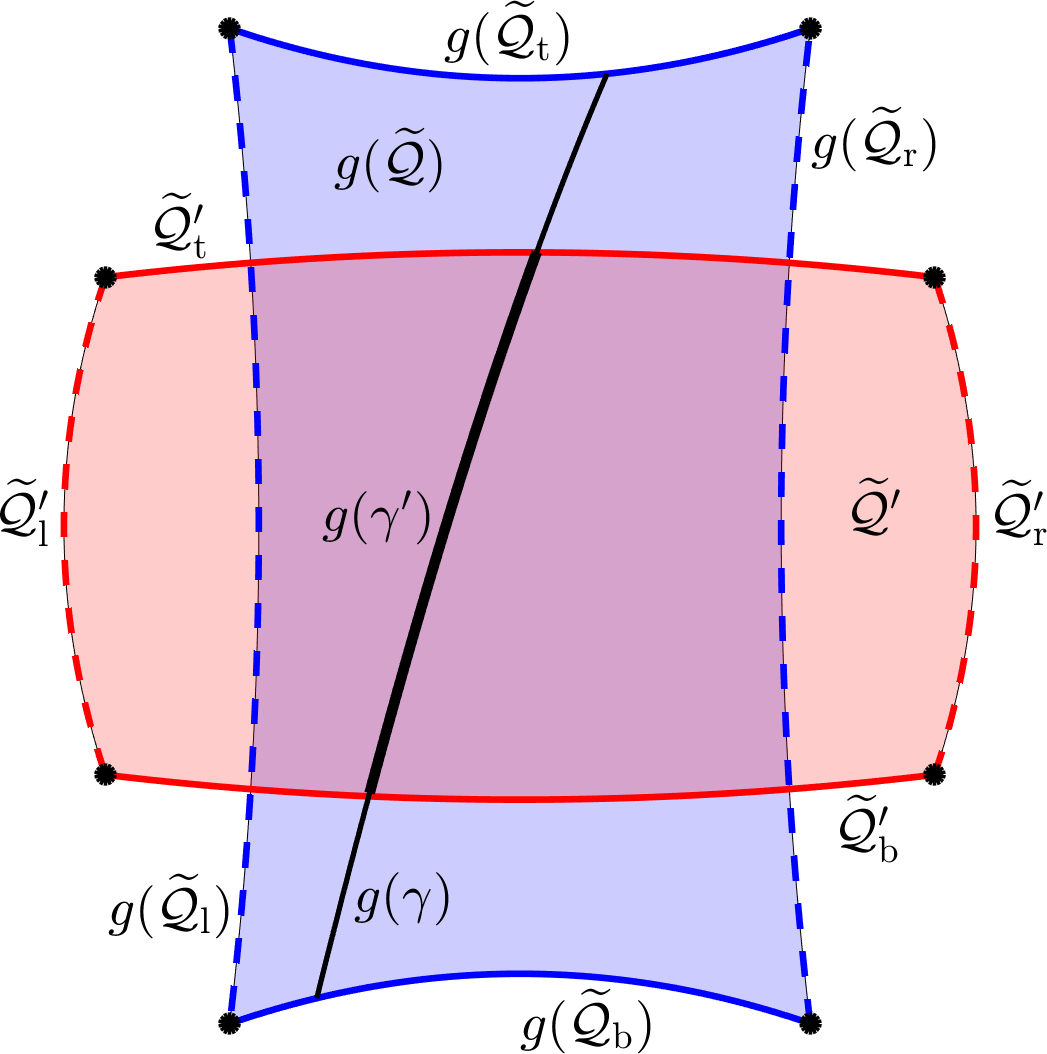}
\caption{\small
Left: An oriented cell $\widetilde{\mathcal{Q}}$ (in blue)
      with a vertical path $\gamma$ (in black).
Right: A homeomorphism $g$ stretches $\widetilde{\mathcal{Q}}$
       to a second oriented cell $\widetilde{\mathcal{Q}}'$ (in red)
       along vertical paths.
The horizontal and vertical sides of both cells are
represented with continuous and dashed lines, respectively.
The vertical path $\gamma$ and its subpath $\gamma'$
(and their corresponding images in the picture on the right)
are represented with thin and thick lines, respectively.}
\label{fig:Cells}
\end{figure}

\begin{definition}\label{def:VerticalPath}
Let $\widetilde{\mathcal{Q}}$ be an oriented cell.
A path $\gamma:[a,b] \to Q$ is \emph{vertical}
(respectively, \emph{horizontal}) in $\widetilde{\mathcal{Q}}$
when it connects the two horizontal (respectively, vertical) sides
of $\widetilde{\mathcal{Q}}$ and
$\gamma(t) \not \in \widetilde{\mathcal{Q}}_{\rm h}$
(respectively, $\gamma(t) \not \in \widetilde{\mathcal{Q}}_{\rm v}$)
for all $t \in (a,b)$.
We say that an oriented cell $\widetilde{\mathcal{K}}$
is a \emph{horizontal slab} in $\widetilde{\mathcal{Q}}$ and write
\[
\widetilde{\mathcal{K}} \subset_{\rm h} \widetilde{\mathcal{Q}}
\]
when $\mathcal{K} \subset \mathcal{Q}$ and, either
$\widetilde{\mathcal{K}}_{\rm l} \subset \widetilde{\mathcal{Q}}_{\rm l}$
and
$\widetilde{\mathcal{K}}_{\rm r} \subset \widetilde{\mathcal{Q}}_{\rm r}$,
or
$\widetilde{\mathcal{K}}_{\rm l} \subset \widetilde{\mathcal{Q}}_{\rm r}$
and
$\widetilde{\mathcal{K}}_{\rm r} \subset \widetilde{\mathcal{Q}}_{\rm l}$.
If, in addition,
$\mathcal{K} \cap \widetilde{\mathcal{Q}}_{\rm h} = \emptyset$,
then we say that $\tilde{\mathcal{K}}$ is a
\emph{strict horizontal slab} in $\widetilde{\mathcal{Q}}$ and write
\[
\widetilde{\mathcal{K}} \varsubsetneq_{\rm h} \widetilde{\mathcal{Q}}.
\]
\end{definition}

\emph{Vertical slabs} can be defined analogously.
Note that
$\widetilde{\mathcal{K}} \varsubsetneq_{\rm h} \widetilde{\mathcal{Q}}$
is a much stronger condition than
$\widetilde{\mathcal{K}} \subset_{\rm h} \widetilde{\mathcal{Q}}$ and
$\mathcal{K} \varsubsetneq \mathcal{Q}$.

\begin{definition}\label{def:Streching}
Let $g: \mathcal{M} \to \mathcal{M}$ be a homeomorphism.
Let $\widetilde{\mathcal{Q}}$ and $\widetilde{\mathcal{Q}}'$
be oriented cells in~$\mathcal{M}$.
We say that $g$ \emph{stretches $\widetilde{\mathcal{Q}}$ to
$\widetilde{\mathcal{Q}}'$ along vertical paths} and write
\[
g: \widetilde{\mathcal{Q}} \stretches \widetilde{\mathcal{Q}}'
\]
when every path $\gamma:[a,b] \to \mathcal{Q}$
that is vertical in $\widetilde{\mathcal{Q}}$
contains a \emph{subpath} $\gamma' = \gamma_{|[s,t]}$
for some $a \le s < t \le b$ such that the image path
$g \circ \gamma':[s,t] \to \mathcal{Q}'$
is vertical in $\widetilde{\mathcal{Q}}'$.
\end{definition}

This stretching condition does not imply that
$g(\mathcal{Q}) \subset \mathcal{Q}'$.
In fact, we see $\mathcal{Q}'$ as a `target set' that we want to `visit',
and not as a codomain.
If $\gamma : [a,b] \to \mathcal{M}$ is a path,
we also use the notation $\gamma$ to mean the set
$\gamma([a,b]) \subset \mathcal{M}$.
This allows us to state the stretching condition more succinctly.
Namely, we ask that every path $\gamma$ vertical in
$\widetilde{\mathcal{Q}}$ contains a subpath $\gamma' \subset \gamma$
such that the image path $g(\gamma')$ is vertical
in $\widetilde{\mathcal{Q}}'$.

\begin{definition}
Let $f:\mathcal{M} \to \mathcal{M}$ be a homeomorphism.
Let $(\mathcal{Q}_i;n_i)_{i \in I}$ be a two-sided sequence: $I \in \Zset$,
one-sided sequence: $I \in \Nset_0 = \Nset \cup \{0\}$,
$p$-periodic sequence: $I=\Zset/p\Zset$,
or finite sequence $I = \{0,1,\ldots,k\}$
with $\mathcal{Q}_i \subset \mathcal{M}$ and $n_i \in \Nset$.
Let $x \in \mathcal{Q}_0$.
We say that the point $x$ \emph{$f$-realises} the sequence
$(\mathcal{Q}_i;n_i)_{i \in I}$ when
\[
f^{-(n_{-1} + \cdots + n_{-i})}(x) \in \mathcal{Q}_{-i},\qquad
f^{n_0 + \cdots + n_{i-1}}(x) \in \mathcal{Q}_i, \qquad
\forall i \ge 1.
\]
Clearly,
condition $f^{-(n_{-1} + \cdots + n_{-i})}(x) \in \mathcal{Q}_{-i}$
does not apply in the case of one-sided or finite sequences.
A subset of $\mathcal{Q}_0$ \emph{$f$-realises}
the sequence $(\mathcal{Q}_i;n_i)_{i \in I}$ when all its points do so.
\end{definition}

The subsets $\mathcal{Q}_i$ in this definition do not have to be cells,
but that is the case considered in the following powerful 3-in-1 theorem
about the existence of points and paths of the phase space $\mathcal{M}$
that $f$-realise certain two-sided, one-sided, and periodic sequences
of oriented cells.

\begin{theorem}[Papini \& Zanolin~\cite{papini2004periodic}]
\label{thm:PapiniZanolin}
Let $f:\mathcal{M} \to \mathcal{M}$ be a homeomorphism.
Let $(\widetilde{\mathcal{Q}}_i;n_i)_{i \in I}$
be a two-sided, one-sided or $p$-periodic sequence
where $\widetilde{\mathcal{Q}}_i$ are oriented cells with
$\mathcal{Q}_i \subset \mathcal{M}$ and $n_i \in \Nset$.
If
\[
f^{n_i}:
\widetilde{\mathcal{Q}}_i \stretches \widetilde{\mathcal{Q}}_{i+1},
\qquad \forall i,
\]
then the following statements hold.
\begin{itemize}
\item[{\rm ({\bf T})}]
If $I = \Zset$,
there is a point $x \in \mathcal{Q}_0$ that $f$-realises the
two-sided sequence $(\mathcal{Q}_i;n_i)_{i \in \Zset}$.
\item[{\rm ({\bf O})}]
If $I = \Nset_0$,
there is a path $\gamma$ horizontal in $\widetilde{\mathcal{Q}}_0$
that $f$-realises the sequence $(\mathcal{Q}_i;n_i)_{i \ge 0}$.
\item[{\rm ({\bf P})}]
If $(\widetilde{\mathcal{Q}}_{i+p};n_{i+p}) =
    (\widetilde{\mathcal{Q}}_i;n_i)$ for all $i \in \Zset$
and $n = n_0 + \cdots + n_{p-1}$,
there is a point $x \in \mathcal{Q}_0$ such that
$f^n(x) = x$ and $x$ $f$-realises the $p$-periodic sequence
$(\mathcal{Q}_i;n_i)_{i \in \Zset/p\Zset}$.
\end{itemize}
\end{theorem}

\begin{remark}
We believe that the following finite version~({\bf F}) also holds:
``If $I = \{ 0,\ldots,k \}$,
  there is a horizontal slab
  $\widetilde{\mathcal{K}} \subset_{\rm h} \widetilde{\mathcal{Q}}_0$
  such that $\mathcal{K}$ $f$-realises the finite sequence
  $(\mathcal{Q}_i;n_i)_{i=0,\ldots,k}$.'',
but we have not found such statement in the literature.
Therefore, we will not use it.
\end{remark}

We refer to Theorem~2.2 in~\cite{papini2004periodic} for
a more general statement which deals with sequences of
maps that are not some power iterates of a single map.
Version~({\bf T}) of Theorem~\ref{thm:PapiniZanolin} is the key tool
to obtain orbits that follow prescribed itineraries,
so that we can establish the existence of topological chaos and
we can construct a suitable symbolic dynamics in
Section~\ref{sec:ChaoticMotions}.
We will use version~({\bf O}) of Theorem~\ref{thm:PapiniZanolin}
to prove the existence of `paths' of generic sliding billiard trajectories
that approach the boundary asymptotically with optimal uniform speed
in Section~\ref{sec:OptimalSpeed}.
Finally, we will establish several lower bounds on the number of
periodic billiard trajectories from version~({\bf P}) of
Theorem~\ref{thm:PapiniZanolin} in Section~\ref{sec:PeriodicTrajectories}.

\section{The fundamental lemma for circular polygons}
\label{sec:Fundamental}

In this section, we define the billiard map,
and describe in Lemma~\ref{lem:BilliardProperties}
the sliding dynamics in circular polygons.
That is,
the dynamics when the angle of reflection $\theta$ is small.
See Definition~\ref{def:GenericSliding} for the precise
formulation.
We then introduce \emph{fundamental quadrilaterals},
which will later serve as symbol sets for symbolic dynamics.
In Lemma~\ref{lem:MinMax} we compute the extreme points of the
fundamental quadrilaterals,
as well as those of their iterates after crossing
the singularities between two consecutive circular arcs.
With these estimates on hand,
we finally state and prove the \emph{fundamental lemma}
(Lemma~\ref{lem:Fundamental}, as well as an important consequence,
 Corollary~\ref{cor:Fundamental}), which describes how orbits
 of the fundamental quadrilaterals visit other
 fundamental quadrilaterals.

We begin with the definition of the billiard map.
Recall that the phase space of the billiard map is
$\mathcal{M} = \Tset \times [0, \pi]$,
and let $(\varphi, \theta) \in \mathrm{Int} \,\mathcal{M}$.
Write $z = z (\varphi)$,
where $z$ is the polar parametrisation of
the circular polygon $\Gamma$ introduced in
Definition~\ref{def:PolarParametrization},
and $v = R_\theta z' (\varphi)$,
where $R_\theta$ is the standard $2 \times 2$
counter-clockwise rotation matrix by an angle $\theta$.
The straight line $L=L(\varphi,\theta)$ passing through $z$
in the direction $v$ has exactly two points of intersection
with $\Gamma$ since $\theta \in (0, \pi)$.
One of these is $z$; denote by $\bar{z}$ the other.
Then there is a unique $\bar{\varphi} \in \Tset$ such that
$\bar{z} = z (\bar{\varphi})$.
Denote by $\bar{\theta}$ the angle between $L$ and
$z'(\bar{\varphi})$ in the counter-clockwise direction.
The \emph{billiard map}
$f: \mathrm{Int} \,\mathcal{M} \to \mathrm{Int} \,\mathcal{M}$
is defined by $f(\varphi,\theta) = (\bar{\varphi}, \bar{\theta})$,
see Figure~\ref{fig:BilliardMap}.

\begin{figure}[t]
\begin{center}
\includegraphics*[width=85mm]{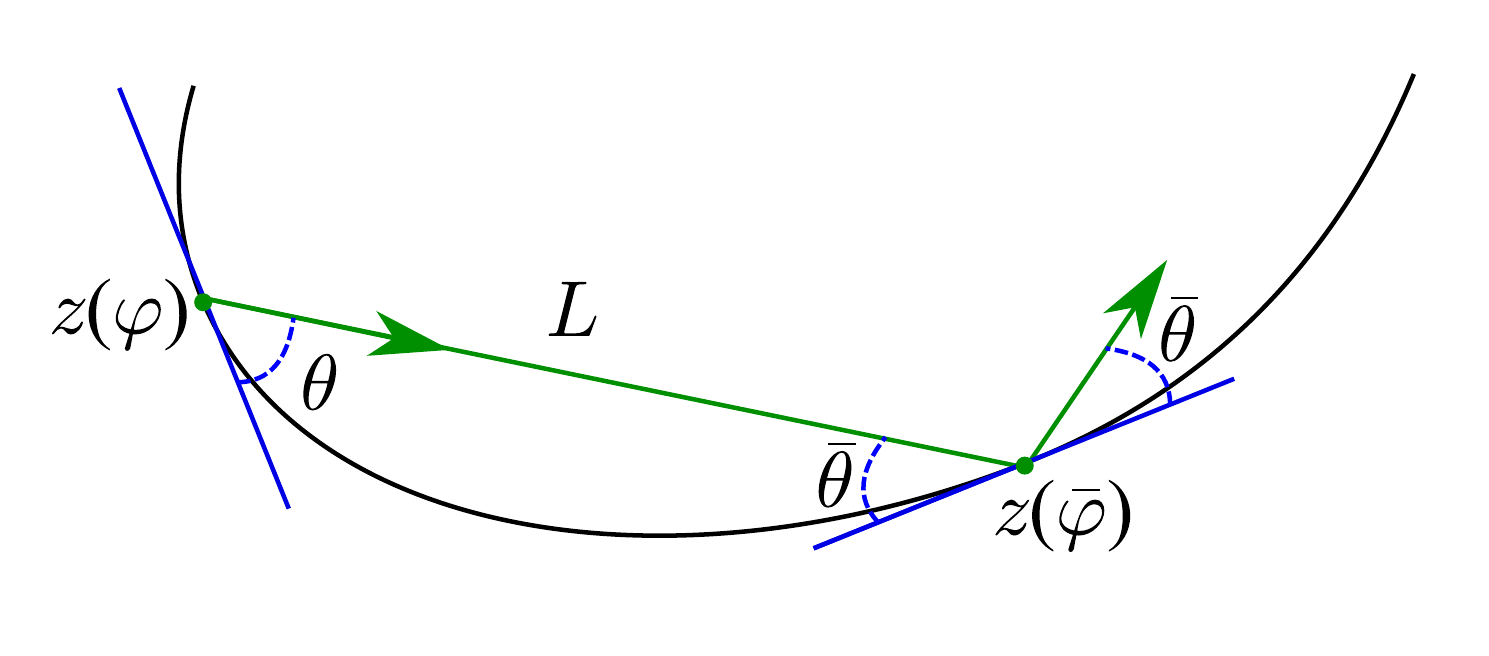}
\end{center}
\caption{\small Definition of the billiard map
$f(\varphi,\theta) = (\bar{\varphi}, \bar{\theta})$.}
\label{fig:BilliardMap}
\end{figure}

Note that $f$ is continuous since $\Gamma$ is $C^1$
and strictly convex.
The billiard map can be extended continuously to $\partial \mathcal{M}$
by setting $f(\varphi,0)=(\varphi,0)$ and
$f(\varphi,\pi) = (\varphi,\pi)$ for each $\varphi \in \Tset$.
The billiard map $f: \mathcal{M} \to \mathcal{M}$ is a homeomorphism;
indeed, the map $f^{-1} = I \circ f \circ I$ is a continuous inverse
where the involution $I : \mathcal{M} \to \mathcal{M}$ is defined by
$I(\varphi,\theta) = (\varphi,\pi - \theta)$.

A key geometric property of the billiard dynamics in the case
of impacts in consecutive arcs was presented
in~\cite{hubacher1987instability}.
Later on, a more detailed description was given in~\cite{balint2011chaos}.
Both results follow from trigonometric arguments.
The following lemma
summarises these properties, thus giving a clear picture of the billiard dynamics near $\partial \mathcal M$.

\begin{lemma}\label{lem:BilliardProperties}
The billiard map $f: \mathcal{M} \to \mathcal{M}$
satisfies the following properties.
\begin{enumerate}[(a)]
\item
\label{item:CircularDynamics}
If $a_j \le \varphi \le \varphi + 2 \theta \le b_j$, then
$(\bar{\varphi}, \bar{\theta}) =
 f(\varphi,\theta) = (\varphi + 2 \theta, \theta)$.

\item
\label{item:Hubacher}
Let $g(\theta;\mu) = \acos\big( (1-\mu^2) + \mu^2 \cos\theta \big)$
and $\mu_j = \sqrt{r_j/r_{j+1}} \neq 1$.
If $0 < \theta \le \delta_j$ and
$\bar{\theta} = g(\theta;\mu_j) \le \delta_{j+1}$,
then
\begin{equation}
\label{eq:Hubacher}
f(b_j-\theta,\theta) = (a_{j+1} + \bar{\theta},\bar{\theta})
\qquad \mbox{and} \qquad
\begin{cases}
\bar{\theta} < \mu_j \theta, & \mbox{when $\mu_j < 1$,} \\
\bar{\theta} > \mu_j \theta, & \mbox{when $\mu_j > 1$.}
\end{cases}
\end{equation}
\item
\label{item:Balint_etal}
Given any $\epsilon > 0$ there exists $\psi = \psi(\epsilon) > 0$
such that
\[
\left.
\begin{array}{c}
f(\varphi,\theta) = (\bar{\varphi},\bar{\theta})
\mbox{ with } 0 < \theta \le \psi \\
\mbox{and } a_j \le \varphi \le a_{j+1} \le \bar{\varphi} \le a_{j+2} \\
\end{array}
\right\} \Rightarrow
\begin{cases}
\bar{\theta} > (\mu_j-\epsilon) \theta, & \mbox{when $\mu_j < 1$,} \\
\bar{\theta} < (\mu_j+\epsilon) \theta, & \mbox{when $\mu_j > 1$.}
\end{cases}
\]
\end{enumerate}
\end{lemma}

\begin{proof}
\begin{enumerate}[(a)]
\item
If $a_j \le \varphi \le \varphi + 2 \theta \le b_j$,
then $z(\varphi), z(\varphi+2\theta) \in \Gamma_j$,
so $f$ behaves as a circular billiard map,
in which case is well-known that
$f(\varphi,\theta) = (\varphi + 2 \theta, \theta)$.
\item
Set $\varphi = b_j - \theta$ and
$(\bar{\varphi},\bar{\theta}) = f(\varphi,\theta)$.
Condition $0 < \theta \le \delta_j$ implies that $z(\varphi) \in \Gamma_j$.
Identity $\varphi + \theta = b_j$ implies that lines
$L = L(\varphi,\theta)$ and $N_j$ are perpendicular,
where $N_j$ is the normal to $\Gamma$ at $z(b_j)$.
If, in addition, $z(\bar{\varphi}) \in \Gamma_{j+1}$,
then Hubacher proved~\eqref{eq:Hubacher}
in~\cite[page~486]{hubacher1987instability}.
Finally, we note that $\bar{\theta} \le \delta_{j+1}$ implies that
$z(\bar{\varphi}) \in \Gamma_{j+1}$.

\item
B\'alint \emph{et al.}~\cite{balint2011chaos} proved the following
generalisation of Hubacher computation.
Set $f(\varphi,\theta) = (\bar{\varphi},\bar{\theta})$.
If $a_j \le \varphi \le a_{j+1} \le \bar{\varphi} \le a_{j+2}$,
so that $z(\varphi) \in \Gamma_j$ and $z(\bar{\varphi}) \in \Gamma_{j+1}$,
then there exist angles $\varphi^+ \in [0,2\theta]$ and
$\varphi^- \in [0,2\bar{\theta}]$ such that
\[
\varphi = b_j - \varphi^+, \qquad
\bar{\varphi} = a_{j+1} + \varphi^-, \qquad
\varphi^+ + \varphi^- = \theta + \bar{\theta},
\]
and
\[
\bar{\theta} =
\acos \left( (1-\mu^2_j) \cos(\theta - \varphi^+) +
             \mu^2_j \cos \theta \right).
\]
Hubacher's computation corresponds to the case
$\varphi^+ = \theta$ and $\varphi^- = \bar{\theta}$.
We introduce the auxiliary coordinate
$s = 1 - \varphi^+/\theta \in [-1,1]$ and the positive function
\begin{equation}
\label{eq:Omega}
\Omega_j:[-1,1] \to \Rset_+, \qquad
\Omega_j(s) = \sqrt{\mu_j^2 + (1-\mu_j^2) s^2}.
\end{equation}
Function $\Omega_j(s)$ is even, $\Omega_j (0) = \mu_j$
and $\Omega_j(\pm 1) = 1$.
If $\mu_j < 1$,
then $\Omega_j(s)$ increases for $s > 0$ and decreases for $s < 0$,
so $\mu_j \le \Omega_j(s) \le 1$ for all $s \in [-1,1]$.
If $\mu_j > 1$,
then $\Omega_j(s)$ decreases for $s > 0$ and increases for $s < 0$,
so $1 \le \Omega_j(s) \le \mu_j$ for all $s \in [-1,1]$.

A straightforward computation with Taylor expansions shows that
\[
1 - \bar{\theta}^2/2 + \Order(\bar{\theta}^4) =
\cos \bar{\theta} =
(1-\mu^2_j) \cos(s\theta) + \mu^2_j \cos \theta =
1 - \Omega^2_j(s) \theta^2/2 + \Order(\theta^4) 
\]
as $\theta \to 0^+$,
where the error term $\Order(\theta^4)$ is uniform
in $s \in [-1,1]$ and $j=1,\ldots,k$.
Therefore,
\[
\bar{\theta} = \big[ \Omega_j(s) + \Order(\theta^2) \big] \theta 
\]
as $\theta \to 0^+$, where the error term $\Order(\theta^2)$ is uniform
in $s \in [-1,1]$ and $j=1,\ldots,k$.
\qedhere
\end{enumerate}
\end{proof}

If  $a_j \le \varphi < \varphi + 2 \theta = b_j$,
then $z(\varphi) \in \Gamma_j$ and
$z(\bar{\varphi}) = b_j = a_{j+1} \in
 \Gamma_j \cap \Gamma_{j+1}$,
but part~(a) of Lemma~\ref{lem:BilliardProperties} still applies,
because the tangents to $\Gamma_j$ and $\Gamma_{j+1}$ agree
at $z(\bar{\varphi})$ by the definition of circular polygon.
This fact will be used in Proposition~\ref{prop:RationalCircularGons}
to construct some special periodic \emph{nodal} billiard trajectories in
\emph{rational} circular polygons,
which are introduced in Definition~\ref{def:RationalCircularPolygon}.
Similar nodal periodic billiard trajectories were constructed
in~\cite{BuhovskyKaloshin2018} to answer a question about
length spectrum and rigidity.

Lemma~\ref{lem:BilliardProperties} and the above observation
describe two rather different ways in which the angle $\theta$ can vary
as a sliding billiard trajectory jumps from one arc to the next.
On the one hand,
if the trajectory impacts at the corresponding node,
there is no change: $\bar{\theta} = \theta$.
On the other hand,
if the billiard trajectory is perpendicular at the normal line
at the corresponding node, we have the largest possible change:
$\bar{\theta} < \mu_j \theta$ for $\mu_j < 1$ or
$\bar{\theta} > \mu_j \theta$ for $\mu_j > 1$.
The great contrast between these two situations is the crucial fact
behind the non-existence of caustics near the boundary obtained by
Hubacher~\cite{hubacher1987instability}.
It is also the main ingredient to obtain all chaotic properties
stated in the introduction.

Next, we introduce the main geometric subsets of the phase
space $\mathcal{M} = \Tset \times [0,\pi]$.
All of them are denoted with calligraphic letters.

\begin{definition}\label{def:SingularitySegment}
The \emph{$j$-singularity segment} is the vertical segment
$\mathcal{L}_j = \{a_j\} \times [0,\pi] \subset \mathcal{M}$.
Given any $s > 0$,
the \emph{$(j,\pm s)$-singularity segments} are the slanted segments
\[
\mathcal{L}_j^{-s} =
\big\{
(\varphi,\theta) \in \mathcal{M} : a_{j-1} \le \varphi = a_j - 2\theta s
\big\}, \quad
\mathcal{L}_j^s =
\big\{
(\varphi,\theta) \in \mathcal{M} : \varphi = a_j + 2\theta s \le a_{j+1}
\big\}.
\]
The \emph{$j$-fundamental domain} is the triangular domain
\[
\mathcal{D}_j = 
\left\{
(\varphi,\theta) \in \mathcal{M} :
a_j \le \varphi \le a_j + 2\theta \le a_{j+1}
\right\}.
\]
Finally,
$\mathcal{L} =
 \bigcup_{j=1}^k
 \left( \mathcal{L}_j \cup \mathcal{L}_j^{1/2} \cup \mathcal{L}_j^1 \right)$
is the \emph{extended singularity set}.
\end{definition}

Note that $\mathcal{L}^n_j \subset f^n(\mathcal{L}_j)$ for all
$n \in \Zset$, so $\mathcal{L}^s_j$ is a generalisation of
the forward and backward iterates under the billiard map
of the $j$-singularity segments when $s \not \in \Zset$.
We will only need the segments $\mathcal{L}^s_j$ for
values $s = n$ and $s = n + 1/2$ with $n\in \Zset$.
The left (respectively, right) side of the triangle $\mathcal{D}_j$
is contained in the vertical segment $\mathcal{L}_j$
(respectively, coincides with the slanted segment $\mathcal{L}_j^1$).

We have used the term `sliding' in a clumsy way until now.
Let us clarify its precise meaning.
Let $\Pi_\varphi:\mathcal{M} \to \Tset$ and
$\Pi_\theta:\mathcal{M} \to [0,\pi]$ be the projections
$\Pi_\varphi(\varphi,\theta) = \varphi$ and
$\Pi_\theta(\varphi,\theta) = \theta$.
Let $J: \mathcal{M} \setminus \mathcal{L} \to \Zset/k\Zset$ be the
piece-wise constant map defined by
$a_j < \Pi_\varphi(x) < b_j \Rightarrow J(x) = j$.
This map is well-defined since
$\Pi_\varphi(x) \not \in \{a_1,\ldots,a_k\} = \{b_1,\ldots,b_k\}$
when $x \not \in \mathcal{L}$.

\begin{definition}
\label{def:GenericSliding}
A billiard orbit is \emph{(counter-clockwise) sliding}
when any consecutive impact points are either in the same arc
or in consecutive arcs in the counter-clockwise direction.
An orbit is \emph{generic} when it avoids
the extended singularity set.
We denote by $\mathcal{S}_0$ the set of all initial conditions
that give rise to generic counter-clockwise sliding orbits.
That is,
\[
\mathcal{S}_0 =
\left\{
x \in \mathcal{M} :
\mbox{$J(f^{n+1}(x))-J(f^n(x)) \in \{0,1\}$ and
$f^n(x) \not \in \mathcal{L}$ for all $n \in \Zset$}
\right\}.
\]
\end{definition}

The \emph{(counter-clockwise) generic sliding set} $\mathcal{S}_0$
is $f$-invariant.
The term \emph{glancing}
---see, for instance, ~\cite{mather1982glancing}---
is also used in the literature,
but sliding is the most widespread term.
A consequence of part~(a) of Lemma~\ref{lem:BilliardProperties}
is that any generic sliding billiard orbit has exactly one point
in $\Interior \mathcal{D}_j$ on each turn around $\Gamma$.
This fact establishes the \emph{fundamental} character of $\mathcal{D}_j$.
Following the notation used in the introduction,
$\mathcal{S}_\pi$ is the clockwise generic sliding set,
but we are not going to deal with it.

\begin{remark}
\label{rem:GenericTrajectories}
If $x \in \mathcal{M}$ is a point such that
$x_i = (\varphi_i,\theta_i) = f^i(x) \in \mathcal{L}$
for some $i \in \Zset$, then its billiard trajectory
$\big( z_n = z(\Pi_\varphi(f^n(x))) \big)_{n \in \Zset}$
has some impact point $z_m \in \Gamma_\star$,
where $\Gamma_\star$ is the set of nodes~\eqref{eq:SetOfNodes},
or has two consecutive impact points $z_m \in \Gamma_j$
and $z_{m+1} \in \Gamma_{j+1}$ such that the segment from $z_m$
to $z_{m+1}$ is perpendicular to the normal to $\Gamma$ at the node
$\Gamma_j \cap \Gamma_{j+1}$.
\end{remark}

\begin{lemma}
\label{lem:ThetaIntersections}
Let $s,t \ge 0$ and $j$ such that $s + t \ge \delta_j/2\pi$.
Then $\mathcal{L}_j^s \cap \mathcal{L}_{j+1}^{-t} \neq \emptyset$
and
\[
\Pi_\varphi \left( \mathcal{L}_j^s \cap \mathcal{L}_{j+1}^{-t} \right) =
a_j + \frac{s\delta_j}{s+t},\qquad
\Pi_\theta \left( \mathcal{L}_j^s \cap \mathcal{L}_{j+1}^{-t} \right) =
\frac{\delta_j}{2s+2t}.
\]
\end{lemma}

\proof
By definition,
$(\varphi,\theta) \in \mathcal{L}_j^s \cap \mathcal{L}_{j+1}^{-t}
 \Leftrightarrow
 a_j \le a_j+ 2\theta s = \varphi = a_{j+1} - 2\theta t \le a_{j+1}$.
Identity $a_j + 2\theta s = a_{j+1} - 2\theta t$ implies that
$2\theta = \delta_j/(s+t)$.
Then inequality $s+t \ge \delta_j/2\pi$ implies that $\theta \le \pi$.
Finally, $a_j + 2\theta s \le a_{j+1}$ and $a_{j+1} - 2\theta t \ge a_j$
because $s/(s+t), t/(s+t) \le 1$.
\qed

This lemma implies that segments
$\mathcal{L}_{j+1}^{-n+1}$ and $\mathcal{L}_{j+1}^{-n}$
intersect segments $\mathcal{L}_j$ and $\mathcal{L}_j^1$
for any integer $n \ge 2 > 1 + \delta_j/2\pi$,
so the following definition makes sense.
See Figure~\ref{fig:FundamentalQuadrilaterals}.

\begin{definition}\label{def:FundamentalQuadrilaterals}
Let $n$ be an integer such that $n \ge 2$.
The \emph{$(j,n)$-fundamental quadrilateral} is the oriented cell
$\widetilde{\mathcal{Q}}_{j,n} \subset \mathcal{M}$ bounded by
$\mathcal{L}_j$ (left side), $\mathcal{L}_{j+1}^{-n}$ (base side),
$\mathcal{L}_j^1$ (right side) and $\mathcal{L}_{j+1}^{-n+1}$ (top side).
We split $\mathcal{Q}_{j,n}$ in two by means of the segment
$\mathcal{L}_j^{-n+1/2}$, which gives rise to two smaller oriented cells:
$\widetilde{\mathcal{Q}}^-_{j,n}$ (the lower one)
and $\widetilde{\mathcal{Q}}^+_{j,n}$ (the upper one),
whose left and right sides are still contained in
$\mathcal{L}_j$ and $\mathcal{L}_j^1$, respectively.
We say that $\tilde{\mathcal{Q}}^\pm_{j,n}$ is the
\emph{$(\pm, j, n)$-fundamental quadrilateral}.
\end{definition}

\begin{figure}[t]
\begin{center}
\includegraphics*[width=15.9cm]{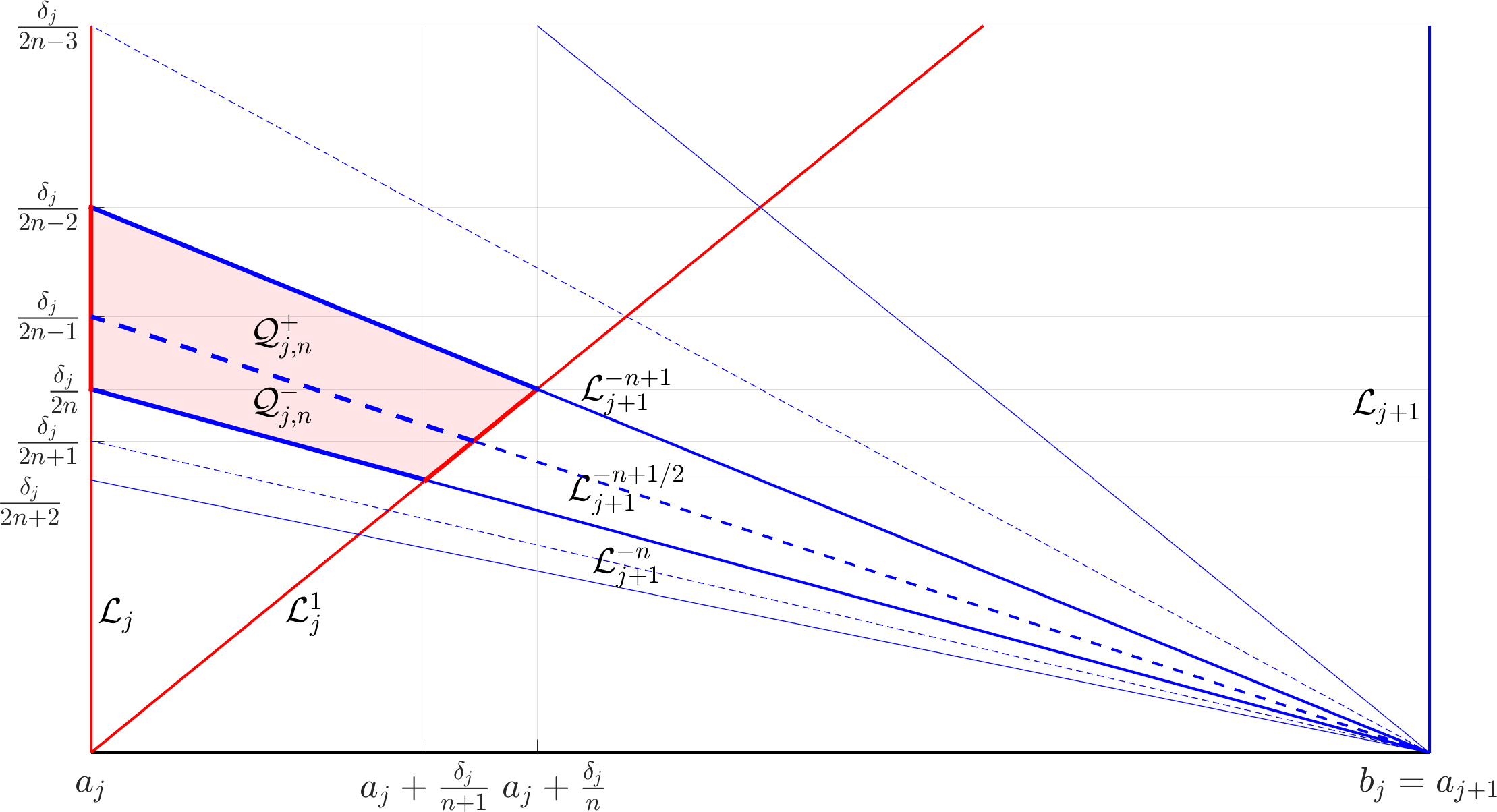}
\end{center}
\caption{\small The fundamental quadrilateral
$\mathcal{Q}_{j,n} =  \mathcal{Q}^-_{j,n} \cup \mathcal{Q}^+_{j,n}$.
Its horizontal (base and top) and vertical (left and right) sides
are displayed in blue and red, respectively.
Accordingly, singularity segments $\mathcal{L}^s_j$ and
$\mathcal{L}^{-t}_{j+1}$ are displayed in red and blue, respectively.
Moreover, they are displayed with continuous and dashed lines
when $s,t \in \Nset$ and $t \in \Nset + \frac{1}{2}$,
respectively.
This is a quantitative representation computed for
$\delta_j = \pi/2$ and $n = 3$.
The images $f^n(\mathcal{Q}^\pm_{j,n})$ are displayed in
Figure~\ref{fig:FundamentalLemma}.
}
\label{fig:FundamentalQuadrilaterals}
\end{figure}

In order to find sufficient conditions for
$f^n: \widetilde{\mathcal{Q}}^\varsigma_{j,n} \stretches
\widetilde{\mathcal{Q}}^{\varsigma'}_{j+1,n'}$,
we need the extreme values of $\Pi_\theta(f^n(x))$ when $x$ moves on
the horizontal sides of $\widetilde{\mathcal{Q}}^\varsigma_{j,n}$
and the extreme values of $\Pi_\theta(x)$ when
$x \in \mathcal{Q}_{j+1,n'} =
 \mathcal{Q}^-_{j+1,n'} \cup \mathcal{Q}^+_{j+1,n'}$.
These extreme values are defined and estimated in the lemma below.
Those estimates are used in the proof of Lemma~\ref{lem:Fundamental}.

\begin{lemma}
\label{lem:MinMax}
Fix any $j$.
With the above notations,
if $\chi_j \ge 2$ is a large enough integer,
then the following properties hold for all $n \ge \chi_j$.
\begin{enumerate}[(a)]
\item
$\nu_{j,n} :=
 \min_{x \in \mathcal{Q}_{j,n}} \Pi_\theta(x) = \delta_j/(2n+2)$
and
$\omega_{j,n} :=
 \max_{x \in \mathcal{Q}_{j,n}} \Pi_\theta(x) = \delta_j/(2n-2)$.
\item
If $\nu^s_{j,n} :=
\min_{x \in \mathcal{Q}_{j,n} \cap \mathcal{L}^{-n+s}_{j+1}}
\Pi_\theta(f^n(x))$ and
$\omega^s_{j,n} :=
\max_{x \in \mathcal{Q}_{j,n} \cap \mathcal{L}^{-n+s}_{j+1}}
\Pi_\theta(f^n(x))$, then
\begin{enumerate}[i)]
\item
$\nu^0_{j,n} = \delta_j/(2n+2)$,
$\nu^1_{j,n} = \delta_j/2n = \omega^0_{j,n}$
and $\omega^1_{j,n} = \delta_j/(2n-2)$;
\item
$\omega^{1/2}_{j,n} < \mu_j \delta_j/(2n-1)$ when $\mu_j < 1$; and
\item
$\nu^{1/2}_{j,n} > \mu_j \delta_j/(2n+1)$ when $\mu_j > 1$.
\end{enumerate}
\end{enumerate}
\end{lemma}

\begin{proof}
The fundamental domain $\mathcal{Q}_{j,n}$ is
only well-defined for $n \ge 2$.
The reader must keep in mind Lemmas~\ref{lem:BilliardProperties}
and~\ref{lem:ThetaIntersections}.
See Figure~\ref{fig:FundamentalQuadrilaterals} for a visual guide.
\begin{enumerate}[(a)]
\item
The minimum and maximum values are attained at the intersections
$\mathcal{L}_j^1 \cap \mathcal{L}_{j+1}^{-n}$ and
$\mathcal{L}_j \cap \mathcal{L}_{j+1}^{-n+1}$, respectively.
\item
\begin{enumerate}[i)]
\item
If $x \in \mathcal{Q}_{j,n} \cap \mathcal{L}^{-n}_{j+1}$
or $x \in \mathcal{Q}_{j,n} \cap \mathcal{L}^{-n+1}_{j+1}$,
then $\Pi_\theta(f^n(x)) = \Pi_\theta(x)$
by part~(\ref{item:CircularDynamics}) of Lemma~\ref{lem:BilliardProperties}.
Therefore, the four extreme values $\nu^0_{j,n}$, $\nu^1_{j,n}$,
$\omega^0_{j,n}$ and $\omega^1_{j,n}$ are attained at
the four intersections
$\mathcal{L}_j^1 \cap \mathcal{L}_{j+1}^{-n}$,
$\mathcal{L}_j^1 \cap \mathcal{L}_{j+1}^{-n+1}$,
$\mathcal{L}_j \cap \mathcal{L}_{j+1}^{-n}$ and
$\mathcal{L}_j \cap \mathcal{L}_{j+1}^{-n+1}$, respectively.
\item
First, the value
$\max_{x \in  \mathcal{Q}_{j,n} \cap \mathcal{L}^{-n+1/2}_{j+1}}
 \Pi_\theta(x)$
is attained at $\mathcal{L}_j \cap \mathcal{L}_{j+1}^{-n+1/2}$.
Second,
if $x \in  \mathcal{Q}_{j,n} \cap \mathcal{L}^{-n+1/2}_{j+1}$ and
$\mu_j < 1$, then $\Pi_\theta(f^n(x)) < \mu_j \Pi_\theta(x)$
by part~(\ref{item:Hubacher}) of Lemma~\ref{lem:BilliardProperties}.

We need hypotheses $0 < \theta \le \delta_j$ and
$\bar{\theta} = g(\theta;\mu_j) \le \delta_{j+1}$ to apply
Lemma~\ref{lem:BilliardProperties}.
In order to guarantee them, it suffices to take $n \ge \chi_j$ with
\begin{equation}
\label{eq:chi_mulessthan1}
\chi_j \ge 1 + \lceil \mu_j \delta_j/2\delta_{j+1} \rceil,
\end{equation}
since then $\chi_j \ge 2$ and $2\chi_j - 2 \ge \mu_j \delta_j/\delta_{j+1}$,
so $\theta \le \omega_{j,n} \le \delta_j/(2\chi_j -2) \le
     \delta_j/2 < \delta_j$
and
$\bar{\theta} < \mu_j \theta \le \mu_j \delta_j/(2\chi_j -2) \le \delta_{j+1}$.
Here $\lceil \cdot \rceil$ denotes the \emph{ceiling} function.

\item
First, the value
$\min_{x \in  \mathcal{Q}_{j,n} \cap \mathcal{L}^{-n+1/2}_{j+1}}
 \Pi_\theta(x)$
is attained at $\mathcal{L}_j^1 \cap \mathcal{L}_{j+1}^{-n+1/2}$.
Second,
if $x \in  \mathcal{Q}_{j,n} \cap \mathcal{L}^{-n+1/2}_{j+1}$ and
$\mu_j > 1$, then $\Pi_\theta(f^n(x)) > \mu_j \Pi_\theta(x)$
by part~(\ref{item:Hubacher}) of Lemma~\ref{lem:BilliardProperties}.
We still need hypotheses $0 < \theta \le \delta_j$ and
$\bar{\theta} = g(\theta;\mu_j) \le \delta_{j+1}$ in
Lemma~\ref{lem:BilliardProperties}.
In order to guarantee them,
it suffices to take $n \ge \chi_j$ for some large enough integer $\chi_j$,
since $\lim_{n \to +\infty} \omega_{j,n} = 0$ and
$\lim_{\theta \to 0^+} g(\theta;\mu_j) = 0$.
\qedhere
\end{enumerate}
\end{enumerate}
\end{proof}

The following lemma (which we refer to as the \emph{fundamental lemma})
is the key step in constructing generic sliding billiard trajectories that
approach the boundary in optimal time, and in constructing symbolic dynamics.
It describes which fundamental quadrilaterals
in $\mathcal{D}_{j+1}$ we can `nicely' visit if we start in a given
fundamental quadrilateral in $\mathcal{D}_j$.
See Figure~\ref{fig:FundamentalLemma} for a visual guide.

\begin{lemma}[Fundamental Lemma]\label{lem:Fundamental}
With the above notations, let
\begin{equation}
\label{eq:SetUpsilon}
\Upsilon_j =
\left\{
(n,n') \in \Nset^2 :
\alpha_j^- n + \beta_j^- \le n' \le \alpha_j^+ n - \beta_j^+,
\ n \ge \chi_j, \ n' \ge \chi_{j+1}
\right\},
\end{equation}
where $\alpha_j^- = \delta_{j+1}/\delta_j\max\{1,\mu_j\}$,
$\alpha_j^+ = \delta_{j+1}/\delta_j\min\{1,\mu_j\}$ and
$\beta_j^\pm = \alpha_j^\pm + 1$ for all $j$.
Then
$f^n:\widetilde{\mathcal{Q}}^\varsigma_{j,n} \stretches
    \widetilde{\mathcal{Q}}^{\varsigma'}_{j+1,n'}$
for all $j$, $(n,n') \in \Upsilon_j$ and
$\varsigma,\varsigma' \in \{-,+\}$.
\end{lemma}

\begin{figure}[t]
\begin{center}
\includegraphics*[width=15.9cm]{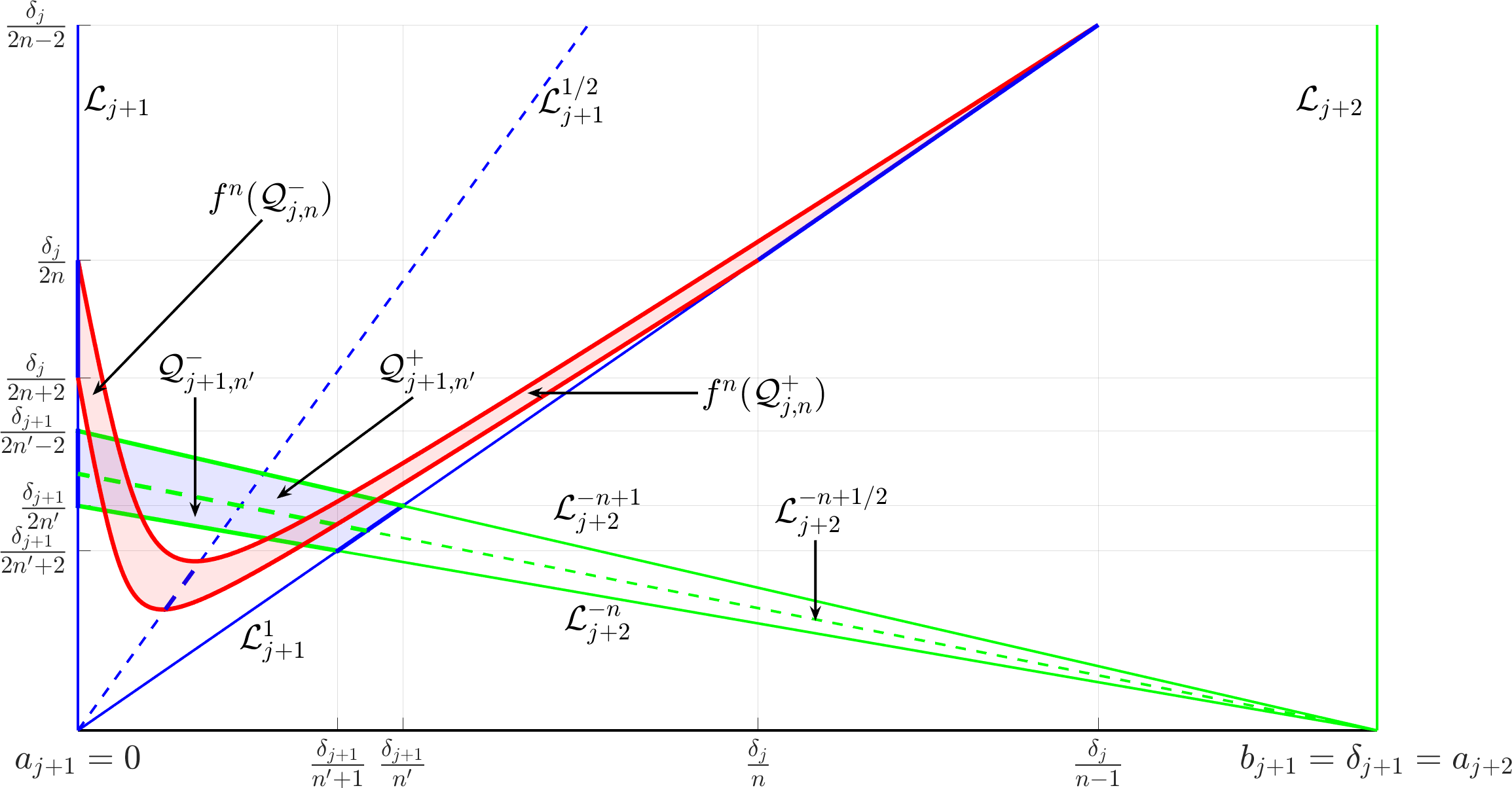}
\end{center}
\caption{\small Overlapping of the image under $f^n$
of the fundamental quadrilateral
$\mathcal{Q}_{j,n} =  \mathcal{Q}^-_{j,n} \cup \mathcal{Q}^+_{j,n}$
displayed in Figure~\ref{fig:FundamentalQuadrilaterals} with
the `target' fundamental quadrilateral
$\mathcal{Q}_{j+1,n'} =  \mathcal{Q}^-_{j+1,n'} \cup \mathcal{Q}^+_{j+1,n'}$
in the case $\mu_j < 1$.
Singularity segments $\mathcal{L}^s_{j+1}$ and
$\mathcal{L}^{-t}_{j+2}$ are displayed in blue and green, respectively.
Besides, they are displayed with continuous and dashed lines
when $s,t \in \Nset$ and $t \in \Nset + \frac{1}{2}$,
respectively.
The upper and lower red thick continuous curves are the images
of the left and right sides of $\mathcal{Q}_{j,n}$, respectively.
We have assumed that $a_{j+1} = 0$ to avoid the overlapping
of labels in the horizontal axis.
This is a quantitative representation computed for
$\delta_j = \pi/2$, $\delta_{j+1} = 1$, $\mu_j = \sqrt{r_j/r_{j+1}} = 0.3$,
$n = 3$ and $n' = 4$, showing that
$f^n: \widetilde{\mathcal{Q}}^\varsigma_{j,n} \stretches
      \widetilde{\mathcal{Q}}^{\varsigma'}_{j+1,n'}$
for all $\varsigma,\varsigma' \in \{ -,+ \}$.
The parabolic-like shape of $f^n(\mathcal{Q}_{j,n})$ was expected,
see~\eqref{eq:Omega}.
}
\label{fig:FundamentalLemma}
\end{figure}

\begin{proof}
We fix an index $j$ such that $\mu_j < 1$,
so $\alpha^-_j = \delta_{j+1}/\delta_j$ and
$\alpha_j^+ = \delta_{j+1}/\mu_j \delta_j$.
Let $(n,n') \in \Upsilon_j$.
We want to show that
$f^n:\widetilde{\mathcal{Q}}^\varsigma_{j,n} \stretches
     \widetilde{\mathcal{Q}}^{\varsigma'}_{j+1,n'}$
for any $\varsigma,\varsigma' \in \{-,+\}$.

Let $\gamma^\varsigma:[a,b] \to \mathcal{Q}^\varsigma_{j,n}$ be
a vertical path in $\widetilde{Q}^\varsigma_{j,n}$.
We assume, without loss of generality,
that $\gamma^\varsigma(a)$ and $\gamma^\varsigma(b)$ belong to the
base side and top side of $\tilde{\mathcal{Q}}^\varsigma_{j,n}$.
Note that $\overline{\mathcal{D}_{j+1} \setminus \mathcal{Q}_{j+1,n'}}$
has a connected component above and other one below $\mathcal{Q}_{j+1,n'}$.
By continuity and using that
$\mathcal{Q}^{\varsigma'}_{j+1,n'} \subset \mathcal{Q}_{j+1,n'}$ and
$f^n(\gamma^\varsigma) \subset \mathcal{D}_{j+1}$,
we know that if $f^n(\gamma^\varsigma(a))$ and $f^n(\gamma^\varsigma(b))$
are in different connected components of
$\overline{\mathcal{D}_{j+1} \setminus \mathcal{Q}_{j+1,n'}}$,
then there is a subpath
$\eta^{\varsigma,\varsigma'} \subset \gamma^\varsigma$ such that
the image path $f^n(\eta^{\varsigma,\varsigma'})$ is vertical in
$\widetilde{\mathcal{Q}}^{\varsigma'}_{j+1,n'}$.
Thus, we only have to check that endpoints
$f^n(\gamma^\varsigma(a))$ and $f^n(\gamma^\varsigma(b))$
are in different connected components of
$\overline{\mathcal{D}_{j+1} \setminus \mathcal{Q}_{j+1,n'}}$
for any path $\gamma^\varsigma$ vertical in $\widetilde{Q}^\varsigma_{j,n}$.

First, we consider the case $\varsigma = -$, so
$\gamma^-(a) \in \mathcal{Q}^-_{j,n} \cap \mathcal{L}^{-n}_{j+1}$ and
$\gamma^-(b) \in \mathcal{Q}^-_{j,n} \cap \mathcal{L}^{-n+1/2}_{j+1}$.
We deduce from Lemma~\ref{lem:MinMax} that
\begin{align*}
\Pi_\theta(f^n(\gamma^-(a))) &\ge
\min_{x \in \mathcal{Q}^-_{j,n} \cap \mathcal{L}^{-n}_{j+1}}
\Pi_\theta(f^n(x)) =
\min_{x \in \mathcal{Q}_{j,n} \cap \mathcal{L}^{-n}_{j+1}}
\Pi_\theta(f^n(x)) = \delta_j/(2n+2), \\
\Pi_\theta(f^n(\gamma^-(b))) &\le
\max_{x \in \mathcal{Q}^-_{j,n} \cap \mathcal{L}^{-n+1/2}_{j+1}}
\Pi_\theta(f^n(x)) < \mu_j \delta_j/(2n-1).
\end{align*}
The above results and identities
$\nu_{j+1,n'} = \delta_{j+1}/(2n'+2)$ and
$\omega_{j+1,n'} = \delta_{j+1}/(2n'-2)$ imply that
if inequalities
\begin{equation}
\label{eq:StrictInequalities1}
\mu_j \delta_j/(2n-1) \le \delta_{j+1}/(2n'+2), \qquad
\delta_{j+1}/(2n'-2) \le \delta_j/(2n+2)
\end{equation}
hold, then $f^n(\gamma^-(a))$ and $f^n(\gamma^-(b))$ are
above and below $\mathcal{Q}_{j+1,n'}$, respectively,
so they are in different connected components of
$\overline{\mathcal{D}_{j+1} \setminus \mathcal{Q}_{j+1,n'}}$
for any path $\gamma^-$ vertical in $\widetilde{Q}^-_{j,n}$.

Second, we consider the case $\varsigma = +$, so
$\gamma^+(a) \in \mathcal{Q}^+_{j,n} \cap \mathcal{L}^{-n+1/2}_{j+1}$ and
$\gamma^+(b) \in \mathcal{Q}^+_{j,n} \cap \mathcal{L}^{-n+1}_{j+1}$.
Similar arguments show that if inequalities
\begin{equation}
\label{eq:StrictInequalities2}
\mu_j \delta_j/(2n-1) \le \delta_{j+1}/(2n'+2), \qquad
\delta_{j+1}/(2n'-2) \le \delta_j/2n
\end{equation}
hold, then $f^n(\gamma^+(a))$ and $f^n(\gamma^+(b))$ are
below and above $\mathcal{Q}_{j+1,n'}$, respectively,
so they are in different connected components of
$\overline{\mathcal{D}_{j+1} \setminus \mathcal{Q}_{j+1,n'}}$
for any path $\gamma^+$ vertical in $\widetilde{Q}^+_{j,n}$.

Finally, after a straightforward algebraic manipulation,
we check that inequalities~(\ref{eq:StrictInequalities1}) 
and~(\ref{eq:StrictInequalities2}) hold for any $(n,n') \in \Upsilon_j$.
This ends the proof for the case $\mu_j < 1$.

The case $\mu_j > 1$ follows from similar arguments.
We skip the details.
\end{proof}

No inequality in~\eqref{eq:SetUpsilon} is strict.
However, we need some strict inequalities for a technical reason.
Let us explain it.
We will use the objects defined above and the fundamental lemma
to construct our symbolic dynamics in Section~\ref{sec:ChaoticMotions}.
However, if we were to try to construct our symbolic dynamics directly with
the symbol sets being the fundamental quadrilaterals $Q^\varsigma_{j,n}$,
we would run into problems at the boundaries,
where neighboring quadrilaterals intersect.
To be precise,
$\mathcal{Q}^+_{j,n}$ and $\mathcal{Q}^-_{j,n}$ have a common
side contained in $\mathcal{L}_j^{-n+1/2}$, whereas
$\mathcal{Q}^-_{j,n}$ and $\mathcal{Q}^+_{j,n+1}$ have a
common side contained in $\mathcal{L}_j^{-n}$.
The following corollary to Lemma~\ref{lem:Fundamental} solves
this problem by establishing the existence of
pairwise disjoint strict horizontal slabs
$\widetilde{\mathcal{K}}^\varsigma_{j,n} \varsubsetneq_{\rm h}
 \widetilde{\mathcal{Q}}^\varsigma_{j,n}$ with
\emph{exactly} the same stretching properties as the original
fundamental quadrilaterals $\widetilde{\mathcal{Q}}^\varsigma_{j,n}$.
It requires some strict inequalities.

\begin{corollary}[Fundamental Corollary]
\label{cor:Fundamental}
With the above notations, let
\[
\Xi_j =
\left\{
(n,n') \in \Nset^2 :
\alpha_j^- n + \beta_j^- < n' < \alpha_j^+ n - \beta_j^+,
\ n \ge \chi_j, \ n' \ge \chi_{j+1}
\right\},
\]
where $\alpha_j^- = \delta_{j+1}/\delta_j\max\{1,\mu_j\}$,
$\alpha_j^+ = \delta_{j+1}/\delta_j\min\{1,\mu_j\}$ and
$\beta_j^\pm = \alpha_j^\pm + 1$ for all $j$.
There are pairwise disjoint strict horizontal slabs
\[
\widetilde{\mathcal{K}}^\varsigma_{j,n} \varsubsetneq_{\rm h}
 \widetilde{\mathcal{Q}}^\varsigma_{j,n}, \qquad
\forall j, \ n \ge \chi_j, \ \varsigma \in \{-,+\},
\]
such that:
\begin{enumerate}[(a)]
\item
$f^n:\widetilde{\mathcal{K}}^\varsigma_{j,n} \stretches
     \widetilde{\mathcal{K}}^{\varsigma'}_{j+1,n'}$; and
\item
$f^n(\mathcal{K}^{\varsigma}_{j,n}) \cap
 \mathcal{K}^{\varsigma'}_{j+1,n'} \cap
 \mathcal{L} = \emptyset$,
\end{enumerate}
for all $j$, $(n,n') \in \Xi_j$
and $\varsigma,\varsigma' \in \{-,+\}$.
(See Definition~\ref{def:SingularitySegment} for the meaning of
 $\mathcal{L}$.)
\end{corollary}

\begin{proof}
Inequalities~(\ref{eq:StrictInequalities1}) and~(\ref{eq:StrictInequalities2})
become \emph{strict} for any $(n,n') \in  \Xi_j$.
We consider the oriented cells
$\widetilde{\mathcal{R}}_{j+1,n}$, where
\[
\mathcal{R}_{j+1,n} =
\bigcup_{(n,n') \in \Xi_j} \mathcal{Q}_{j+1,n'} \subset
\mathcal{D}_{j+1},
\]
and orientations are chosen in such a way that the left and right sides
of these big oriented cells are still contained in
$\mathcal{L}_{j+1}$ and $\mathcal{L}_{j+1}^1$, respectively.
Note that $\mathcal{D}_{j+1} \setminus \mathcal{R}_{j+1,n}$ has
a connected component above and other one below $\mathcal{R}_{j+1,n}$.

Fix an index $j$ such that $\mu_j < 1$.
Let $n \ge \chi_j$.
We refer to Figure~\ref{fig:FundamentalLemma} for a visual guide.
The reader should imagine that the blue quadrilateral shown in
that figure is our whole cell $\mathcal{R}_{j+1,n}$.

\emph{Strict} versions of inequalities~(\ref{eq:StrictInequalities1})
and~(\ref{eq:StrictInequalities2}) imply that the images by $f^n$ of
both the base side of $\widetilde{\mathcal{Q}}^-_{j,n}$
and the top side of $\widetilde{\mathcal{Q}}^+_{j,n}$ are
\emph{strictly} above $\mathcal{R}_{j+1,n}$; whereas
the image by $f^n$ of the top side of $\widetilde{\mathcal{Q}}^-_{j,n}$,
which coincides with the base side of $\widetilde{\mathcal{Q}}^+_{j,n}$,
is \emph{strictly} below $\mathcal{R}_{j+1,n}$.
Thus, there are strict horizontal slabs
$\widetilde{\mathcal{K}}^\varsigma_{j,n} \varsubsetneq_{\rm h}
 \widetilde{\mathcal{Q}}^\varsigma_{j,n}$, with $\varsigma \in \{-,+\}$,
such that the images by $f^n$ of both
the base side of $\widetilde{\mathcal{K}}^-_{j,n}$
and the top side of $\widetilde{\mathcal{K}}^+_{j,n}$ are
strictly above $\mathcal{R}_{j+1,n}$; whereas
the image by $f^n$ of both the top side of $\widetilde{\mathcal{K}}^-_{j,n}$
and the base side of $\widetilde{\mathcal{K}}^+_{j,n}$
are strictly below $\mathcal{R}_{j+1,n}$.
Consequently,
$f^n(\mathcal{K}^{\varsigma}_{j,n}) \cap \mathcal{R}_{j+1,n} \cap
 \mathcal{L} = \emptyset$ and
$f^n:\widetilde{\mathcal{K}}^\varsigma_{j,n} \stretches
     \widetilde{\mathcal{R}}_{j+1,n}$, which implies that
$f^n(\mathcal{K}^{\varsigma}_{j,n}) \cap
 \mathcal{K}^{\varsigma'}_{j+1,n'} \cap
 \mathcal{L} = \emptyset$ and
$f^n:\widetilde{\mathcal{K}}^\varsigma_{j,n} \stretches
    \widetilde{\mathcal{K}}^{\varsigma'}_{j+1,n'}$
for all $(n,n') \in \Xi_j$ and $\varsigma,\varsigma' \in \{-,+\}$,
since
$\widetilde{\mathcal{K}}^{\varsigma'}_{j+1,n'} \subset_{\rm h}
 \widetilde{\mathcal{R}}_{j+1,n}$ for all $(n,n') \in \Xi_j$ and
 $\varsigma' \in \{-,+\}$.
This ends the proof of the stretching and intersecting properties
when $\mu_j < 1$. The case $\mu_j > 1$ follows from similar arguments.
We omit the details.

Finally,
these strict horizontal slabs are necessarily pairwise disjoint
because the original fundamental quadrilaterals
$\widetilde{\mathcal{Q}}^\varsigma_{j,n}$ only
share some of their horizontal sides.
\end{proof}

\section{Symbols, shift spaces and shift maps}
\label{sec:Symbols}

In this section,
we define an alphabet $\Qv \subset \Zset^k$ with infinitely many symbols,
then we consider two shift spaces
$\Qf^+ \subset \Qv^{\Nset_0}$ and $\Qf \subset \Qv^\Zset$ of
admissible one-sided and two-sided sequences,
and finally we study some properties of the
shift map $\sigma:\Qf \to \Qf$.
We present these objects in a separate section,
minimizing their relation with circular polygons,
since we believe that they will be useful in future works about
other problems.

For brevity, we will use the term shift instead of subshift,
but $\Qf^+ \varsubsetneq \Qv^{\Nset_0}$ and $\Qf \varsubsetneq \Qv^\Zset$. The sets $\Qv$, $\Qf^+$ and $\Qf$ are defined in terms of
some positive factors $\alpha^\pm_j$,
some positive addends $\beta^\pm_j$
and some integers $\chi_j \ge 2$ for $j=1,\ldots,k$, with $k \ge 1$.
(There are interesting billiard problems that will require $k < 3$,
 or even $k=1$.)

We only assume three hypotheses about these factors and integers:
\begin{itemize}
\item[({\bf A})]
$0 < \alpha^-_j < \alpha^+_j$ for $j=1,\ldots,k$;
\item[({\bf B})]
$\alpha := \alpha^+ > 1$ and $\alpha^+ \alpha^- = 1$,
where $\alpha^\pm = \prod_{j=1}^k \alpha^\pm_j$; and
\item[({\bf X})]
Integers $\chi_2,\ldots,\chi_k \ge 2$ are large enough and
$\chi_1 \gg \chi_2,\ldots,\chi_k$.
\end{itemize}
There is no assumption on the addends.
Clearly, all arcs $\Gamma_1,\ldots,\Gamma_k$ of the circular polygon~$\Gamma$
are equally important,
so $\chi_1 \gg \chi_2,\ldots,\chi_k$ is a purely technical hypothesis.
It is used only once, at the beginning of the proof of
Lemma~\ref{lem:AdmissibleSymbols}.
It is needed just to establish the topological transitivity of
the subshift map.
The rest of Theorem~\ref{thm:SymbolicDynamicsIntro},
as well as Theorems~\ref{thm:BoundaryIntro}, \ref{thm:PeriodicIntro}
and~\ref{thm:AsymptoticConstantIntro} do not need it.

We remark two facts related to billiards in circular polygons,
although we forget about billiards in the rest of this section.
The first remark is a trivial verification.

\begin{lemma}\label{lem:Factors}
The factors $\alpha^\pm_j$ defined in Lemma~\ref{lem:Fundamental}
satisfy hypotheses {\rm ({\bf A})} and {\rm ({\bf B})}.
\end{lemma}

\begin{proof}
Hypothesis ({\bf A}) follows from properties $\mu_j \neq 1$.
Hypothesis ({\bf B}) follows from the telescopic products
\[
\prod_{j=1}^k \frac{\delta_j}{\delta_{j+1}} = 1, \qquad
\prod_{j=1}^k \mu_j = \prod_{j=1}^k \sqrt{\frac{r_j}{r_{j+1}}} = 1,
\]
which are easily obtained from the cyclic identities
$\delta_{k+1} = \delta_1$ and $r_{k+1} = r_1$.
\end{proof}

The second remark is that we prefer to encode in a single symbol
all information related to each complete turn around $\Gamma$,
although we may construct our symbolic dynamics directly with
the disjoints sets $\widetilde{\mathcal{K}}^\varsigma_{j,n}$ as symbols.
That is, if a generic sliding orbit follows,
along a complete turn around the boundary $\Gamma$,
the itinerary
\[
\mathcal{K}^{\varsigma_1}_{1,n_1} \subset \mathcal{D}_1,
\mathcal{K}^{\varsigma_2}_{2,n_2} \subset \mathcal{D}_2, \ldots,
\mathcal{K}^{\varsigma_k}_{k,n_k} \subset \mathcal{D}_k,
\]
where $\widetilde{\mathcal{K}}^\varsigma_{j,n}$ are the pairwise disjoint
horizontal slabs described in Corollary~\ref{cor:Fundamental},
then we construct the symbol
\[
\qv = (q_1,\ldots,q_k) \in \Zset^k, \qquad
|q_j| = n_j, \qquad
\Sign(q_j) = \varsigma_j,
\]
which motivates the following definition.

\begin{definition}
\label{def:Alphabet}
The \emph{alphabet of admissible symbols} is the set
\[
\Qv =
\left\{
\qv = (q_1,\ldots,q_k) \in \Zset^k :
\begin{array}{l}
\alpha_j^- |q_j| + \beta_j^- < |q_{j+1}| <
\alpha_j^+ |q_j| - \beta_j^+, \, \mbox{$\forall j=1,\ldots,k-1$} \\
|q_j| \ge \chi_j, \mbox{ $\forall j=1,\ldots,k$}
\end{array}
\right\}.
\]
\end{definition}

This alphabet has infinitely many symbols by hypothesis~({\bf A}).
Thinking in the billiard motivation behind these symbols,
we ask that symbols associated to consecutive turns around $\Gamma$
satisfy the following admissibility condition.

\begin{definition}
\label{def:Admissible}	
We say that a finite, one-sided, or two-sided sequence of admissible symbols
$\qf = (\qv^i)_{i \in I} \subset \Qv$,
with $\qv^i = (q_1^i,\ldots,q_k^i)$,
is \emph{admissible} if and only if
\[
\alpha_k^- |q^i_k| + \beta_k^- < |q^{i+1}_1| <
\alpha_k^+ |q^i_k| - \beta_k^+, \qquad \forall i.
\]
\end{definition}

Admissible sequences are written with Fraktur font: $\qf$.
Its vector symbols are written with boldface font and labeled
with superscripts: $\qv^i$.
Components of admissible symbols are written with the standard font
and labeled with subscripts: $q^i_j$ or $q_j$.

\begin{definition}
\label{def:ShiftSpace}
The \emph{shift spaces of admissible sequences} are
\begin{align*}
\Qf^+ &=
\left\{ \qf = (\qv^i)_{i \in \Nset_0} \in \Qv^{\Nset_0} :
  \mbox{$\qf$ is admissible} \right\}, \\
\Qf &=
\left\{ \qf = (\qv^i)_{i \in \Zset} \in \Qv^\Zset :
  \mbox{$\qf$ is admissible} \right\}.
\end{align*}
\end{definition}

If $\qv = (q_1,\ldots,q_k) \in \Qv$, then we write
$|\qv| = |q_1| + \cdots + |q_k|$.
We equip $\Qf$ with the topology defined by the metric
\[
d_{\Qf} : \Qf \times \Qf \to [0,+\infty),\qquad
d_{\Qf} (\pf,\qf) =
\sum_{i \in \Zset}
\frac{1}{2^{|i|}} \frac{|\pv^i - \qv^i |}{1 + | \pv^i - \qv^i |}.
\]

We want to estimate the size of the maxima
(sometimes, the minima as well) of the sets
\begin{equation}
\label{eq:XiDefinition}
\Xi^i_j(n) =
\left\{
n^i_j \in \Nset :
\exists \qf \in \Qf \mbox{ such that $|q^0_1| = n$ and $|q^i_j| = n^i_j$}
\right\}, \quad
i \in \Zset, \ j=1,\ldots,k,
\end{equation}
when $n \ge \chi_1$ or $n \gg 1$.
We ask in~\eqref{eq:XiDefinition} for the existence of some $\qf \in \Qf$
---that is, some two-sided infinite sequence---,
but it does not matter. We would obtain exactly the same sets
just by asking the existence of some finite sequence
$q^0_1,\ldots,q^0_k,q^1_1,\ldots,q^1_k,\ldots,q^i_1,\ldots,q^i_j$
that satisfies the corresponding admissibility conditions.

Several estimates about maxima and minima of
sets~\eqref{eq:XiDefinition} are listed below.
Their proofs have been postponed to Appendix~\ref{app:AdmissibleSymbols}.

\begin{lemma}
\label{lem:AdmissibleSymbols}
We assume hypotheses {\rm ({\bf A})}, {\rm ({\bf B})} and {\rm ({\bf X})}.
Set $\zeta^i_j(n) = \min \Xi^i_j(n)$, $\xi^i_j(n) = \max \Xi^i_j(n)$.
Let $\rho^0(n) = n$ and
$\rho^i(n) = \sum_{j=1}^k \sum_{m=0}^{i-1} \xi^m_j(n)$ for all $i \ge 1$.
\begin{enumerate}[(a)]
\item
There are positive constants
$\nu < \lambda$, $\nu' < \lambda'$, $\tau < 1$ and $\gamma^\pm$,
which depend on factors $\alpha^\pm_j$ and addends $\beta^\pm_j$
but not on integers $\chi_j$, such that the following properties hold.
\begin{enumerate}[i)]
\item
\label{item:AdmissibleSymbols1}
$\nu n \le \xi^0_j(n) \le \lambda n$
for all $j=1,\ldots,k$ and $n \ge \chi_1$;
\item
\label{item:AdmissibleSymbols2}
$\tau \alpha^{|i|} \xi^0_j(n) \le \xi^i_j(n) \le \alpha^{|i|} \xi^0_ j(n)$
for all $j=1,\ldots,k$, $i \in \Zset$ and $n \ge \chi_1$;
\item
\label{item:AdmissibleSymbols3}
$\nu' \xi_j^i(n) \le \rho^i(n) \le \rho^{i+1}(n) \le \lambda' \xi_j^i(n)$
for all $j=1,\ldots,k$, $i \ge 0$, and $n \ge \chi_1$;
\item
\label{item:AdmissibleSymbols4}
$\zeta^1_1(n) \le \max\{ \chi_1,n/\alpha + \gamma^-\} \le n-1 < n+1 \le
 \alpha n - \gamma^+ \le \xi^1_1(n)$
for all $n > \chi_1$ and
$\zeta^1_1(n) = n < n+1 \le
 \alpha n - \gamma^+ \le \xi^1_1(n)$ for $n=\chi_1$; and
\item
\label{item:AdmissibleSymbols5}
Once fixed any $N \in \Nset$, we have that
\[
\chi_1 \le \zeta^1_1(n) \le n/\alpha + \gamma^- \le
n - N < n + N \le
 \alpha n - \gamma^+ \le \xi^1_1(n), 
\]
for all sufficiently large $n$. 
\end{enumerate}
\item
\label{item:XiSets}
$\Xi^i_j(n) = [\zeta^i_j(n), \xi^i_j(n)] \cap \Nset$
for all $j \mod k$, $i \in \Zset$ and $n \ge \chi_1$;
that is, $\Xi^i_j(n)$ has no gaps in $\Nset$.
Besides,
$\big[ \max\{\chi_1,n-|i|\}, n+|i| \big] \cap \Nset
\subset \Xi^i_1(n)$ for all $i \in \Zset$ and $n \ge \chi_1$.
\end{enumerate}
\end{lemma}

\begin{corollary}\label{cor:AdmissibleSequences}
We assume hypotheses {\rm ({\bf A})}, {\rm ({\bf B})} and {\rm ({\bf X})}.
\begin{enumerate}[(a)]
\item
Given any $\qv^-, \qv^+ \in \Qv$ there is
an admissible sequence of the form
$\big( \qv^-, \qv^1, \ldots, \qv^l, \qv^+ \big)$ for some $l \in \Nset$ and
$\qv^1,\ldots,\qv^l \in \Qv$.
\item
Given any $N \in \Nset$, there is a subset $\Qv_N \subset \Qv$,
with $\# \Qv_N = N$, such that the short sequence $(\qv,\qv')$
is admissible for all $\qv,\qv' \in \Qv_N$.
\item
$\Qf \neq \emptyset$.
\end{enumerate}
\end{corollary}

\begin{proof}
\begin{enumerate}[(a)]
\item
Let $l = |q^-_1 - q^+_1| - 1$.
Part~(\ref{item:XiSets}) of Lemma~\ref{lem:AdmissibleSymbols} implies
that $q^+_1 \in \Xi^{l+1}_1(q^-_1)$.
Therefore, we can construct iteratively such a sequence $\qv^1,\ldots,\qv^l$.
\item
Fix $N \in \Nset$.
Part~(\ref{item:AdmissibleSymbols5}) of Lemma~\ref{lem:AdmissibleSymbols}
implies that if $\qv,\qv' \in \Qv$ with $|q_1 - q'_1| \le N$ and
$|q_1|,|q'_1| \gg 1$, then $(\qv,\qv')$ is admissible.
So, we can take any subset
$\Qv_N = \{ \qv^1, \ldots, \qv^N\} \subset \Qv$
such that $|q_1^n| = |q^1_1| + n-1$ with $|q_1^1| \gg 1$.
Clearly, $\# \Qv_N = N$.
\item
Let $\qf = (\qv^i)_{i \in \Zset}$ with
$\qv^i = (q_1^i,\ldots,q_k^i) \in \Qv$
such that $|q_1^{i+1} - q_1^i| \le 1$.
Then $\qf \in \Qf$.
\qedhere
\end{enumerate}
\end{proof}

\begin{definition}
\label{def:ShiftMap}
The \emph{shift map} $\sigma: \Qf \to \Qf$, $\pf = \sigma(\qf)$,
is given by $\pf^i = \qf^{i+1}$ for all $i \in \Zset$.
\end{definition}

The following proposition tells us some important properties of
the shift map.
Note that by \emph{topological transitivity} we mean for any
nonempty open sets $U,V \subset \Qf$ there is $n \in \Nset$
such that $\sigma^n(U) \cap V \neq \emptyset$. 
If $N \in \Nset$, $\Sigma_N = \{1,\ldots,N\}^\Zset$ and
the shift map $\sigma_N:\Sigma_N \to \Sigma_N$,
$(t_i)_{i \in \Zset} = \sigma_N \big( (s_i)_{i \in \Zset} \big)$,
is given by $t_i = s_{i+1}$ for all $i \in \Zset$,
then we say that $\sigma_N:\Sigma_N \to \Sigma_N$ is the
\emph{full $N$-shift}.
We denote by $h_{\rm top}(f)$ the \emph{topological entropy} of a
continuous self-map $f$.

\begin{proposition}
\label{prop:ShiftMapProp}
We assume hypotheses {\rm ({\bf A})}, {\rm ({\bf B})} and {\rm ({\bf X})}.

The shift map $\sigma : \Qf \to \Qf$ exhibits topological transitivity and
sensitive dependence on initial conditions,
has infinite topological entropy,
and contains the full $N$-shift as a topological factor
for any $N \in \Nset$.
Besides, the subshift space of periodic admissible sequences
\[
\Pf =
\left\{
\qf \in \Qf : \mbox{$\exists p \in \Nset$ such that $\sigma^p(\qf) = \qf$}
\right\}
\]
is dense in the shift space $\Qf$.
\end{proposition}

\begin{proof}
On the one hand,
part~(a) of Corollary~\ref{cor:AdmissibleSequences} implies that
the shift map $\sigma : \Qf \to \Qf$ is equivalent to
a transitive topological Markov chain.
It is well-known that such objects exhibit topological transitivity,
sensitive dependence on initial conditions,
and density of periodic points.
See, for example, Sections~1.9 and~3.2 of~\cite{katok1995introduction}.

On the other hand,
let $\Qv_N = \{ \qv^1,\ldots,\qv^N \}$ be the set provided in part~(b)
of Corollary~\ref{cor:AdmissibleSequences}.
Set $\Qf_N = (\Qv_N)^{\Zset}$.
We consider the bijection
\[
g = (g^i)_{i \in \Zset}: \Qf_N \to \Sigma_N,\quad
g^i(\qv^n) = n, \qquad
\forall i \in \Zset, \ \forall n \in \{1,\ldots,N\}.
\]
Then $\Qf_N$ is a subshift space of $\Qf$.
That is, $\sigma(\Qf_N) = \Qf_N$.
Besides, the diagram
\[
\begin{tikzcd}
\Qf_N \arrow{r}{\sigma_{|\Qf_N}} \arrow[swap]{d}{g} &
\Qf_N \arrow{d}{g} \\
\Sigma_N \arrow[swap]{r}{\sigma_N} &
\Sigma_N
\end{tikzcd}
\]
commutes, so
$h_{\rm top}(\sigma) \ge
 h_{\rm top}(\sigma_{|\Qf_N}) =
 h_{\rm top}(\sigma_N) = \log N$ for all $N \in \Nset$.
This means that $\sigma$ has infinite topological entropy.
\end{proof}

\section{Chaotic motions}\label{sec:ChaoticMotions}

In this section, we detail the construction of a domain accumulating
on the boundary of the phase space on which the dynamics is semiconjugate
to a shift on infinitely many symbols,
thus proving Theorem~\ref{thm:SymbolicDynamicsIntro}; in fact,
we reformulate Theorem~\ref{thm:SymbolicDynamicsIntro}
in the form of Theorem~\ref{thm:SymbolicDynamics} below.
The proof uses the method of \emph{stretching along the paths}
summarised in Section~\ref{sec:Stretching},
the Fundamental Corollary obtained in Section~\ref{sec:Fundamental}
and the shift map described in Section~\ref{sec:Symbols}.

Recall that the quantities $\alpha^\pm_j$,
$\beta^\pm_j = \alpha_j^\pm + 1$ and $\chi_j \ge 2$,
introduced in Lemma~\ref{lem:MinMax}, Lemma~\ref{lem:Fundamental},
and Corollary~\ref{cor:Fundamental} satisfy
hypotheses {\rm ({\bf A})} and {\rm ({\bf B})},
and moreover we assume hypothesis {\rm ({\bf X})},
so Proposition~\ref{prop:ShiftMapProp} holds.

We now introduce some notation that is convenient for the statements
and proofs in this section.

\begin{definition}\label{def:PartialSums}
The \emph{partial sums} $s^i,s^i_j:\Qf \to \Zset$ for
$i \in \Zset$ and $j=1,\ldots,k$ are defined by
\[
s^i(\qf) = 
\begin{cases}
\displaystyle
\phantom{-}
 \sum_{m=0}^{i-1} \sum_{j=1}^k |q^m_j|, & \mbox{ for $i \geq 0$}, \\
\displaystyle
-\sum_{m=i}^{-1} \sum_{j=1}^k |q^m_j|,  & \mbox{ for $i < 0$},
\end{cases}
\qquad \qquad
s^i_j(\qf) = s^i(\qf) + \sum_{m=1}^{j-1} |q^i_m|.
\]
The partial sums $s^i,s^i_j:\Qf^+ \to \Nset_0$ are analogously defined
for $i \ge 0$ and $j=1,\ldots,k$.
\end{definition}

The following proposition gives the relationship between
some types of admissible sequences
(two-sided: $\qf \in \Qf$, one-sided: $\qf \in \Qf^+$
 and periodic: $\qf \in \Pf$)
and orbits of $f$ with prescribed itineraries in the set of
pairwise disjoint cells $\mathcal{K}^\varsigma_{j,n}$
(introduced in Corollary \ref{cor:Fundamental})
with $j = 1,\ldots,k$, $n \ge \chi_j$ and $\varsigma \in \{-,+\}$.
It is the key step in obtaining chaotic properties.

\begin{proposition}\label{prop:Chaos}
We have the following three versions.
\begin{itemize}
\item[{\rm ({\bf T})}]
If $\qf \in \Qf$,
then there is $x \in \mathcal{D}_1$ such that
\begin{equation}
\label{eq:PrescribedItinerary}
f^{s^i_j(\qf)}(x) \in \mathcal{K}^{\Sign(q^i_j)}_{j,|q^i_j|}, \qquad
\forall i \in \Zset, \  \forall j = 1,\ldots,k.
\end{equation}
\item[{\rm ({\bf O})}]
If $\qf \in \Qf^+$,
then there is a path $\gamma \subset \mathcal{D}_1$ such that:
\begin{enumerate}[i)]
\item
$f^{s^i_j(\qf)}(\gamma) \subset \mathcal{K}^{\Sign(q^i_j)}_{j,|q^i_j|}$
for all $i \ge 0$ and $j = 1,\ldots,k$; and
\item
$\gamma$ is horizontal in $\mathcal{D}_1$ (that is, $\gamma$ connects
the left side $\mathcal{L}_1$ with the right side $\mathcal{L}_1^1$).
\end{enumerate}
\item[{\rm ({\bf P})}]
If $\qf \in \Pf$ has period $p$,
then there is a point $x \in \mathcal{D}_1$ such that:
\begin{enumerate}[i)]
\item
$f^{s^i_j(\qf)}(x) \in \mathcal{K}^{\Sign(q^i_j)}_{j,|q^i_j|}$
for all $i \in \Zset$ and $j = 1,\ldots,k$; and
\item
$f^{s^p(\qf)}(x) = x$, so $x$ is a $(p,s^p(q))$-periodic point of $f$
with period $s^p(\qf)$ and rotation number $p/s^p(\qf)$.
\end{enumerate}
\end{itemize}
All these billiard orbits are contained in the generic sliding
set $\mathcal{S}_0$.
In particular, they have no points in the extended singularity set
$\mathcal{L} =
 \bigcup_j \left( \mathcal{L}_j \cup \mathcal{L}_j^{1/2} \cup
                  \mathcal{L}_j^1 \right)$.
Obviously,
these claims only hold for forward orbits in version~{\rm ({\bf O})}.
\end{proposition}

\begin{proof}
It is a direct consequence of Theorem~\ref{thm:PapiniZanolin},
Corollary~\ref{cor:Fundamental},
the definitions of admissible symbols and admissible sequences,
and the definition of rotation number.
\end{proof}

To adapt the language
of~\cite{papini2004fixed, papini2004periodic, pireddu2009fixed}
to our setting, one could say that Proposition~\ref{prop:Chaos}
implies that the billiard map
\emph{induces chaotic dynamics on infinitely many symbols}.

\begin{remark}
\label{rem:PartialSums}
The partial sum $s^i(\qf)$, with $i \ge 0$, introduced in
Definition~\ref{def:PartialSums} counts the number of impacts that any of
its corresponding sliding billiard trajectories have after the first
$i$ turns around $\Gamma$.
Analogously, $s^i_j(\qf)$, with $i \ge 0$, adds to the previous count
the number of impacts in the first $j-1$ arcs at the $(i+1)$-th turn.
There is no ambiguity in these counts, because generic sliding billiard
trajectories have no impacts on the set of nodes $\Gamma_\star$,
see Remark~\ref{rem:GenericTrajectories}.
The partial sums with $i < 0$ store information about the
backward orbit.
\end{remark}

Let us introduce four subsets of the first fundamental
domain that will be invariant under a return map $F$ yet to be defined.
First, we consider the \emph{fundamental generic sliding set}
\[
\mathcal{R} =
\mathcal{S}_0 \cap \mathcal{D}_1 \subset \Interior \mathcal{D}_1.
\]
Any $f$-orbit that begins in $\mathcal{R}$ returns to $\mathcal{R}$
after a finite number of iterations of the billiard map $f$.
Let $\tau:\mathcal{R} \to \Nset$ be the \emph{return time} defined as
$\tau(x) = \min \{ n \in \Nset : f^n(x) \in \mathcal{R} \}$.
Then $F:\mathcal{R} \to \mathcal{R}$, $F(x) = f^{\tau(x)}(x)$,
is the promised \emph{return map}.
The return map $F:\mathcal{R} \to \mathcal{R}$ is a homeomorphism
since the billiard map $f:\mathcal{M} \to \mathcal{M}$ is a homeomorphism
and $\mathcal{R}$ is contained in the interior of the
fundamental set $\mathcal{D}_1$.

Next, we define the sets
\begin{align*}
\mathcal{I} &=
\big\{
x \in \mathcal{M} :
\mbox{$\exists \qf \in \Qf$ such that the prescribed
      itinerary~\eqref{eq:PrescribedItinerary} takes place}
\big\}, \\
\mathcal{P} &=
\big\{
x \in \mathcal{I} :
\mbox{$\exists p \in \Nset$ such that $F^p(x) = x$}
\big\}
\end{align*}
and the map $h:\mathcal{I} \to \Qf$, $h(x) = \qf$,
where $\qf$ is the unique admissible sequence such that
the prescribed itinerary~\eqref{eq:PrescribedItinerary} takes place.
It is well-defined because cells $\mathcal{K}_{j,n}^\varsigma$
are pairwise disjoint.
This is the topological semiconjugacy we were looking for.
Clearly,
\begin{equation}
\label{eq:ReturnTime}
\tau(x) = s^1(\qf) = |q^0_1| + \cdots + |q^0_k|,
\qquad \forall x \in \mathcal{I}, \ \qf = h(x)
\end{equation}
where $\tau(x)$ is the return time and the partial sum $s^1(\qf)$
counts the number of impacts after the first turn around $\Gamma$
of the billiard orbit starting at $x$.

\begin{theorem}\label{thm:SymbolicDynamics}
The sets $\mathcal{P}$, $\mathcal{J} := \overline{\mathcal{P}}$,
$\mathcal{I}$ and $\mathcal{R}$ are $F$-invariant:
\[
F(\mathcal{P}) = \mathcal{P}, \qquad
F(\mathcal{J}) = \mathcal{J}, \qquad
F(\mathcal{I}) \subset \mathcal{I}, \qquad
F(\mathcal{R}) = \mathcal{R}.
\]
Besides, $\emptyset \neq \mathcal{P} \varsubsetneq
 \mathcal{J} \subset \mathcal{I} \subset \mathcal{R}$.
The maps $h:\mathcal{I} \to \Qf$,
$h_{|\mathcal{J}}:\mathcal{J} \to \Qf$ and
$h_{|\mathcal{P}}:\mathcal{P} \to \Pf$ are continuous surjections,
and the three diagrams
\begin{equation}\label{eq:Diagrams}
\begin{tikzcd}
\mathcal{I} \arrow{r}{F_{|\mathcal{I}}} \arrow[swap]{d}{h} &
\mathcal{I} \arrow{d}{h} \\
\Qf \arrow[swap]{r}{\sigma} &
\Qf
\end{tikzcd}
\qquad \qquad
\begin{tikzcd}
\mathcal{J} \arrow{r}{F_{|\mathcal{J}}} \arrow[swap]{d}{h_{|\mathcal{J}}} &
\mathcal{J} \arrow{d}{h_{|\mathcal{J}}} \\
\Qf \arrow[swap]{r}{\sigma} &
\Qf
\end{tikzcd}
\qquad \qquad
\begin{tikzcd}
\mathcal{P} \arrow{r}{F_{|\mathcal{P}}} \arrow[swap]{d}{h_{|\mathcal{P}}} &
\mathcal{P} \arrow{d}{h_{|\mathcal{P}}} \\
\Pf \arrow[swap]{r}{\sigma_{| \Pf}} &
\Pf
\end{tikzcd}
\end{equation}
commute.
Periodic points of $F_{|\mathcal{J}}$ are dense in $\mathcal{J}$.
Given any $\qf \in \Pf$ with period $p$,
there is at least one
$x \in (h_{|\mathcal{P}})^{-1}(\qf) \in \mathcal{P}$ such that
$f^{s^p(\qf)}(x) = F^p(x) = x$.
\end{theorem}

\begin{proof}
Properties $F(\mathcal{P}) = \mathcal{P}$,
$F(\mathcal{R}) = \mathcal{R}$
and $\mathcal{P} \subset \mathcal{I}$ are trivial,
by construction.
Inclusion $\mathcal{I} \subset \mathcal{R}$ follows from
the definitions of both sets and
property~(b) of Corollary~\ref{cor:Fundamental}.

Let us prove that $h:\mathcal{I} \to \Qf$ is continuous and surjective.
Surjectivity follows directly from version~({\bf T}) of
Proposition~\ref{prop:Chaos}.
Choose any $x \in \mathcal{I}$ and $\epsilon > 0$.
Choose $l \in \Nset$ such that $\sum_{|i|> l} 2^{-|i|} < \epsilon$.
Let $\qf = (\qv^i)_{i \in \Zset} = h(x)$
with $\qv^i = (q_1^i,\ldots,q_k^i)$.
Using that the compact sets $\mathcal{K}_{j,n}^\varsigma$ are mutually disjoint,
$F$ is a homeomorphism and condition~\eqref{eq:PrescribedItinerary},
we can find $\delta_j^i > 0$ for each $|i| \le l$ and $j=1,\ldots,k$
such that
\[
f^{s^i_j(\qf)} \big( \mathcal{B}_{\delta_j^i}(x) \cap \mathcal{I} \big)
\subset \mathcal{K}^{\Sign(q^i_j)}_{j,|q^i_j|}, \qquad
\forall |i| \le l, \  \forall j = 1,\ldots,k.
\]
Here, $\mathcal{B}_\delta(x)$ is the disc of radius $\delta$ centered at $x$.
If $\delta = \min \{ \delta_j^i : |i|\le l, \ j=1,\ldots,k \}$,
then
\[
\mbox{$d(x,y) < \delta$ and $\pf = (\pv^i)_{i \in \Zset} = h (y)
      \Longrightarrow \pv^i = \qv^i$ for each $|i| \le l$}.
\]
Therefore,
\[
d_\Qf \big( h (y), h (x) \big) =
d_{\Qf}(\pf,\qf) =
\sum_{|i| > l} \frac{1}{2^{|i|}}
\frac{|\pv^i - \qv^i|}{1 + |\pv^i - \qv^i|} <
\sum_{|i| > l} \frac{1}{2^{|i|}} < \epsilon,
\]
which implies that $h:\mathcal{I} \to \Qf$ is continuous.

Next, we prove simultaneously that $F(\mathcal{I}) \subset \mathcal{I}$
and that $\sigma \circ h = h \circ F_{| \mathcal{I}}$.
Let $x \in \mathcal{I}$, $y = F(x) \in \mathcal{R}$,
$\qf = (\qv^i)_{i \in \Zset} = h(x)$ with $\qv^i = (q_1^i,\ldots,q_k^i)$,
and $\pf = (\pv^i)_{i \in \Zset} = \sigma(\qf) \in \Qf$
with $\pv^i = (p_1^i,\ldots,p_k^i)$,
so $\pv^i = \qv^{i+1}$ and $p_j^i = q_j^{i+1}$.
The prescribed itinerary~\eqref{eq:PrescribedItinerary} and
relation~\eqref{eq:ReturnTime} imply that
\[
f^{s^i_j(\pf)}(y) =
f^{s^i_j(\sigma(\qf))}(F(x)) =
f^{s^{i+1}_j(\qf) - s^1(\qf)} \big( f^{s^1(\qf)}(x) \big) =
f^{s^{i+1}_j(\qf)}(x) \in
\mathcal{K}^{\Sign(q^{i+1}_j)}_{j,|q^{i+1}_j|} =
\mathcal{K}^{\Sign(p^i_j)}_{j,|p^i_j|}
\]
for all $i \in \Zset$ and $j = 1,\ldots,k$,
so $\sigma(h(x)) = \sigma(\qf) = \pf = h(y) = \sigma(F(x))$
and $F(x) = y \in h^{-1}(\pf) \subset \mathcal{I}$
for all $x \in \mathcal{I}$, as we wanted to prove.
Hence, the first diagram in~\eqref{eq:Diagrams}
defines a topological semiconjugacy.

Let us check that $\mathcal{J} := \overline{\mathcal{P}} \subset \mathcal{I}$
and $F(\mathcal{J})=\mathcal{J}$.
We have $\mathcal{J} \subset \mathcal{I}$,
because $\mathcal{I} = h^{-1} (\Qf)$ is closed
(continuous preimage of a closed set).
Besides, on one hand we have
$\mathcal{P} = F(\mathcal{P}) \subset F(\overline{\mathcal{P}}) = F (\J)$
implying that $\J = \overline{\mathcal{P}} \subset \overline{F(\J)} = F(\J)$,
while on the other we have
$F(\J) = F(\overline{\mathcal{P}}) \subset
 \overline{F(\mathcal{P})} = \overline{\mathcal{P}} = \J$.
To establish that the second diagram in~\eqref{eq:Diagrams}
is still a topological semiconjugacy,
we must prove that $h(\mathcal{J}) = \Qf$.
We clearly have $h(\mathcal{J}) \subset \Qf$ since
$\mathcal{J} \subset \mathcal{I}$,
and since $h_{|\mathcal{I}}$ is a semiconjugacy.
Meanwhile, since $\mathcal{J}$ is compact
(closed by definition and contained in the bounded set $\mathcal{D}_1$)
so too is $h(\mathcal{J})$;
moreover $h(\mathcal{J})$ contains $h(\mathcal{P}) = \Pf$
which is dense in $\Qf$ by Proposition~\ref{prop:ShiftMapProp},
and so we obtain $\Qf = \overline{\Pf} \subset h(\mathcal{J})$.
Therefore $h(\mathcal{J}) = \Qf$.

To complete the proof of the theorem,
notice that periodic points of $F_{|\mathcal{J}}$
are dense in $\mathcal{J}$ by construction
and the last claim of the Theorem~\ref{thm:SymbolicDynamics} follows from
version~({\bf P}) of Proposition~\ref{prop:Chaos},
which also implies $\mathcal{P} \neq \emptyset$.
\end{proof}

Proposition~\ref{prop:ShiftMapProp} and Theorem~\ref{thm:SymbolicDynamics}
imply Theorem~\ref{thm:SymbolicDynamicsIntro} as stated in the introduction
and it is the first step in proving Theorems~\ref{thm:BoundaryIntro},
\ref{thm:PeriodicIntro} and~\ref{thm:AsymptoticConstantIntro}.
For instance,
upon combining Theorem~\ref{thm:SymbolicDynamics} with the topological
transitivity of the shift map guaranteed by
Proposition~\ref{prop:ShiftMapProp},
we already obtain the existence of trajectories approaching the boundary
asymptotically.
It remains to determine the \emph{optimal} rate of diffusion.
This is done in Section~\ref{sec:OptimalSpeed} by analysing the
sequences $\qf = (\qv^i)_{i \ge 0} \in \Qf^+$ for which
$s^i(\qf)$ increases in the \emph{fastest} possible way as $i \to +\infty$.
Lemma~\ref{lem:AdmissibleSymbols} plays a role in that analysis.

We end this section with three useful corollaries.
First, we prove Corollary~\ref{cor:PossibleBehaviors}
on final sliding motions.

\begin{proof}[Proof of Corollary~\ref{cor:PossibleBehaviors}]
The clockwise case is a by-product of the
counter-clockwise one, because if we concatenate the arcs
$\Gamma_1,\ldots,\Gamma_k$ of the original circular polygon $\Gamma$
in the reverse order $\Gamma_k,\ldots,\Gamma_1$, then we obtain
the reversed circular polygon $\Gamma'$ with the property that
counter-clockwise sliding billiard trajectories in $\Gamma'$ are
in 1-to-1 correspondence with clockwise sliding
billiard trajectories in $\Gamma$.
Thus, it suffices to consider the counter-clockwise case.

Symbols $\qv = (q_1,\ldots,q_k) \in \Qv \subset \Zset^k$
keep track of the proximity of the fundamental quadrilaterals
$\mathcal{Q}_{j,|q_j|}$ to the inferior boundary of $\mathcal{M}$.
That is, the larger the absolute value $|q_j|$,
the smaller the angle of reflection $\theta$ for any
$x = (\varphi,\theta) \in \mathcal{Q}_{j,|q_j|}$.
For this reason, by construction,
if one considers a bounded sequence in $\Qf$,
(respectively, a sequence $\qf \in \Qf$ such that
 $\chi_j \le \min_{i \in \Zset} q^i_j <
  \lim \sup_{|i| \to +\infty} q^i_j = +\infty$ for all $j=1,\ldots,k$)
(respectively, a sequence $\qf \in \Qf$ such that
 $\lim_{|i| \to +\infty} q^i_j = +\infty$ for all $j=1,\ldots,k$),
the corresponding sliding orbit in $\mathcal{J} \subset \mathcal{M}$
belongs to $\mathcal{B}_0^- \cap \mathcal{B}_0^+$
(respectively, $\mathcal{O}_0^- \cap \mathcal{O}_0^+$)
(respectively, $\mathcal{A}_0^- \cap \mathcal{A}_0^+$).
By considering two-sided sequences $\qf \in \Qf$ which have
different behaviors at each side,
one can construct trajectories which belong
to $\mathcal{X}_0^- \cap \mathcal{Y}_0^+ \neq \emptyset$
for any prescribed choice
$\mathcal{X},\mathcal{Y} = \mathcal{B}, \mathcal{O}, \mathcal{A}$
such that $\mathcal{X} \neq \mathcal{Y}$.
The existence of all these sequences comes from
part~(\ref{item:XiSets}) of Lemma~\ref{lem:AdmissibleSymbols},
since we can control the size of $|q^i_j|$ just from
the size of $|q^i_1|$.
\end{proof}

\begin{corollary}
\label{cor:InfiniteEntropy}
With the notation as in Theorem~\ref{thm:SymbolicDynamics},
the following properties are satisfied. 
\begin{enumerate}[(a)]
\item
The return map $F|_{\mathcal{J}}$ has infinite topological entropy. 
\item
There is a compact $F$-invariant set $\mathcal{K} \subset \mathcal{J}$
such that $F|_{\mathcal{K}}$ is topologically semiconjugate to
the shift $\sigma : \Qf \to \Qf$ via the map $h|_{\mathcal{K}}$
in the sense of~\eqref{eq:Diagrams}; it is topologically transitive;
and it has sensitive dependence on initial conditions.
\end{enumerate}
\end{corollary}

\begin{proof}
\begin{enumerate}[(a)]
\item
It follows from the fact that $\sigma:\Qf \to \Qf$ has
infinite topological entropy and it is a topological
factor of $F:\mathcal{J} \to \mathcal{J}$.
\item
It is a direct consequence of our Theorem~\ref{thm:SymbolicDynamics}
and a theorem of Auslander and Yorke.
See~\cite[Item~(v) of Theorem~2.1.6]{pireddu2009fixed} for details.
\qedhere
\end{enumerate}
\end{proof}

Given any integers $1 \le p < q$,
let $\Pi(p,q)$ be the set of $(p,q)$-periodic billiard trajectories
in the circular $k$-gon $\Gamma$.
That is, the set of periodic trajectories that close
after $p$ turns around $\Gamma$ and $q$ impacts in $\Gamma$,
so they have rotation number $p/q$.
The symbol~$\#$ denotes the \emph{cardinality} of a set.
Let $2^{\Rset^{n+1}}$ be the power set of $\Rset^{n+1}$.
Let $G_q: 2^{\Rset^{n+1}}\to \Nset_0$ be the function
\[
G_q(K) =
\# \left\{
\xv = (x_1,\ldots,x_{n+1}) \in K \cap \Zset^{n+1} :
x_1 + \cdots + x_{n+1} = q
\right\}
\]
that counts the integer points in any subset
$K \subset \Rset^{n+1}$ whose coordinates sum $q \in \Nset$.

\begin{corollary}
\label{cor:ConvexPolytopes}
Let $\alpha^\pm_j$, $\beta^\pm_j = \alpha^\pm_j + 1$
and $\chi_j$ be the quantities defined in
Lemma~\ref{lem:MinMax} and Corollary~\ref{cor:Fundamental}.
If $p,q \in \Nset$ with $1 \leq p < q$, then
\begin{equation}
\label{eq:LowerBoundPpq}
\# \Pi(p,q) \ge 2^{n+1} G_q\big( P^{(p)} \big),
\end{equation}
where $n+1 = k p$ and
\begin{equation}
\label{eq:UnboundedPolytope}
P^{(p)} =
\left\{
\xv \in \Rset^{n+1} :
\begin{array}{l}
\alpha^-_j x_j + \beta^-_j < x_{j+1} < \alpha^+_j x_j - \beta^+_j,
\quad \forall j=1,\ldots,n \\
\alpha^-_{n+1} x_{n+1} + \beta^-_{n+1} < x_1 <
\alpha^+_{n+1} x_{n+1} - \beta^+_{n+1}, \\
x_j \ge \chi_j, \quad \forall j=1,\ldots,n+1
\end{array}
\right\}
\end{equation}
is an unbounded convex polytope of $\Rset^{n+1}$. 
\end{corollary}

\begin{proof}
Let $p,q \in \Nset$ such that $1 \le p < q$.
Set $n+1 = kp$.
Let $\Pf_p$ be the set of admissible periodic sequences of period $p$.
We consider the map $\psi_p: \Pf_p \to \Nset^{n+1}$
defined by
\[
\psi_p(\qf) = \xv = (x_1,\ldots,x_{n+1}) =
\left(
|q^0_1|,\ldots,|q^0_k|,
|q^1_1|,\ldots,|q^1_k|,\ldots,
|q^{p-1}_1|,\ldots,|q^{p-1}_k|
\right),
\]
where $\qf = (\qv^i)_{i \in \Zset} \in \Pf_p$
and $\qv^i = (q_1^i,\ldots,q_k^i) \in \Qv$.
Note that $s^p(\qf) = x_1 + \cdots + x_{n+1}$ when $x = \psi_p(\qf)$.
Besides, $\psi_p(\Pf_p) \subset P^{(p)} \cap \Zset^{n+1}$ and
the map $\psi_p: \Pf_p \to P^{(p)} \cap \Zset^{n+1}$
is $2^{n+1}$-to-1 by construction.
Therefore, each point $\xv \in P^{(p)} \cap \Zset^{n+1}$
whose coordinates sum $q$
gives rise to, at least, $2^{n+1}$ different generic sliding
$(p,q)$-periodic billiard trajectories,
see version~({\bf P}) of Proposition~\ref{prop:Chaos}.
\end{proof}

Lower bound~\eqref{eq:LowerBoundPpq} is far from optimal,
since it does not take into account the periodic billiard trajectories
that are not generic or not sliding.
But we think that it captures with great accuracy the growth rate
of $\# \Pi(p,q)$ when $p/q$ is relatively small and $q \to +\infty$.
It will be the first step in proving
Theorem~\ref{thm:PeriodicIntro} in Section~\ref{sec:PeriodicTrajectories}.

\section{Optimal linear speed for asymptotic sliding orbits}
\label{sec:OptimalSpeed}

In this section we establish the existence of uncountably many
\emph{points} in the fundamental domain $\mathcal{D}_1$
that give rise to generic asymptotic sliding billiard trajectories
(that is, those trajectories in the intersection
 $\mathcal{A}_0^- \cap \mathcal{A}_0^+ \subset \mathcal{S}_0$
 described in the introduction)
that approach the boundary asymptotically with optimal
uniform linear speed as $|n| \to +\infty$.
We also look for trajectories just in
$\mathcal{A}_0^+ \subset \mathcal{S}_0$,
in which case we obtain uncountably many \emph{horizontal paths}
(not points) in $\mathcal{D}_1$.
The dynamic feature that distinguishes such trajectories
in that they approach the boundary in the fastest way possible
among all trajectories that give rise to admissible
sequences of symbols.

We believe that the union of all these horizontal paths
(respectively, all these points)
is a Cantor set times an interval
(respectively, the product of two Cantor sets).
However, in order to prove it rigorously,
we would need to prove that our semiconjugacy
$h_{|\mathcal{J}}:\mathcal{J} \to \Qf$, see~\eqref{eq:Diagrams},
is, indeed, a full conjugacy.
Both sets are $F$-invariant and they accumulate on
the first node of the circular polygon.
Obviously, there are similar sets for each one of the other nodes.

The reader must keep in mind the notations listed at the beginning
of Section~\ref{sec:ChaoticMotions},
the estimates in Lemma~\ref{lem:AdmissibleSymbols},
and the interpretation of the partial sums
$s^i,s^i_j:\Qf \to \Nset_0$, with $i \in \Zset$ and $j = 1,\ldots,k$,
presented in Remark~\ref{rem:PartialSums}.

\begin{definition}
\label{def:SignSpaces}
The uncountably infinite \emph{sign spaces} are
\begin{align*}
\Tf^+ &=
\left\{
\tf = (\tv^i)_{i \ge 0} :
\tv^i = (t_1^i,\ldots,t_k^i) \in \{-,+\}^k
\right\}, \\
\Tf &=
\left\{
\tf = (\tv^i)_{i \in \Zset} :
\tv^i = (t_1^i,\ldots,t_k^i) \in \{-,+\}^k
\right\}.
\end{align*}
\end{definition}

To avoid any confusion,
be aware that the dynamical index of the iterates
of asymptotic generic sliding trajectories was called
$n \in \Zset$ in Theorem~\ref{thm:BoundaryIntro},
but it is called $l \in \Zset$ in Theorem~\ref{thm:Boundary} below.

\begin{theorem}\label{thm:Boundary}
There are constants $0 < d_- < d_+$ such that the following
properties hold.
\begin{enumerate}[(a)]
\item
There are pairwise disjoint paths
$\gamma_n^{\tf} \subset \mathcal{K}_{1,n}^{t_1^1} \subset \mathcal{D}_1$
for any $n \ge \chi_1$ and $\tf \in \Tf^+$,
`horizontal' since they connect the left side
$\mathcal{L}_1$ with the right side $\mathcal{L}_1^1$,
such that
\[
\left.
\begin{array}{c}
\Pi_\theta \big(f^l(x)\big) = \Pi_\theta(x),
\quad \forall l = 0,\ldots, n-1 \\
n d_- \Pi_\theta(x) \le l \Pi_\theta \big(f^l(x)\big) \le n d_+ \Pi_\theta(x),
\quad \forall l \ge n
\end{array}
\right\}
\quad \forall x \in \gamma_n^{\tf}, \
\forall n \ge \chi_1, \ \forall \tf \in \Tf^+.
\]
\item
There are pairwise distinct points
$x_n^{\tf} \in \mathcal{K}_{1,n}^{t_1^1} \subset \mathcal{D}_1$
for any $n \ge \chi_1$ and $\tf \in \Tf$ such that
\[
\left.
\begin{array}{c}
\Pi_\theta \big(f^l(x_n^\tf)\big) =
\Pi_\theta \big( x_n^\tf \big),
\quad \forall l = 0,\ldots, n-1 \\
\Pi_\theta \big(f^{l}(x_n^\tf)\big) =
\Pi_\theta\big( f^{-1}(x_n^\tf) \big),
\quad \forall l = -1,\ldots,-m \\
n d_- \Pi_\theta(x_n^{\tf}) \le
|l| \Pi_\theta \big( f^l( x_n^{\tf} ) \big) \le
n d_+ \Pi_\theta( x_n^\tf ),
\quad \forall l \ge n \mbox{ or } l<-m
\end{array}
\right\}
\ \forall n \ge \chi_1, \ \forall \tf \in \Tf,
\]
where $m = -\xi^{-1}_k(n) \in \Nset$.
\end{enumerate}
\end{theorem}

\begin{proof}
\begin{enumerate}[(a)]
\item
Identity $\Pi_\theta \big(f^l(x)\big) = \Pi_\theta(x)$
for all $x \in \mathcal{Q}_{1,n}$ and $l=0,\ldots,n-1$ is trivial,
because these first impacts are all over the first arc $\Gamma_1$,
so the angle of reflection remains constant.
Henceforth, we just deal with the case $l \ge n$.

Fix $n \ge \chi_1$ and $\tf = (\tv^i)_{i \ge 0} \in \Tf^+$ with
$\tv^i = (t_1^i,\ldots,t_k^i)$.
Let $\nf = (\nv^i)_{i \ge 0} \in \big( \Nset^k \big)^{\Nset_0}$
with $\nv^i = (n_1^i,\ldots,n_k^i) \in \Nset^k$ be the sequence given by
\[
n^i_j := \xi^i_j(n) = \max \Xi^i_j(n),
\]
where $\Xi^i_j(n) \subset \Nset$ is the set~\eqref{eq:XiDefinition}.
We view $n^0_1 = n$ as the `starting' value,
since the sequence $\nf$ is completely determined by $n$.
However, we do not make this dependence on $n$ explicit for the sake of brevity.
Let $\rho^0 = n$,
\begin{align*}
\rho^i &=
s^i(\nf) =
\sum_{m=0}^{i-1} \sum_{j=1}^k n^m_j, \quad \forall i > 0,\\
\rho^i_j &=
s^i_j(\nf) = s^i(\nf) + \sum_{m=1}^{j-1} n^i_m, \quad
\forall j \mod k, \ \forall i \ge 0.
\end{align*}
Note that $\rho_1^i = \rho^i$.
We use the convention $\rho_{k+1}^i = \rho^{i+1}$.

There is $\qf = (\qv^i)_{i \ge 0} \in \Qf^+$
with $\qv^i = (q_1^i,\ldots,q_k^i)$ such that $\Sign(\qf) = \tf$
and $|\qf| = \nf$ by definition.
Note that $s^i_j(\qf) = \rho^i_j$ for any $i \ge 0$ and $j=1,\ldots,k$.
Version~({\bf O}) of Proposition~\ref{prop:Chaos} implies that
there is a path $\gamma_n^\tf \in \mathcal{D}_1$,
horizontal in the sense that it connects the left side
$\mathcal{L}_1$ with the right side $\mathcal{L}_1^1$,
such that
\[
f^{\rho^i_j}(x) \subset \mathcal{K}^{t^i_j}_{j,n^i_j},\qquad
\forall x \in \gamma_n^\tf, \quad
\forall i \ge 0, \quad
\forall j = 1,\ldots,k.
\]
In particular, $\gamma_n^{\tf} \subset \mathcal{K}_{1,n}^{t_1^1}$.
The paths $\gamma_n^{\tf}$ are pairwise disjoint,
because the cells $\mathcal{K}_{j,n}^\varsigma$ are.

Fix $x = (\varphi,\theta) \in \gamma_n^\tf$ and $l \ge n$.
Set $(\varphi_l,\theta_l) = f^l (\varphi,\theta)$.
Our goal is to prove that
\begin{equation}
\label{eq:GoalInequalities}
n d_- \le l \theta_l/\theta \le n d_+,
\end{equation}
for some constants $0 < d_- < d_+$ that do no depend on the choices
of the starting value $n \ge \chi_1$,
the sign sequence $\tf \in \Tf^+$,
the point $x \in \gamma_n^\tf$ or the forward iterate $l \ge n$.

Let $i \ge 0$ be the number of complete turns around $\Gamma$ that
this billiard trajectory performs from the $0$-th impact
to the $l$-th impact,
and let $j \in \{1,\ldots, k\}$ be the arc index where
the $l$-th impact lands,
so $\rho^i \le \rho_j^i \leq l < \rho_{j+1}^i \le \rho^{i+1}$.
Set $r = \rho_j^i$.
Then
$(\varphi_r,\theta_r) \in
 \mathcal{K}_{j,n_j^i}^{t_j^i} \subset
 \mathcal{Q}_{j,n_j^i}$, and so,
since the orbit segment
\(
(\varphi_r,\theta_r), (\varphi_{r+1},\theta_{r+1}), \ldots,
(\varphi_{l-1},\theta_{l-1}), (\varphi_l,\theta_l)
\)
remains in the circular arc $\Gamma_j$ without crossing
the singularity segment $\mathcal{L}_{j+1}$,
we have
\[
\frac{\delta_j}{2 n_j^i + 2} =
\min_{y \in Q_{j,n_j^i}} \Pi_{\theta} (y) \le
\theta_l =
\theta_r \leq \max_{y \in Q_{j,n_j^i}} \Pi_{\theta} (y) =
\frac{\delta_j}{2 n_j^i - 2},
\]
see Lemma~\ref{lem:MinMax}.
From
$x = (\varphi,\theta) \in \gamma_n^\tf \subset
 \mathcal{K}_{1,n}^{t_1^1} \subset \mathcal{Q}_{1,n}$,
we also have
\[
\frac{\delta_1}{2 n + 2} =
\min_{y \in Q_{1,n}} \Pi_{\theta} (y) \le
\theta \le \max_{y \in Q_{1,n}} \Pi_{\theta} (y) =
\frac{\delta_1}{2 n - 2}.
\]
By combining the last three displayed sets of inequalities,
we get that
\begin{equation}
\label{eq:ToTheGoal}
\frac{\delta_j}{\delta_1} \frac{n-1}{n_j^i + 1} \rho^i \le
l \theta_l/\theta \le
\frac{\delta_j}{\delta_1} \frac{n+1}{n_j^i - 1} \rho^{i+1}.
\end{equation}
Let $\nu' < \lambda'$ be the positive constants that appear in
Part~(\ref{item:AdmissibleSymbols3}) of Lemma~\ref{lem:AdmissibleSymbols},
so
\begin{equation}
\label{eq:ToTheGoal2}
\nu' n^i_j \le \rho^i \le \rho^{i+1} \le \lambda' n^i_j.
\end{equation}
Bound~\eqref{eq:GoalInequalities} follows
from~\eqref{eq:ToTheGoal} and~\eqref{eq:ToTheGoal2}
if we take
\begin{align*}
d_+
&=
\frac{\lambda'}{\delta_1} \max \{ \delta_1,\ldots,\delta_k \}
\max \left\{ \frac{(n+1) n^i_j}{(n^i_j -1) n} :
             n \ge \chi_1, \ n^i_j \ge \chi_j, \ j=1,\ldots,k \right\} \\
&=
\frac{\lambda'}{\delta_1} \max \{ \delta_1,\ldots,\delta_k \}
\max \left\{ \frac{(\chi_1 + 1) \chi_j}{(\chi_j - 1) \chi_1} :
             j=1,\ldots,k \right\}, \\
d_- &=
\frac{\nu'}{\delta_1} \min \{ \delta_1,\ldots,\delta_k \}
\min \left\{ \frac{(n-1) n^i_j}{(n^i_j +1) n} :
             n \ge \chi_1, \ n^i_j \ge \chi_j, \ j=1,\ldots,k \right\} \\
&=
\frac{\nu'}{\delta_1} \min \{ \delta_1,\ldots,\delta_k \}
\min \left\{ \frac{(\chi_1 - 1) \chi_j}{(\chi_j + 1) \chi_1} :
             j=1,\ldots,k \right\}.
\end{align*}

\item
The proof is similar, but using
version~({\bf T}) of Proposition~\ref{prop:Chaos}.
We omit the details.
We just stress that if $x \in \mathcal{Q}_{1,n}$,
$h(x) = \qf = (\qv^i)_{i \in \Zset}$ and $m := |q^{-1}_k| = -\xi^{-1}_k(n)$,
then the first $m$ backward iterates of the point $x$
impact on the last arc $\Gamma_k$.
\qedhere
\end{enumerate}
\end{proof}

The constants $0 < a < b$ in Theorem~\ref{thm:BoundaryIntro} are directly
related to the constants $0 < d_- < d_+$ in Theorem~\ref{thm:Boundary}.
To be precise, we can take
\begin{align*}
a &=
\min_{n \ge \chi_1}
\frac{1}{n d_+ \max_{x \in \mathcal{Q}_{1,n}} \Pi_\theta(x)} =
\min_{n \ge \chi_1} \frac{2n-2}{n d_+ \delta_1} =
\frac{2\chi_1 - 2}{\chi_1 \delta_1 d_+} > 0,\\
b &=
\max_{n \ge \chi_1}
\frac{1}{n d_- \min_{x \in \mathcal{Q}_{1,n}} \Pi_\theta(x)} =
\max_{n \ge \chi_1} \frac{2n+2}{n d_- \delta_1} =
\frac{2 \chi_1 + 2}{\chi_1 \delta_1 d_-} > a.
\end{align*}

The sequences
$\left(\cup_{\tf \in \Tf^+} \gamma^\tf_n \right)_{n \ge \chi_1}$ and
$\left( \cup_{\tf \in \Tf^+} x^\tf_n \right)_{n \ge \chi_1}$
are composed by uncountable sets of horizontal paths and points,
respectively, with the desired optimal uniform linear speed.
The index $n \ge \chi_1$ of the sequence counts the number of
impacts that the corresponding billiard trajectories have in the
first arc $\Gamma_1$ at the beginning.
The fundamental quadrilaterals $\mathcal{Q}_{1,n}$ tend
to the first node as $n \to +\infty$:
$\lim_{n \to +\infty} \mathcal{Q}_{1,n} = (a_1,0)$,
so we conclude that both sequences accumulate on that node
when $n \to +\infty$.

Let us justify the optimality of linear speed.

\begin{proposition}
\label{prop:OptimalSpeed}
There is no billiard trajectory in a circular polygon such that
\[
\lim_{n \to +\infty} n \theta_n = 0.
\]
\end{proposition}

\begin{proof}
We have already proved that all asymptotic billiard trajectories
that give rise to admissible sequences of symbols satisfy an upper bound
of the form
\[
1/\theta_n \le b |n|, \qquad \forall |n| \gg 1
\]
for some uniform constant $b > 0$
The problem is that there could be some \emph{slightly faster}
billiard trajectories that \emph{do not} give rise to admissible sequences.

For instance,
if we look at the fundamental quadrilateral $\mathcal{Q}_{j,n}$
displayed in Figure~\ref{fig:FundamentalQuadrilaterals} and its
image $f^n(Q_{j,n})$ displayed in Figure~\ref{fig:FundamentalLemma},
we see that all points $x \in \mathcal{Q}_{j,n}$ close enough to
$\mathcal{L}_{j+1}^{-n+1/2}$ have an image $f^n(x)$
below the lowest admissible fundamental quadrilateral
$\mathcal{Q}_{j+1,m}$ with $m = \max \{ n' \ge \chi_1 : (n,n') \in \Xi_j\}$.
Therefore, since we only deal with admissible sequences of symbols,
we have `lost' the lower non-admissible portion of the red quadrilateral
with parabolic shape in Figure~\ref{fig:FundamentalLemma}.

However, part~(\ref{item:Balint_etal}) of Lemma~\ref{lem:BilliardProperties}
shows that, once we fix any
$\epsilon \in \big( 0,\min\{ \mu_1,\ldots,\mu_k \} \big)$,
we have
\[
\Pi_\theta \big( f^n(x) \big) \ge (\mu_j - \epsilon) \Pi_\theta(x),
\qquad \forall x \in \mathcal{Q}_{j,n}, \quad
\forall j \mod k,\quad
\forall n \gg 1,
\]
provided $\mu_j < 1$,
so these lower non-admissible portions can not be much lower than
the ones that we have already taken into account.
This means that if we repeat the computations of all constants
that appear along our proofs,
but replacing $\mu_j$ with $\mu_j - \epsilon$ provided $\mu_j < 1$,
then we obtain a new uniform constant $\hat{b} \in (b,+\infty)$ such that
\[
1/\theta_n \le \hat{b} |n|, \qquad \forall |n| \gg 1
\]
for all billiard trajectories, with no exceptions.
\end{proof}

\section{On the number of periodic trajectories}
\label{sec:PeriodicTrajectories}

In this section,
we construct exponentially large (in $q$) lower bounds on the number
of periodic trajectories of period $q$,
thus proving Theorem~\ref{thm:PeriodicIntro}.
The strategy of the proof is to use the lower
bound~\eqref{eq:LowerBoundPpq}
provided in Corollary~\ref{cor:ConvexPolytopes}.
In Section~\ref{ssec:StatementResults} we state the main results.
Then Section~\ref{ssec:ProofPolynomial} contains the proof of
a general polynomial lower bound from which we deduce the
asymptotic exponential lower bound in Section~\ref{ssec:ProofAsymptotic}.

\subsection{Statement of the results}
\label{ssec:StatementResults}

Recall that factors $\alpha^\pm_j$ from the fundamental lemma
satisfy hypotheses ({\bf A}) and~({\bf B}) (see Lemma~\ref{lem:Factors}).
Throughout this section we do not increase the size of $\chi_j$.
Indeed,
we no longer need the estimates contained in
Lemma~\ref{lem:AdmissibleSymbols},
although we still need those contained in Lemma~\ref{lem:MinMax}.
So, we may consider significantly smaller integers $\chi_j$.
For instance, we may take~\eqref{eq:chi_mulessthan1} when $\mu_j < 1$.
Recall also the unbounded convex polytope $P^{(p)} \subset \Rset^{n+1}$ 
introduced in Corollary~\ref{cor:ConvexPolytopes},
with $p \in \Nset$ and $n+1 = kp$.

Let $\Pi(p,q)$ be the set of $(p,q)$-periodic billiard
trajectories for any $1 \le p < q$.
Let $\Pi(q) = \cup_{1 \le p < q} \Pi(p,q)$
be the set of all periodic trajectories with period $q$.
We state three lower bounds on the number of periodic billiard
trajectories in the theorem below.
First, a polynomial general lower bound of $\# \Pi(p,q)$.
Second, an exponential asymptotic lower bound of
$\# \Pi(q)$ as $q \to +\infty$.
Third,
a polynomial asymptotic lower bound of $\# \Pi(p,q)$
as $q \to +\infty$, for any fixed $p \in \Nset$.
The symbol~$\#$ denotes the \emph{cardinality} of a set.
The \emph{floor} and \emph{ceiling} functions are denoted with symbols
$\lfloor \cdot \rfloor$ and $\lceil \cdot \rceil$.

\begin{theorem}
\label{thm:LowerBound}
If $\Gamma$ is a circular $k$-gon and $p \in \Nset$,
there are constants
$a_\star, b_\star, h_\star, x_\star, M_\star, c_\star(p) > 0$
such that the following three lower bounds hold:
\begin{enumerate}[(a)]
\item
\(
\displaystyle
\# \Pi(p,q) \ge
2\left( a_\star q/kp - b_\star \right)^{kp-1}/kp
\)
for all $q > b_\star k p/a_\star$.

\item
$\# \Pi(q) \ge
 \# \Pi(p,q) \ge
 M_\star \rme^{h_\star q}/q$
when $p = \lfloor x_\star q/k\rfloor$ and $q \to +\infty$.

\item
$\# \Pi(p,q) \ge c_\star q^{kp - 1} + \Order(q^{kp - 2})$
as $q \to +\infty$ for any fixed $p \in \Nset$.
\end{enumerate}
\end{theorem}

\begin{remark}
\label{rem:Constants}
We give explicit expressions for all involved constants.
We can take
\begin{align*}
a_\star &=
4\min
\left\{
\frac{(\alpha_1 - \alpha^-_1) A_1}{(1 + \alpha^-_1)A},
\frac{(\alpha^+_1 - \alpha_1) A_1}{(1 + \alpha^+_1)A},
\ldots,
\frac{(\alpha_k - \alpha^-_k) A_k}{(1 + \alpha^-_k)A},
\frac{(\alpha^+_k - \alpha_k) A_k}{(1 + \alpha^+_k)A}
\right\}, \\
b_\star &= 6 + 4 \max \{ \chi_1,\ldots,\chi_k \}, \\
h_\star &= a_\star W_0(b_\star/\rme)/b_\star,\\
x_\star &= a_\star W_0(b_\star/\rme)/((1 + W_0(b_\star/\rme))b_\star), \\
M_\star &= 2(a_\star / x_\star - b_\star)^{-k-1}/x_\star,\\
c_\star(p) &= 2 (a_\star)^{kp-1}/(kp)^{kp},
\end{align*}
where
$\alpha_j = \sqrt{\alpha^-_j \alpha^+_j}$,
$A_j = \prod_{i=1}^{j-1} \alpha_i$,
$A = \frac{1}{k} \sum_{j=1}^k A_j$ and
$W_0:[-1/\rme,+\infty) \to [-1,+\infty)$ is the real part of
the principal branch of the Lambert $W$ function.
Note that $A_{k+1} = A_1 = 1$ by hypothesis ({\bf B}).
Therefore, $A_j = A_{j \mod k}$.
Function~$W_0(x)$ is implicitly determined by relations
$W_0( x \rme^x) = x$ for all $x \ge -1$ and
$W_0(x) \rme^{W_0(x)} = x$ for all $x \ge -1/\rme$,
see~\cite{Corless_etal1996}.
\end{remark}

The exponent $h_\star = a_\star W_0(b_\star/\rme)/b_\star > 0$
in the exponentially small lower bound
is the most important constant in Theorem~\ref{thm:LowerBound}.
It is `proportional' to $a_\star$.
We note that there is $i \in \{1,\ldots,k\}$ such that
$A = \frac{1}{k} \sum_{j=1}^k A_j \ge A_i$,
so $a_\star < 4$.
The exponent $h_\star$ also depends on $b_\star$ through
the Lambert function $W_0$, but we believe that this
is due to the techniques used and does not come from
any fundamental characteristic of the problem.
It is known that $W_0(x)/x$ is decreasing for $x > 0$,
$\lim_{x \to 0^+} W_0(x)/x = W'(0) = 1$ and
$W_0(x)/x$ is asymptotic to $\frac{\log x}{x}$ as $x \to +\infty$.
Hence, $h_\star < a_\star/\rme < 4/\rme$ for any $\Gamma$.
We conclude that the expression
$h_\star = a_\star W_0(b_\star/\rme)/b_\star$ is, by no means,
optimal.
If $\Gamma$ tends to a circle,
then $\alpha_j^-$ and $\alpha_j^+$ become closer and closer,
so $h_\star$ tends to zero.

The optimal constant $c_\star(p)$ that satisfies the third bound
can be way bigger than the crude value
$c_\star(p) = 2 (a_\star)^{kp-1}/(kp)^{kp}$
obtained directly from the first bound.
We give a way to compute the optimal value
$c_\star(p) = 2^{kp} \lim_{q \to +\infty} q^{1-kp} G_q\big( P^{(p)} \big)$
in Proposition~\ref{prop:Optimalc},
whose proof is postponed to Appendix~\ref{app:Optimalc}.
If $P$ is a Jordan measurable set of $\Rset^n$,
let $\Volume(P)$ be its \emph{$n$-dimensional volume}.
Let
$H_{n+1} = \left\{ \xv \in \Rset^{n+1} : x_1 + \cdots + x_{n+1} = 1 \right\}$.
Let $\Pi_{n+1}: \Rset^{n+1} \to \Rset^n$ be the projection
\[
\xv = (x_1,\ldots,x_{n+1}) \mapsto \tilde{\xv} = (x_1,\ldots,x_n).
\]
Projected objects onto $\Rset^n$ are distinguished with a tilde.
Recall that $n +1 = pk$.

\begin{proposition}
\label{prop:Optimalc}
\begin{enumerate}[(a)]
\item
If $\Gamma$ is a circular $k$-gon and $p \in \Nset$, then
\[
\# \Pi(p,q) \ge 2^{kp} G_q \big( P^{(p)} \big) \ge
2^{kp} V \big( \tilde{K}^{(p)}_{\infty} \big) q^{kp-1} + \Order(q^{kp-2})
\quad \mbox{ as $q \to +\infty$},
\]
where
$\tilde{K}^{(p)}_\infty = \overline{\lim_{q \to +\infty} \tilde{P}^{(p)}_q}$
is the closure of the limit of the bounded convex polytopes
\[
\tilde{P}^{(p)}_q = \Pi_{n+1}\big( P^{(p)}_q \big),\qquad
P^{(p)}_q = P^{(p)}/q \cap H_{n+1},\qquad
P^{(p)}/q = \{ \xv /q : \xv \in P^{(p)} \},
\]
which are computed by $q$-contraction,
section with hyperplane $H_{n+1}$ and
projection by $\Pi_{n+1}$ of the unbounded
convex polytope $P^{(p)}$ defined in~\eqref{eq:UnboundedPolytope}.
\item
This lower bound is optimal in the sense that
\[
\lim_{q \to +\infty} q^{1-kp} G_q\big( P^{(p)} \big) =
\Volume\big( \tilde{K}^{(p)}_\infty \big).
\]
\item
The half-space representation of the limit compact convex polytope is
\begin{equation}
\label{eq:ProjectedLimitPolytope}
\tilde{K}^{(p)}_\infty =
\overline{\lim_{q \to +\infty} \tilde{P}^{(p)}_q} =
\left\{
\tilde{\xv} \in \Rset^n :
\begin{array}{l}
\alpha^-_j x_j \le x_{j+1} \le \alpha^+_j x_j,
\quad \forall j=1,\ldots,n-1 \\
\alpha^-_n x_n \le 1 - \varsigma(\tilde{\xv}) \le \alpha^+_n x_n \\
\alpha^-_{n+1} (1 - \varsigma(\tilde{\xv})) \le x_1 \le
\alpha^+_{n+1} (1 - \varsigma(\tilde{\xv})) \\
x_j \ge 0, \quad \forall j=1,\ldots,n \\
\varsigma(\tilde{\xv}) \le 1
\end{array}
\right\},
\end{equation}
where $\varsigma(\tilde{\xv}) = x_1 + \cdots + x_n$.
\end{enumerate}
\end{proposition}
There exist several algorithms to compute the volume of
compact convex polytopes from their half-space representations,
so expression~\eqref{eq:ProjectedLimitPolytope} can be used
to compute $V\big( \tilde{K}^{(p)}_\infty \big)$.

\subsection{Proof of the polynomial general lower bound}
\label{ssec:ProofPolynomial}

Recall that $P^{(p)}$ is the unbounded convex 
polytope~\eqref{eq:UnboundedPolytope}.
We will introduce a cube
\begin{equation}
\label{eq:Cube}
K = \{ \xv \in \Rset^{n+1} : |\xv - \ov|_\infty \le t \},
\end{equation}
which is the ball centered at the point $\ov \in \Rset^{n+1}$
of radius $t$ in the infinity norm $| \cdot |_\infty$.
Its center $\ov=(o_1,\ldots,o_{n+1})$ will have three key properties:
1)~$\ov \in P^{(p)}$, 2)~$\sum_{j=1}^{n+1} o_j = q$, and
3)~$o_j = o_{j \mod k}$.
Then, radius $t$ is taken as the largest value such that $K \subset P^{(p)}$.
For convenience, we will not make explicit the dependence of $K$ on
the integers $1 \le p < q$.

\begin{lemma}
\label{lem:ContainedCube}
Let $k,n,p,q \in \Nset$ such that $1 \le p < q$ and $n+1 = kp$.
Recall constants listed in Remark~\ref{rem:Constants}.
If $\kappa_\star = a_\star/4$,
$\tau_\star = \max \{ \chi_1,\ldots,\chi_k \}$,
$0 \le t < t_\star = \kappa_\star q/(n+1) - \tau_\star$ and
\[
\ov = (o_1,\ldots,o_{n+1}) \in \Rset^{n+1}, \qquad
o_j = \frac{qA_{j \mod k}}{(n+1)A},
\]
then $o_1 + \cdots + o_{n+1} = q$ and
$K = \{ \xv \in \Rset^{n+1} : |\xv - \ov|_\infty \le t \} \subset P^{(p)}$.
\end{lemma}

\begin{proof}
Clearly,
$o_1 + \cdots + o_{n+1} =
 \frac{q}{(n+1)A} \sum_{j=1}^{n+1} A_j =
 \frac{q p}{(n+1)A} \sum_{j=1}^k A_j =
 \frac{kp}{n+1}q = q$.

If $\xv \in K$, then $\xv = \ov + t \uv$ for some $\uv \in \Rset^{n+1}$
such that $|\uv|_\infty \le 1$.
With the suitable choice of the radius $t$,
the point $\xv$ satisfies the following three sets of inequalities
that define the unbounded convex polytope
$P^{(p)}$ given in~(\ref{eq:UnboundedPolytope}):
\begin{itemize}
\item
\emph{First set (with $2n$ inequalities).}
Since $o_{j+1} = \alpha_j o_j$ for all $j=1,\ldots,n$, we see that
\[
\alpha^-_j x_j + \beta^-_j < x_{j+1} < \alpha^+_j x_j - \beta^+_j
\Leftrightarrow
\left\{
\begin{array}{l}
(\alpha^-_j u_j - u_{j+1})t < (\alpha_j - \alpha^-_j)o_j - \beta^-_j \\
(u_{j+1} - \alpha^+_j u_j)t < (\alpha^+_j - \alpha_j)o_j - \beta^+_j
\end{array}
\right.
\]
\item
\emph{Second set (with $2$ inequalities).}
Since $A_1 = 1$ and
$A_{n+1} = A_k = \prod_{j=1}^{k-1} \alpha_j = 1/\alpha_k$,
we get that $o_1 = \alpha_k o_{n+1} = \alpha_k o_k$.
Besides,
$\beta^\pm_{n+1} = \beta^\pm_k$ and $\alpha^\pm_{n+1} = \alpha^\pm_k$.
Hence,
\[
\hspace{-3pt}
\alpha^-_{n+1} x_{n+1} + \beta^-_{n+1} < x_1 <
\alpha^+_{n+1} x_{n+1} - \beta^+_{n+1}
\Leftrightarrow
\left\{
\begin{array}{l}
(\alpha^-_k u_{n+1} - u_1)t < (\alpha_k - \alpha^-_k)o_k - \beta^-_k \\
(u_1 - \alpha^+_k u_{n+1})t < (\alpha^+_k - \alpha_k)o_k - \beta^+_k
\end{array}
\right.
\]
\item
\emph{Third set (with $n+1$ inequalities).}
$x_j \ge \chi_j \Leftrightarrow
-u_j t \le o_j - \chi_j$.
\end{itemize}
Let us analyse the RHS and LHS of the above $3n+3$ inequalities.
Coordinates $o_j$ can be as big as needed if we take $q/(n+1) \gg 1$,
because quotients $A_{j \mod k}/A$ do not depend on $p$, $q$ or $n$.
Thus, using that $\alpha^-_j < \alpha_j < \alpha^+_j$ for all $j=1,\ldots,k$,
all RHS can be made positive if we take $q/(n+1) \gg 1$.
On the other hand,
we can bound the LHS as follows:
\[
(\alpha^-_j u_j - u_{j+1}) t \le (1 + \alpha^-_j) t,\qquad
(u_{j+1} - \alpha^+_j u_j) t \le (1 + \alpha^+_j) t,\qquad
-u_j t \le t,
\]
because $|u_j| \le |\uv|_\infty \le 1$ for all $j=1,\ldots,n+1$
and $t \ge 0$.
Therefore, these $3n+3$ inequalities hold when we take any
$t \in [0, t_\star)$ with
\begin{align*}
t_\star
&=
\min
\left\{
\frac{(\alpha_j - \alpha^-_j)o_j - \beta^-_j}{1 + \alpha^-_j},
\frac{(\alpha^+_j - \alpha_j)o_j - \beta^+_j}{1 + \alpha^+_j},
o_j - \chi_j : j=1,\ldots,n+1
\right\} \\
&=
\min
\left\{
\kappa^-_j q/(n+1) - \tau^-_j, \kappa^+_j q/(n+1) - \tau^+_j,
\kappa_j q/(n+1) - \tau_j :
j=1,\ldots,k
\right\},
\end{align*}
where
\[
\kappa^\pm_j = \frac{|\alpha_j - \alpha^\pm_j| A_j}{(1 + \alpha^\pm_j)A}, \qquad
\kappa_j = \frac{A_j}{A}, \qquad
\tau^\pm_j = \frac{\beta^\pm_j}{1 + \alpha^\pm_j} = 1,\qquad
\tau_j = \chi_j.
\]
All these arguments imply that $K \subset P^{(p)}$
provided that $0 \le t < t_\star :=\kappa_\star q/(n+1) - \tau_\star$,
where
\[
\kappa_\star =
\min
\left\{
\kappa^-_1, \kappa^+_1, \kappa_1, \ldots, \kappa^-_k, \kappa^+_k, \kappa_k
\right\} > 0, \quad
\tau_\star =
\max\left\{
\tau^-_1, \tau^+_1, \tau_1, \ldots, \tau^-_k, \tau^+_k, \tau_k
\right\} > 0.
\]
These constants $\kappa_\star$ and $\tau_\star$
do not depend on $p$, $q$ or $n$.

We note that $(\alpha^+_j - \alpha_j)/(1+\alpha^+_j) < 1$.
Hence, $\kappa^+_j < \kappa_j$ and we can take
\[
\kappa_\star =
\min \left\{\kappa^-_1, \kappa^+_1, \ldots, \kappa^-_k,\kappa^+_k \right\} =
a_\star/4,
\]
see Remark~\ref{rem:Constants}.
\end{proof}

We look for a lower bound on $G_q(K)$,
where $K$ is a cube of the form~(\ref{eq:Cube}) such that
$\sum_{j=1}^{n+1} o_j= q$.
Note that $G_q\big( \{0,1\}^n \big) = \binom{n}{q}$,
so $G_q\big( \{0,1\}^{2q} \big)  = \binom{2q}{q} \ge 4^q/(2q+1)$
grows exponentially fast as $q \to +\infty$.
We want to generalise this idea.
Since there is no standard notation for the generalised
binomial coefficients that we need
---for instance, symbols $\binom{n,m}{q}$ and $\binom{n}{q}^{(m)}$
can be found in~\cite{Neuschel2014,Li2020}---,
we use our own notation.
Set
\[
[0..m] := \Zset \cap [0,m] = \{0,1,\ldots,m-1,m\}.
\]
Then $G_q\big( [0..m]^n \big)$ counts the number of \emph{weak compositions}
of $q$ into $n$ parts with no part exceeding $m$.
Note that $G_q\big( [0..m]^n \big) = 0$ for any $q \not \in [0..nm]$.
It is well know~\cite[section I.3]{FlajoletSedgewick2009} that
\[
\sum_{q=0}^{\infty} G_q\big( [0..m]^n \big) x^q =
(1 + x + x^2 + \cdots + x^m)^n.
\]
Using this polynomial identity,
Andrews~\cite{andrews1975polynomials} deduced that, once $m,n \in \Nset$ are fixed,
the sequence $G_q\big( [0..m]^n \big)$ is unimodal in $q$
and reaches its maximum at $q = \lfloor nm/2 \rfloor$.

\begin{lemma}
\label{lem:MultinomialCoefficients}
$G_{\lfloor nm/2 \rfloor}\big( [0..m]^n\big) \ge
 \frac{(m+1)^n}{nm+1} \ge
 \frac{(m+1)^{n-1}}{n}$ for all $m,n \in \Nset$.
\end{lemma}

\begin{proof}
It follows from
$\# [0..nm] = nm+1$,
$\sum_{q=0}^{nm} G_q\big( [0..m]^n \big) =
 \# \left( [0..m]^n \right) = (m+1)^n$, and inequalities
$G_q\big( [0..m]^n \big) \le G_{\lfloor nm/2 \rfloor}\big( [0..m]^n\big)$
for all $q \in [0..mn]$.
\end{proof}

Now we are ready to establish the lower bound on $G_q(K)$ that
we are looking for.

\begin{lemma}
\label{lem:LowerBoundCube}
Let $n,q \in \Nset$ and $t > 0$.
If $K$ is a cube of the form~(\ref{eq:Cube})
such that $\sum_{j=1}^{n+1} o_j = q$ and $t \ge 3/2$,
then
\begin{equation}
\label{eq:LowerBoundCube}
G_q(K) \ge \frac{(2t-3)^n}{n+1}.
\end{equation}
\end{lemma}

\begin{proof}
There exists an integer point
$\ov' \in \Zset^{n+1}$ such that $|\ov - \ov'|_\infty \le 1$
and $\sum_{j=1}^{n+1} o'_j = q$.
If $\ov \in \Zset^{n+1}$, we take $\ov' = \ov$.
If $\ov \not \in \Zset^{n+1}$, we can take, for instance,
\[
o'_j =
\begin{cases}
\lfloor o_j \rfloor + 1, & \text{for } j \le i, \\
\lfloor o_j \rfloor,     & \text{otherwise},
\end{cases}
\]
where $i = q - \sum_{j=1}^{n+1} \lfloor o_j \rfloor \in [1..n]$,
so that
$\sum_{j=1}^{n+1} o'_j =
 i + \sum_{j=1}^{n+1} \lfloor o_j \rfloor = q$.

Set $m = \lfloor t \rfloor - 1 \in \Nset \cup \{ 0 \}$ and
$v_j = o'_j - m$.
Clearly, $[v_j, v_j + 2m] \subset [o_j - t, o_j + t]$.
Hence, given any $\yv \in [0..2m]^{n+1}$
such that $\sum_{j=1}^{n+1} y_j = (n+1)m$,
the sum of the components of the vector
$\xv = \yv + \vv \in K \cap \Zset^{n+1}$ is equal to
\[
\sum_{j=1}^{n+1} x_j =
(n+1)m + \left( \sum_{j=1}^{n+1} o'_j \right) - (n+1)m =
q.
\]
Besides, the correspondence
$[0..2m]^{n+1} \ni \yv \mapsto \xv = \yv + \vv \in K \cap \Zset^{n+1}$
is injective, which implies that
\[
G_q(K) \ge
G_{(n+1)m} \big( [0..2m]^{n+1} \big) =
G_{\lfloor (n+1)2m/2\rfloor} \big( [0..2m]^{n+1} \big) \ge
\frac{(2m+1)^n}{n+1} \ge
\frac{(2t-3)^n}{n+1}.
\]
We have used Lemma~\ref{lem:MultinomialCoefficients}
and $m = \lfloor t \rfloor - 1 \ge t - 2$
in the last two inequalities.
\end{proof}

To end, we prove the first lower bound stated in
Theorem~\ref{thm:LowerBound}.

\emph{Proof of the polynomial general lower bound.}
This bound follows from bound~(\ref{eq:LowerBoundPpq}),
the inclusion $K \subset P^{(p)}$, bound~(\ref{eq:LowerBoundCube}),
condition $0 \le t < t_\star := \kappa_\star q/(n+1) - \tau_\star$
required in Lemma~\ref{lem:ContainedCube},
and the identities $a_\star = 4 \kappa_\star$,
$b_\star = 4\tau_\star + 6$ and $n+1 = kp$.
Namely,
\begin{align*}
\# \Pi(p,q) &\ge
2^{n+1} G_q(P^{(p)}) \ge
\max_{t \in [3/2,t_\star)} \big\{ 2^{n+1} G_q(K) \big \} \ge
\max_{t \in [3/2,t_\star)} \frac{2(4t-6)^n}{n+1} =
\frac{2(4t_\star-6)^n}{n+1} \\
&=
\frac{2}{n+1}
\left( \frac{4\kappa_\star q}{n+1} - 4\tau_\star - 6 \right)^n =
\frac{2}{n+1} \left( \frac{a_\star q}{n+1} - b_\star \right)^n =
\frac{2}{kp} \left( \frac{a_\star q}{kp} - b_\star \right)^{kp-1}.
\end{align*}
Note that $[3/2,t_\star) \neq \emptyset$ since
$q > b_\star kp/a_\star$ implies that
$t_\star = \kappa_\star q/(n+1) - \tau_\star > 3/2$.
\qed

\subsection{Proof of the two asymptotic lower bounds}
\label{ssec:ProofAsymptotic}

We describe the exponentially fast growth
of $\#\Pi(p,q)$ when $p = \lfloor xq/k \rfloor$ and $q \to +\infty$
for some fixed limit ratio $x > 0$.
We shall also determine the limit ratio $x_\star > 0$
that gives the largest exponent $h_\star$ in the exponential bound.

\begin{lemma}
\label{lem:ExponentialLowerBound}
Let $0 < a < b$ and $k \in \Nset$.
If
\begin{itemize}
\item
$M(x) = 2(a/x- b)^{-k-1}/x$ for $0 < x < a/b$;
\item
$h(x) = x \log(a/x -b)$ for $0 < x < a/b$; and
\item
$G(p,q) = 2(aq/kp - b)^{kp-1}/kp$
for $p,q \in \Nset$ such that $0 < kp/q < a/b$,
\end{itemize}
then
\[
G(\lfloor xq/k \rfloor,q)
\ge M(x) \rme^{h(x) q}/q, \qquad
\forall q \ge (1+b)kp/a, \ \forall x \in \big( 0,a/(b+1) \big].
\]
The exponent $h:(0,a/b) \to \Rset$ reaches its maximum value
$h_\star = h(x_\star) = a W_0(b/\rme)/b > 0$ at the point
$x_\star = a W_0(b/\rme)/((1 + W_0(b/\rme))b)$.
\end{lemma}

\begin{proof}
If $q \in \Nset$, $x \in (0,a/(b+1)]$ and $p = \lfloor xq/k \rfloor$,
then $aq/kp - b \ge 1$, $xq - k < kp \le xq$ and
\[
G(p,q) =
\frac{2}{kp} \left( \frac{aq}{kp} - b \right)^{kp - 1} \ge
\frac{2}{kp} \left( \frac{aq}{kp} - b \right)^{xq-k-1} \ge
\frac{2}{xq}\left( \frac{a}{x} - b \right)^{xq-k-1} =
\frac{1}{q} M(x) \rme^{h(x)q}.
\]
Next, we look for the global maximum of $h(x)$.
After the changes of variable
\begin{align*}
(0,+\infty) \ni \hat{x} &\leftrightarrow
x = \frac{a}{b + \rme^{\hat{x}}} \in \big( 0,a/(b + 1) \big), \\
(0,+\infty) \ni \hat{h} &\leftrightarrow h = a\hat{h} \in(0,+\infty),
\end{align*}
we get that
$h(x) = x\log(a/x - b) = x\hat{x} = a \hat{x}/(b + \rme^{\hat{x}})$,
so $\hat{h}(\hat{x}) = \hat{x}/(b + \rme^{\hat{x}})$.
We have reduced the search of the global maximum point
$x_\star \in(0,a/(b +1))$ of $h(x)$ to
the search of the global maximum point $\hat{x}_\star > 0$
of $\hat{h}(\hat{x})$.
Since
\[
\frac{\rmd \hat{h}}{\rmd \hat{x}}(\hat{x}) =
\frac{b + (1-\hat{x})\rme^{\hat{x}}}{(b + \rme^{\hat{x}})^2} = 0
\Leftrightarrow
(\hat{x}-1)\rme^{\hat{x}-1} = b/\rme \Leftrightarrow
\hat{x} = \hat{x}_\star := 1 + W_0(b/\rme),
\]
we deduce that $\hat{h}(\hat{x})$ reaches its maximum value
\[
\hat{h}_\star =
\hat{h}(\hat{x}_\star) =
\frac{\hat{x}_\star}{b + \rme^{\hat{x}_\star}} =
\frac{1}{\rme^{\hat{x}_\star}} =
\frac{1}{\rme}\frac{1}{\rme^{W_0(b/\rme)}} =
\frac{W_0(b/\rme)}{b}
\]
at the point $\hat{x} = \hat{x}_\star$.
In order to compute $\hat{x}_\star$,
we have used that $W_0(b/\rme) \rme^{W_0(b/\rme)} = b/\rme$.
Expressions for $x_\star$ and $h_\star$
are obtained by undoing both changes of variable.
\end{proof}

We can prove now the second and third lower bounds stated in
Theorem~\ref{thm:LowerBound}.

\emph{Proof of both asymptotic lower bounds.}
The second bound of Theorem~\ref{thm:LowerBound}
follows from the first one by applying Lemma~\ref{lem:ExponentialLowerBound}
with $a = a_\star$ and $b = b_\star$.
Analogously,
the third bound of Theorem~\ref{thm:LowerBound}
follows from the first one by taking
$c_\star(p) = 2 (a_\star)^{kp-1}/(kp)^{kp}$.
\qed

\section{The length spectrum of circular polygons}
\label{sec:LengthSpectrum}

The purpose of this section is to prove
Theorem~\ref{thm:AsymptoticConstantIntro},
which shows an unusual feature of the length spectrum of
billiards in circular polygons.
In particular,
it shows that the well-known results of
Marvizi-Melrose~\cite{marvizi1982spectral}
fail to hold for circular polygons.
This was expected because there are so many periodic billiard trajectories
inside circular polygons ---as we have seen in the previous section---
that we can construct sequences of them whose
lengths have rather different asymptotic behaviors.

Let $|\Gamma|$ be the length of $\Gamma$.
Let $\kappa(s)$ be the curvature of $\Gamma$ as a function of an
arc-length parameter $s \in [0,|\Gamma|)$.
If $z,z' \in \Gamma$ are any two consecutive impact points of a
billiard trajectory $g$,
then the segment $[z,z'] \subset \Rset^2$ is a \emph{link} of $g$ and
$\int_z^{z'} \rmd s$ is the distance from $z$ to $z'$ along $\Gamma$.
Note that $|z' - z| < \int_z^{z'} \rmd s$ by convexity.
If $g = \{ z_0,\ldots,z_{q-1}\} \subset \Gamma$ is a
$q$-periodic billiard trajectory,
let $L(g) = |z_1 - z_0| + \cdots + |z_{q-1} - z_0|$
be its \emph{length}.
Let
\[
\underline{L}_q =
\inf \{ L(g) : g \in \Pi(1,q) \},\qquad
\overline{L}_q =
\sup \{ L(g) : g \in \Pi(1,q) \}.
\]

To begin with,
let us recall the Marvizi-Melrose results for smooth ovals.
A \emph{smooth oval} is a regular, simple, closed,
oriented $C^\infty$ curve with positive curvature everywhere.

\begin{theorem}[Marvizi \& Melrose~\cite{marvizi1982spectral}]
\label{thm:MarviziMelrose}
Let $\Gamma$ be any smooth oval.
\begin{enumerate}[(a)]
\item
\label{item:MM1}
$\lim_{q \to +\infty}
q^i
\big( \overline{L}_q - \underline{L}_q
\big)= 0$ for all $i \in \Nset$.
\item
\label{item:MM2}
There are asymptotic coefficients $c_i \in \Rset$
such that if $g_q \in \Pi(1,q)$, then
\[
L(g_q) \asymp
|\Gamma| + \sum_{i = 1}^{\infty} \frac{c_i}{q^{2i}}, \quad
\mbox{as $q \to +\infty$.}
\]
\item
\label{item:MM3}
$c_1 =
-\frac{1}{24}
\left[ \int_{\Gamma} \kappa^{2/3}(s) \rmd s \right]^3 < 0$.
\item
\label{item:MM4}
If $[z,z']$ is a link of $g_q \in \Pi(1,q)$, then
\[
\int_{z}^{z'} \rmd s \asymp \frac{1}{q} |\Gamma|, \quad
\mbox{uniformly as $q \to +\infty$.}
\]
\end{enumerate}
\end{theorem}

The symbol $\asymp$ means that the RHS is asymptotic to the LHS.
The first property implies that the Melrose-Marvizi asymptotic
coefficients $c_i$ do not depend on the choice of the
sequence of periodic trajectories $(g_p)_q$.
All of them can be explicitly written in terms of the curvature.
For instance, the formulas for $c_1$, $c_2$, $c_3$, and $c_4$
can be found in~\cite{sorrentino2015Mather}.
Property~(d) means that not only are the lengths of $g_q$
asymptotically well-behaved, but as $q \to +\infty$, the distribution
of the points in $g_q$ is asymptotically well-behaved with respect to
any one point.
Hence, property~(d) is like a weak local version of property~(b).
There is also a strong local version
in~\cite[Theorem~5.9]{marvizi1982spectral}.

From now on, with $\Gamma$ denoting a circular polygon,
we use the notation introduced in previous sections,
although we are now only interested in $(1,q)$-periodic trajectories,
so $1 \le j \le k$ and we no longer need to consider $j$ modulo $k$.
Recall that $p=1$ along this section.

We will check that none of properties~(\ref{item:MM1})--(\ref{item:MM4})
of Theorem~\ref{thm:MarviziMelrose} hold for circular polygons.

First, we consider the simplest periodic sliding trajectories
in a circular polygon,
wich are the trajectories that impact \emph{all its nodes}
in such a way that the angle of reflection
remains \emph{constant} along the whole trajectory.
These \emph{nodal} sliding periodic trajectories can only take place
in certain circular polygons, which we call \emph{rational}.

\begin{definition}
\label{def:RationalCircularPolygon}
We say that a circular polygon $\Gamma$ is \emph{rational} when all
its central angles are rational multiples of $\pi$, so
\[
\delta_j = m_j \delta,
\]
for some $\delta = \gcd(\delta_1,\ldots,\delta_k)$ and
$m_j = \delta_j/\delta \in \Nset$.
Set $M = \sum_{j=1}^k m_j$.
Then $M \delta = 2\pi$.
A billiard trajectory inside a rational circular polygon
is \emph{nodal} when it impacts all nodes (interspersed with
possibly many other non-nodal impacts)
in the counter-clockwise ordering.
\end{definition}

Squared pseudo-ellipses and Moss's eggs are rational circular polygons,
see Section~\ref{sec:CircularPolygons}.
Any nodal orbit in a rational circular polygon has
constant angle of reflection, is sliding, and is periodic with
a rotation number of the form $1/q$, $q$ being the period.

Nodal billiard trajectories give the simplest examples of sequences
of sliding periodic billiard trajectories in circular polygons
where properties~(\ref{item:MM3}) and~(\ref{item:MM4}) of
Theorem~\ref{thm:MarviziMelrose} fail,
because it is really easy to compute their lengths.

\begin{proposition}
\label{prop:RationalCircularGons}
Let $\Gamma$ be a rational circular $k$-gon with
arcs $\Gamma_j$, radii $r_j$ and central angles $\delta_j$.
Set $\delta = \gcd(\delta_1,\ldots,\delta_k)$.
Fix some $\psi = \delta/2i$ with $i \in \Nset$.
Let $g_q$ be the billiard trajectory generated by
\[
(\varphi_n,\theta_n) = f^n(a_1,\psi), \qquad \forall n \in \Zset.
\]
\begin{enumerate}[(a)]
\item
The billiard trajectory $g_q$ is nodal and
$g_q \in \Pi(1,q)$ with period $q = Mi$.
\item
$L(g_{q}) =
 |\Gamma| - \pi^2 |\Gamma|/6q^2 + \Order(1/q^4)$
as $q = Mi \to +\infty$.
\item
If $[z,z']$ is a circular link of $g_q$ associated to the arc $\Gamma_j$,
then
\[
\int_{z}^{z'} \rmd s = \frac{1}{q} 2\pi r_j \neq \frac{1}{q} |\Gamma|.
\]
\end{enumerate}
\end{proposition}

\begin{proof}
\begin{enumerate}[(a)]
\item
Once we fix the index $i \in \Nset$,
we deduce that
$\varphi_{im_1} = a_1 + 2i m_1 \psi = a_1 + \delta_1 = b_1 = a_2$.
That is, the first $im_1$ links of the billiard trajectory connect
both nodes of the first arc.
In particular, the angle of reflection does not change
when we enter the second arc, so the next $im_2$ links connect both its
nodes, and so on.
This means that the orbit is nodal and periodic with
rotation number $1/q$ and period $q = Mi$.

\item
The length of each link in the $j$-th arc is equal to
$\ell_j = 2 r_j \sin \psi$, so
\begin{align*}
L(g_q) &=
\sum_{j=1}^k i m_j \ell_j =
\sum_{j=1}^k 2 i m_j r_j
\left[ \psi - \frac{1}{6} \psi^3 + \Order\big(\psi^5\big) \right] \\
&= \sum_{j=1}^k m_j \delta r_j -
\frac{\delta^2}{24} \left[\sum_{j=1}^k m_j \delta r_j \right] \frac{1}{i^2} +
\Order\big( i^{-4} \big) =
|\Gamma| - \frac{\pi^2 |\Gamma|}{6q^2} + \Order(1/q^4),
\end{align*}
where we have used that $\delta_j = m_j \delta$,
$|\Gamma| = \sum_{j=1}^k \delta_j r_j$, $q = Mi$,
and $M \delta = 2\pi$.

\item
If two consecutive impact points $z$ and $z'$ belong to
the arc $\Gamma_j$, then
\[
\int_{z}^{z'} \rmd s =
2 \psi r_j =
\frac{1}{i} \delta r_j =
\frac{1}{q} M\delta r_j =
\frac{1}{q} 2\pi r_j \neq
\frac{1}{q} |\Gamma|.
\qedhere
\]
\end{enumerate}
\end{proof}

The previous proposition has been obtained without
the heavy machinery developed in this paper,
but it needs a rather special type of circular polygons.
Next, we deal with general circular polygons,
where the computations are more involved.

\begin{remark}
\label{rem:PeriodicTrajectories}
Corollary~\ref{cor:ConvexPolytopes} implies that there are
at least $2^k$ generic sliding periodic billiard trajectories
$g_q \in \Pi(1,q)$ with \emph{exactly}
$x_j \in \Nset$ impacts in the arc $\Gamma_j$, $j=1,\ldots,k$,
for any integer point $\xv = (x_1,\ldots,x_k) \in
 P^{(1)} \cap \Zset^k = P^{(1)} \cap \Nset^k$ such that
$x_1 + \cdots + x_k = q$.
Here, $P^{(1)}$ is the unbounded convex polytope of $\Rset^k$
defined in~\eqref{eq:UnboundedPolytope} for $p = 1$.
\end{remark}

We need a couple of technical results before tackling
the proof of Theorem~\ref{thm:AsymptoticConstantIntro}.

First, we compute the lengths of generic
sliding $(1,q)$-periodic billiard trajectories.
By definition, they impact all arcs but no nodes.
Angles similar to $\varphi_j^\pm$ and $\psi_j$ below were
considered in the proof of part~(\ref{item:Balint_etal}) in
Lemma~\ref{lem:BilliardProperties}.

\begin{lemma}
\label{lem:LengthPeriodicTrajectory}
If $g_q \in \Pi(1,q)$ is a generic sliding periodic billiard trajectory
inside $\Gamma$, then
\begin{equation}
\label{eq:LengthPeriodicTrajectory}
L(g_q) =
\sum_{j=1}^k \big(\ell^-_j + (x_j-1) \ell_j + \ell^+_j\big),\qquad
\ell_j = 2r_j \sin\psi_j, \qquad
\ell^\pm_j = \frac{r_j \sin \varphi^\pm_j}
     {\cos(\psi_j - \varphi_j^\pm)},
\end{equation}
where
\begin{itemize}
\item[(i)]
$x_j \in \Nset$ is the number of impact points in $\Gamma_j$;
\item[(ii)]
$\psi_j > 0$ is the constant angle of reflection along the $x_j$
impacts in $\Gamma_j$; and
\item[(iii)]
$\varphi^\pm_j \in (0,2\psi_j)$ are the impact angles such that
$[z(b_j - \varphi^+_j),z(a_{j+1} + \varphi^-_{j+1})]$ is the transition
link connecting $\Gamma_j$ and $\Gamma_{j+1}$.
\end{itemize}
Besides, $\varphi^-_j + 2 (x_j-1) \psi_j + \varphi^+_j = \delta_j$
for all $j=1,\ldots,k$.
\end{lemma}

\begin{figure}[t]
\begin{center}
\includegraphics*[width=6.5cm]{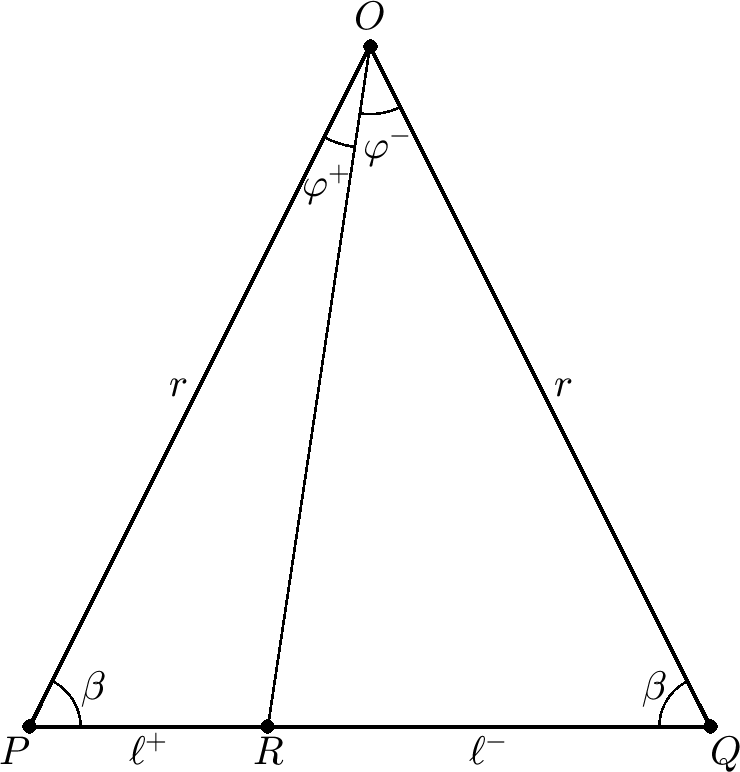}
\end{center}
\caption{\small Isosceles triangle $\Delta OPQ$ and
a point $R \in [P,Q] \setminus \{P,Q\}$
such that $\angle QPO = \angle OQP = \beta$,
$|\vec{OP}| = | \vec{OQ}| = r$,
$\angle POQ = 2 \psi = \varphi^+ + \varphi^- = \angle POR + \angle ROQ$,
and $|\vec{PQ}| = \ell = \ell^+ + \ell^- = |\vec{PR}| + |\vec{RQ}|$.
}
\label{fig:Trigonometry}
\end{figure}

\begin{proof}
If we apply the law of sinus to the three triangles $\Delta POQ$,
$\Delta POR$ and $\Delta ROQ$ displayed in Figure~\ref{fig:Trigonometry}
(see its caption for the definition of each quantity),
we get that
\[
\frac{\ell}{2\sin \psi \cos \psi} =
\frac{\ell}{\sin 2\psi} =
\frac{r}{\sin \beta} =
\frac{r}{\cos \psi}, \qquad
\frac{\ell^\pm}{\sin \varphi^\pm} =
\frac{r}{\sin(\pi-\beta-\varphi^\pm)} =
\frac{r}{\cos(\psi-\varphi^{\pm})},
\]
since $\beta = \pi/2 - \psi$,
$\angle ORP = \pi - \beta - \varphi^+$ and
$\angle QRO = \pi - \beta - \varphi^-$.
Therefore,
\[
\ell = 2r \sin \psi,\qquad
\ell^\pm = \frac{r \sin \varphi^\pm}{\cos(\psi - \varphi^\pm)}.
\]

We deduce~\eqref{eq:LengthPeriodicTrajectory} from those three formulas.

If $g_q$ has $x_j$ impacts in $\Gamma_j$ with constant
angle of reflection $\psi_j$,
then it has $x_j-1$ circular links with a certain constant
length~$\ell_j$.
Each one of these circular links $[P,Q]$ is the base of
an isosceles triangle $\Delta OPQ$ like the one displayed
in Figure~\ref{fig:Trigonometry},
with $O = O_j$, $r = r_j$ and $\psi = \psi_j$.
Hence $\ell_j = 2 r_j \sin \psi_j$.

Let us consider the transition link
$[z(b_j - \varphi^+_j),z(a_{j+1} + \varphi^-_{j+1})]$
connecting $\Gamma_j$ and $\Gamma_{j+1}$ and
the isosceles triangle $\Delta OPQ$ with
\[
O=O_j, \qquad
P = z(b_j - \varphi^+_j) \in \Gamma_j, \qquad
Q = O_j + r_j \rme^{\rmi (b_j - \varphi^+_j + 2\psi_j)}.
\]
We stress that $Q$ is an auxiliary point: $Q \not \in \Gamma$.
Let $R = [P,Q] \cap [O,z(b_j)]$.
Then $r = r_j$, $\varphi^+ = \varphi^+_j$ and $\ell^+ = \ell^+_j$.
Therefore,
$\ell^+_j = r_j \sin \varphi^+_j / \cos(\psi_j - \varphi^+_j)$.
The formula for $\ell^-_{j+1}$ is deduced in a similar way,
but taking
\[
O = O_{j+1}, \qquad
P = O_{j+1} +
    r_{j+1} \rme^{\rmi (a_{j+1} - \varphi^-_{j+1} - 2\psi_{j+1})}, \qquad
Q=z(a_{j+1} + \varphi^-_{j+1}) \in \Gamma_{j+1},
\]
$R = [P,Q] \cap [O,z(a_{j+1})]$,
$r = r_{j+1}$, $\varphi^- = \varphi^-_{j+1}$ and
$\ell^- = \ell^-_{j+1}$.
(In this case, the auxiliary point is $P$:
 $P \not \in \Gamma$).
By construction, the transition link
$[z(b_j - \varphi^+_j),z(a_{j+1} + \varphi^-_{j+1})]$
has length $\ell^+_j + \ell^+_{j+1}$.
This proves~\eqref{eq:LengthPeriodicTrajectory}.

Finally,
relation $\varphi^-_j + 2 (x_j-1) \psi_j + \varphi^+_j = \delta_j$
is geometrically evident.
\end{proof}

Next, we need a technical result about the extreme values of
the differentiable strictly concave function~(\ref{eq:AuxiliaryFunction})
over a bounded convex polytope $P^{(1)}_\infty$ related to the
unbounded convex polytope $P^{(1)}$ of $\Rset^k$ defined
in~\eqref{eq:UnboundedPolytope} for $p=1$.
Recall that
$\Delta_{k-1} = \{\xv \in \Rset^k : \xv > 0, \ x_1 + \cdots + x_k = 1\}$
is the open $(k-1)$-simplex and
$H_k = \{ \xv \in \Rset^k : x_1 + \cdots + x_k = 1\}$.

\begin{lemma}
\label{lem:AsymptoticConstantInterval}
The bounded convex polytope
$P^{(1)}_\infty =
\lim_{q \to \infty}
\left( \big\{ \xv/q : \xv \in P^{(1)} \big\} \cap H_k \right)$
is given by
\begin{equation}
\label{eq:LimitPolytope1}
P^{(1)}_\infty =
\left\{
\xv \in \Rset^k :
\begin{array}{l}
\alpha^-_j x_j < x_{j+1} < \alpha^+_j x_j,
\quad \forall j=1,\ldots,k-1 \\
\alpha^-_k x_k < x_1 < \alpha^+_k x_k  \\
x_j > 0, \quad \forall j=1,\ldots,k \\
x_1 + \cdots + x_k = 1
\end{array}
\right\}
\end{equation}
and its compact closure $K^{(1)}_\infty$ is contained in
the open simplex $\Delta_{k-1}$.
Let
\begin{equation}
\label{eq:AuxiliaryFunction}
h: \Delta_{k-1}  \to (-\infty,0), \qquad
h(\yv) =
-\frac{1}{24} \sum_{j=1}^k \frac{\delta_j^3 r_j}{y^2_j}.
\end{equation}
Set $I_1 = h\big( P^{(1)}_\infty \big)$,
$c^-_1 = \inf I_1$ and $c^+_1 = \sup I_1$.
Then $c_1^+ \in I_1$ and 
\begin{equation}
\label{eq:ExtremalValues}
-\infty < c_1^- \le -\pi^2 |\Gamma|/6 < c_1^+ =
\frac{1}{24} \left[ \int_\Gamma \kappa^{2/3}(s) \rmd s \right]^3 < 0.
\end{equation}
\end{lemma}

\begin{proof}
Expression~\eqref{eq:LimitPolytope1} is trivial.
We check that $K^{(1)}_{\infty} \subset \Delta_{k-1}$
by a reductio ad absurdum argument.
Let us assume that $\xv = (x_1,\ldots,x_k) \in K^{(1)}_{\infty}$
and $x_i = 0$ for some $i$.
Then inequalities
$\alpha^-_j x_j \le x_{j+1} \le \alpha^+_j x_j$ for $j=1,\ldots,k-1$
and $\alpha^-_k x_k \le x_1 \le \alpha^+_k x_k$
imply that
\[
x_{i+1} = \cdots = x_k = x_1 = \cdots = x_{i-1} = 0,
\]
so identity $x_1 + \cdots + x_k = 1$ fails. Contradiction.

The image of a compact convex set by a continuous function
that only takes negative values is a compact interval of $(-\infty,0)$,
so $\overline{I_1} = h \big( K_\infty^{(1)} \big) =
[c^-_1,c^+_1]$ for some numbers $-\infty < c^-_1 \le c^+_1 < 0$.
Let us estimate the minimum value $c^-_1$,
compute exactly the maximum value $c^+_1$,
prove that $c^-_1 < c^+_1$,
and check that $c^+_1 \in I_1$.

We claim that function~(\ref{eq:AuxiliaryFunction})
attains its maximum value only at $\yv = \wv(1/3)$, where
\[
\wv(\xi) =
\frac{1}{S(\xi)}\big( s_1(\xi), \ldots, s_k(\xi) \big) \in
\Delta_{k-1}, \qquad
s_j(\xi) = \delta_j r^\xi_j, \qquad
S(\xi) = \sum_{j=1}^k \delta_j r^\xi_j
\]
for all $\xi \in \Rset$.
On the one hand,
the gradient of $\Rset_+^k \ni \yv \mapsto \sum_{j=1}^k \delta^3_j r_j/y^2_j$
is the vector with components $-2\delta^3_j r_j/y^3_j$,
so $\yv \in \Delta_{k-1}$ is a critical point of~(\ref{eq:AuxiliaryFunction})
if and only if
\[
\big( s_i(1/3)/y_i \big)^3 =
\delta^3_i r_i /y^3_i =
\delta^3_j r_j /y^3_j =
\big( s_j(1/3)/y_j \big)^3, \qquad \forall i \neq j.
\]
This means that $\wv(1/3)$ is the only critical point
of~(\ref{eq:AuxiliaryFunction}).
On the other hand,
$\sum_{j=1}^k \delta^3_j r_j/y^2_j$ is a nonnegative weighted
sum of convex terms $1/y^2_j$,
so $-\frac{1}{24}\sum_{j=1}^k \delta^3_j r_j/y^2_j$ is
a strictly concave function on $\Rset_+^k$
and~(\ref{eq:AuxiliaryFunction}) is a differentiable strictly concave function.
Hence, the local maximum $\wv(1/3)$ is a strict global maximum.
This proves the claim.
Furthermore,
\[
h(\wv(\xi)) =
-\frac{1}{24}
\sum_{j=1}^k \frac{\delta^3_j r_j}{\big( s_j(\xi)/S(\xi) \big)^2} =
-\frac{1}{24} S(\xi)^2 \sum_{j=1}^k \delta_j r^{1-2\xi}_j =
-\frac{1}{24} S(\xi)^2 S(1-2\xi).
\]
In particular, $h(\wv(0)) < h(\wv(1/3))$ and
\begin{align*}
h(\wv(0)) &=
-S(0)^2 S(1)/24 = -(2\pi)^2 |\Gamma|/24 = -\pi^2 |\Gamma|/6, \\
h(\wv(1/3)) &=
-S(1/3)^3/24 =
-\frac{1}{24} \left[ \sum_{j=1}^k \delta_j r^{1/3}_j \right]^3 =
-\frac{1}{24} \left[ \int_{\Gamma} \kappa^{2/3}(s) \rmd s \right]^3.
\end{align*}
Here we have used that $|\Gamma_j| = \delta_j r_j$ and
$\int_{\Gamma_j} \kappa^{2/3}(s) \rmd s = |\Gamma_j|/r_j^{2/3} =
 \delta_j r_j^{1/3}$ since $\Gamma_j$ is a circular arc of radius $r_j$
and central angle $\delta_j$.

Hence, property $c_1^+ \in I_1$ and
inequalities~(\ref{eq:ExtremalValues}) hold provided
$\wv(0) \in K^{(1)}_\infty$ and $\wv(1/3) \in P_\infty^{(1)}$.
It turns out that $\wv(\xi)$ satisfies the $3k+1$ conditions listed
in~(\ref{eq:LimitPolytope1}), so that $\wv(\xi) \in P^{(1)}_\infty$,
for all $\xi \in (0,1/2]$.
For instance, $\wv(\xi)$ satisfies the first $2k-2$ inequalities:
\begin{align*}
\xi \in (0,1/2]
&\Rightarrow
r^\xi_j \min\left\{ 1, \sqrt{r_{j+1}/r_j} \right\} <
r^\xi_{j+1} <
r^\xi_j \max\left\{ 1, \sqrt{r_{j+1}/r_j} \right\} \\
&\Rightarrow
\alpha^-_j w_j(\xi) < w_{j+1}(\xi) < \alpha^+_j w_j(\xi)
\mbox{ for $j=1,\ldots,k-1$}.
\end{align*}
Inequalities
$\alpha^-_k w_k(\xi) < w_1(\xi) < \alpha^+_k w_k(\xi)$ are proved
in a similar way.
Inequalities $w_j(\xi) > 0$ and
identity $w_1(\xi) + \cdots + w_k(\xi) = 1$ are trivial.
Finally,
$\wv(0) = \lim_{\xi \to 0^+} \wv(\xi) \in K^{(1)}_\infty$
\end{proof}

The main result of this section is nothing more than a reformulation of
Theorem~\ref{thm:AsymptoticConstantIntro}.

\begin{theorem}
\label{thm:AsymptoticConstant}
Let $P^{(1)}_\infty \subset \Delta_{k-1}$,
$I_1 = h\big( P^{(1)}_\infty \big) \subset (-\infty,0)$,
$c_1^- = \inf I_1$ and $c_1^+ = \max I_1$
be the open bounded convex polytope of $\Rset^k$,
the image interval, and the extreme values introduced in
Lemma~\ref{lem:AsymptoticConstantInterval}, respectively.
Extreme values $c_1^\pm$ satisfy inequalities~\eqref{eq:ExtremalValues}.
For any fixed $c \in [c_1^-,c_1^+]$ there exist a period $q_0 \in \Nset$
and a sequence $(g_q)_{q \ge q_0}$ of generic sliding periodic
billiard trajectories $g_q \in \Pi(1,q)$ such that
\[
L(g_q) = |\Gamma| + c/q^2 + \Order(1/q^3),\quad
\mbox{as $q \to +\infty$.}
\]
Consequently, there exist a
sequence $(h_q)_q$, with $h_q \in \Pi(1,q)$, such that
\[
c_1^- =
\liminf _{q \to +\infty} \big( (L(h_q) - |\Gamma|)q^2 \big) <
\limsup _{q \to +\infty} \big( (L(h_q) - |\Gamma|)q^2 \big) =
c_1^+,\quad
\mbox{as $q \to +\infty$.}
\]
\end{theorem}

\begin{proof}
If $c \in (c_1^-,c_1^+]$,
then $c = h(\yv)$ for some $\yv \in P^{(1)}_\infty$.
If $q \in \Nset$ is big enough,
then there exists a point $\xv = (x_1,\ldots,x_k) \in \Nset^k$
such that $|q\yv - \xv |_\infty \le 1$ and $\xv \in P^{(1)} \cap qH_k$,
where $P^{(1)}$ is the unbounded convex polytope
defined in (\ref{eq:UnboundedPolytope}) for $p=1$.
Let us prove this claim.

First, we observe that $y_j > 0$, so $qy_j \ge 1$ when $q \gg 1$.
If $q\yv \in \Nset^k$, then we take $\xv = q\yv$.
If $q\yv \not \in \Nset^k$, then we can take, for instance,
\[
x_j =
\begin{cases}
\lfloor q y_j \rfloor + 1,   & \text{for } j \le i, \\
\lfloor q y_j \rfloor, & \text{otherwise},
\end{cases}
\]
where $i = q - \sum_{j=1}^k \lfloor qy_j \rfloor \in \{1,\ldots,k-1\}$,
so that
$\sum_{j=1}^k x_j = i + \sum_{j=1}^k \lfloor qy_j \rfloor = q$.
This means that $\xv \in qH_k$.
To end the proof of the claim,
we deduce that $\xv \in P^{(1)}$ from limits
$\lim_{q \to +\infty} \xv/q =
 \yv \in P^{(1)}_\infty$ and
$P^{(1)}_\infty =
 \lim_{q \to \infty}
 \left( \big\{ \xv/q : \xv \in P^{(1)} \big\} \cap H_k\right)$.
Recall that $P_\infty^{(1)}$ is an open set in $H_k$.

As we have explained before, see Remark~\ref{rem:PeriodicTrajectories},
if $q \gg 1$ then there are at least $2^k$
generic sliding periodic billiard trajectories $g_q \in \Pi(1,q)$
with exactly $x_j \in \Nset$ impacts on the arc $\Gamma_j$ and
length~(\ref{eq:LengthPeriodicTrajectory}).
The numbers $x_j \in \Nset$,
the constant angles of reflection $\psi_j > 0$
and the impact angles $\varphi^\pm_j \in (0,2\psi_j)$
described in Lemma~\ref{lem:LengthPeriodicTrajectory} satisfy identity
$\varphi^-_j + 2 (x_j-1) \psi_j + \varphi^+_j = \delta_j$
and uniform estimates
$x_j = qy_j + \Order(1)$, $\varphi^\pm_j = \Order(1/q)$ and
$\psi_j = \delta_j/2(x_j - 1) + \Order(1/q^2) = \Order(1/q)$
as $q \to +\infty$.
Therefore,
\[
\ell^\pm_j =
\frac{r_j \sin \varphi^\pm_j}{\cos(\psi_j - \varphi_j^\pm)} =
r_j \varphi^\pm_j +
\Order\Big(
(\varphi^\pm_j)^3,\varphi^\pm_j|\psi_j - \varphi^\pm_j|^2 \Big) =
r_j \varphi^\pm_j  + \Order(1/q^3)
\]
and
\begin{align*}
(x_j - 1)\ell_j
&=
2 r_j (x_j-1) \sin \psi_j =
2 r_j (x_j - 1) \Big( \psi_j - \psi_j^3/6 +
\Order\big(\psi^5_j \big) \Big) \\
&= 
2 r_j (x_j-1)\psi_j - r_j (x_j-1) \psi_j^3/3 + \Order( 1/q^4) \\
&=
2 r_j (x_j-1)\psi_j -
\frac{\delta^3_j r_j}{24(x_j -1)^2} + \Order(1/q^3) \\
&=
2 r_j (x_j-1)\psi_j - 
\frac{\delta^3_j r_j}{24y^2_j} \frac{1}{q^2} + \Order(1/q^3).
\end{align*}
Finally, we estimate the total length~(\ref{eq:LengthPeriodicTrajectory})
as follows:
\begin{align*}
L(g_q)
&=
\sum_{j=1}^k \left( \ell^-_j + (x_j-1) \ell_j + \ell^+_j \right) \\
&=
\sum_{j=1}^k r_j (\varphi^-_j + 2(x_j-1)\psi_j + \varphi^+_j) +
h(\yv)/q^2 + \Order(1/q^3) \\
&=
|\Gamma| + c/q^2 + \Order(1/q^3).
\end{align*}
We have used that
$\varphi^-_j + 2 (x_j-1) \psi_j + \varphi^+_j = \delta_j$,
$|\Gamma| = \sum_{j=1}^k \delta_j r_j$ and $c = h(\yv)$
in the last line.
Function $h(\yv)$ was defined in~(\ref{eq:AuxiliaryFunction}).
This ends the proof of the case $c \in (c_1^-,c_1^+]$.

The case $c = c_1^-$ can be obtained from the case
$c \in (c_1^-,c_1^+]$ by using a classical diagonalisation argument
about sequences of sequences of lengths.

Finally, sequence $(h_q)_q$ is constructed by
interleaving two sequences of generic sliding $(1,q)$-periodic
billiard trajectories associated with the asymptotic
coefficients $c_1^-$ and $c_1^+$ respectively.
\end{proof}

As a by-product of Theorem~\ref{thm:AsymptoticConstant},
we get that generic sliding $(1,q)$-periodic billiard trajectories
inside circular polygons are \emph{asymptotically shorter}
than the ones inside smooth ovals,
since $c_1^+$ in~\eqref{eq:ExtremalValues} has the same
formula that constant $c_1$ in part~(\ref{item:MM3})
of Theorem~\ref{thm:MarviziMelrose}.
The relation between the left endpoint $c_1^-$ of the interval $I_1$
and $-\pi^2 |\Gamma|/6$ is an open problem.
If $c_1^- = -\pi^2 |\Gamma|/6$ were true,
nodal billiard trajectories would be the \emph{asymptotically shortest}
sliding $(1,q)$-periodic billiard trajectories as $q \to +\infty$.
This is one of the reasons to take them into account.

Identity $c_1^- = \pi^2 |\Gamma|/6$ holds for squared pseudo-ellipses.
Let $\Gamma = E_{\pi/2,r,R}$ be a squared pseudo-ellipse
of radii $r$ and $R > r$, see Section~\ref{sec:CircularPolygons}.
That is, $\Gamma = E_{\pi/2,r,R}$ is the circular $4$-gon with
radii $r_1 = r_3 = r$ and $r_2 = r_4 = R$, and central angles
$\delta_1 = \delta_2 = \delta_3 = \delta_4 = \pi/2$.
A tedious computation that we omit for the sake of brevity
shows that
\begin{align*}
c^-_1 &=
-\pi^2 |E_{\pi/2,r,R}|/6 =
-\pi^3 (R + r)/6, \\
c^+_1 &=
-\frac{1}{24}
\left[ \int_{E_{\pi/2,r,R}} \kappa^{2/3}(s) \rmd s \right]^3 =
-\frac{1}{24} \left[ \sum_{j=1}^k \delta_j \sqrt[3]{r_j} \right]^3 =
-\frac{\pi^3}{24} \left[\sqrt[3]{R} + \sqrt[3]{r}) \right]^3.
\end{align*}
These two expressions above coincide when $R=r$.
In general, $c^+_1 - c^-_1$
tends to zero when $\Gamma$ tends to a circle of finite radius.

\section*{Acknowledgments}
A. C. has received funding for this project from
the European Research Council (ERC) under the European Union’s Horizon 2020
research and innovation programme (Grant Agreement No 757802).
R.~R.-R. was supported in part by the grant PID-2021-122954NB-100
which was funded by MCIN/AEI/10.13039/501100011033 and
“ERDF: A way of making Europe”.
Thanks to Aida Chaikh and Pau Mart{\'\i}n for useful and
stimulating conversations.
We are also indebted to the referee for several suggestions
that helped us to improve the exposition.

\appendix

\section{Proof of Lemma~\ref{lem:AdmissibleSymbols}}
\label{app:AdmissibleSymbols}

Throughout this proof we shall freely use the natural convention
that the objects $\Xi^i_{k+1},\zeta^i_{k+1},\xi^i_{k+1}$ should be
identified with $\Xi^{i+1}_1,\zeta^{i+1}_1,\xi^{i+1}_1$, respectively.

\begin{enumerate}[(a)]
\item
The key observation is that functions $\zeta^i_j(n)$ and $\xi^i_j(n)$
can be recursively bounded.
To be precise,
since $\zeta^i_{j+1}(n)$ is the smallest integer such that
$\zeta^i_{j+1}(n) > \alpha_j^- \zeta^i_j(n) + \beta_j^-$,
$\xi^i_{j+1}(n)$ is the largest integer such that
$\xi^i_{j+1}(n) < \alpha_j^+ \xi^i_j(n) - \beta_j^+$
and $\beta_j^+ = \alpha_j^+ + 1 > 1$, we deduce that
\begin{equation}
\label{eq:ZetaXiSingle}
\left.
\begin{array}{rl}
\zeta^0_{j+1}(n) \le
\alpha_j^- \zeta^0_j(n) + \beta_j^- + 1, &  \quad \forall j=1.\ldots,k \\
\alpha_j^+ \xi^i_j(n) - \beta_j^+ -1 \le
\xi^i_{j+1}(n) \le
\alpha^+_j \xi^i_j(n), & \quad
\forall j=1,\ldots,k \ \forall i \ge 0.
\end{array}
\right\}
\end{equation}
(A comment is in order.
 The careful reader may notice that, by the definition of the alphabet $\Qv$,
 $\zeta^0_j(n) \ge \chi_j$.
 Thus, it looks like we should have written the bound
\[
\zeta^0_{j+1}(n) \le
\max\big\{
\chi_{j+1},\alpha_j^- \zeta^0_j(n) + \beta_j^- + 1
\big\}, \qquad
\forall j=1,\ldots,k
\]
 instead of the first bound in~\eqref{eq:ZetaXiSingle}.
 However, we do not need it when $\chi_1 \gg \chi_2,\ldots,\chi_k$.
 Under this assumption, which is the second part of hypothesis~({\bf X}),
 we know that the first $k-1$ minima
 $\zeta^0_2(n),\ldots,\zeta^0_{k}(n)$ are not affected by restrictions
 $\xi_j^0(n) \ge \chi_j$, and,
 if necessary, we replace $\zeta^1_1(n) = \zeta_{k+1}^0(n)$
 ---which is the last minimum that we need to take care of---  by $\chi_1$.)

If we apply recursively $k$ times the bounds~\eqref{eq:ZetaXiSingle},
we get the cyclic bounds
\begin{equation}
\label{eq:ZetaXiCyclic}
\left.
\begin{array}{l}
\zeta^1_1(n) \le \max\{ \chi_1,\zeta^0_1(n)/\alpha + \gamma^-_1 \} \\
\alpha \xi^i_j(n) - \gamma^+_j \le \xi^{i+1}_j(n) \le \alpha \xi^i_j(n),
\quad \forall j=1,\ldots,k, \forall i \ge 0
\end{array}
\right\}
\end{equation}
where $\alpha = \prod_{j=1}^k \alpha_j^+$,
$1/\alpha = \prod_{j=1}^k \alpha_j^-$ and
\[
\gamma^\pm_j =
\sum_{m=1}^k \left( \prod_{l=1}^{m-1} \alpha_{j-l}^\pm \right)
\big( \beta_{j-m}^\pm + 1 \big),\qquad
\forall j=1,\ldots,k.
\]
\begin{enumerate}[i)]
\item
If we apply recursively $j-1$ times the bounds for the maxima
in the second line of equation~\eqref{eq:ZetaXiSingle}, we get
\begin{equation}
\label{eq:Xi}
\lambda_j n - \gamma_j = \lambda_j \xi^0_1(n) - \gamma_j \le
\xi^0_j(n) \le \lambda_j \xi^0_1(n) = \lambda_j n,
\qquad \forall j=1,\ldots,k,
\end{equation}
where $\lambda_1 = 1$, $\gamma_1 = 0$, and
\[
\lambda_j = \prod_{l=1}^{j-1} \alpha_l^+,\qquad
\gamma_j =
\sum_{m=1}^{j-1} \left( \prod_{l=1}^{m-1} \alpha_{j-l}^\pm \right)
\big( \beta_{j-m}^\pm + 1 \big),\qquad
\forall j=2,\ldots,k.
\]
If $\lambda = \max\{\lambda_1,\ldots,\lambda_k\}$
and we choose any $\nu$ such that
$0 < \nu < \min\{\lambda_1,\ldots,\lambda_k\}$,
then~\eqref{eq:Xi} implies that $\nu n \le \xi^0_j(n) \le \lambda n$
for all $j=1,\ldots,k$ provided that $\chi_1$ is large enough,
as it is assumed in hypothesis~({\bf X}).
To be precise, if we assume that $n \ge \chi_1 \gg 1$,
then $\nu n \le \lambda_j n - \gamma_j$ for all $j=1,\ldots,k$.
It suffices to take
\[
\chi_1 \ge \max \big\{ \gamma_j/(\lambda_j - \nu) : j=1,\ldots,k \big\}.
\]

\item
We assume that $i \ge 0$.
The upper bound $\xi^i_j(n) \le \alpha^i \xi^0_j(n)$ follows directly
from~\eqref{eq:ZetaXiCyclic}.
The lower bound $\xi^i_j(n) \ge \tau \alpha^i \xi^0_j(n)$
for some $\tau \in (0,1)$ is more tricky.
First, we realise that if we choose any $\kappa \in (1,\alpha)$,
then~\eqref{eq:ZetaXiCyclic} implies the weaker lower bound
$\xi^i_j(n) \ge \kappa^i \xi^0_j(n)$ provided
$n \ge
 \chi_1 \ge \max\big\{\gamma_1^+,\ldots,\gamma_k^+\big\}/(\alpha-\kappa)$.
This means that $\xi^i_j(n)$ grows geometrically as $i \to +\infty$.
Second, we know that
\[
\xi^i_j(n) \ge
\alpha \xi^{i-1}_j(n) - \gamma_j^+ =
\left(1 - \frac{\gamma_j^+}{\alpha \xi^{i-1}_j(n)} \right) \alpha \xi^{i-1}_j(n)
\ge \cdots \ge \tau_{i,j} \alpha^i \xi^0_j(n),
\]
where
\[
0 < \prod_{l=0}^{+\infty}
\left(1 - \frac{\gamma_j^+}{\alpha \xi^l_j(n)} \right) =:
\tau_j < \tau_{i,j} =
\prod_{l=0}^{i-1}
\left(1 - \frac{\gamma_j^+}{\alpha \xi^l_j(n)} \right) < 1, \qquad
\forall i \ge 0.
\]
The above infinite product converges to a non-zero value $\tau_j$
because
\[
\sum_{l=0}^{+\infty} \frac{\gamma_j^+}{\alpha \xi^l_j(n)} \le
 \frac{\gamma_j^+}{\alpha \xi^0_j(n)} \sum_{l=0}^{+\infty} \kappa^{-l} <
 +\infty.
\]
If we set $\tau = \min \{ \tau_1,\ldots,\tau_k \}$,
then $\xi^i_j(n) \ge \tau \alpha^i \xi^i_0(n)$.
This ends the proof for the forward case $i \ge 0$.
The backward case $i < 0$ is proved in a similar way.

\item
Inequality $\rho^i(n) \le \rho^{i+1}(n)$ is trivial.
Using the already proved parts~(\ref{item:AdmissibleSymbols1})
and~(\ref{item:AdmissibleSymbols2}) of this lemma
and the formula for geometric sums, we get
\begin{align*}
\frac{\rho^{i+1}(n)}{\xi^i_j(n)}
&\le
\frac{\sum_{j=1}^k \sum_{m=0}^i \alpha^m \xi^0_j(n)}
     {\tau \alpha^i \xi^0_j(n)} \le
\frac{\sum_{j=1}^k \sum_{m=0}^i \alpha^m \lambda n}{\tau \alpha^i \nu n} =
\frac{k\lambda (\alpha^{i+1} - 1)}{\tau \nu (\alpha - 1) \alpha^i} \\
&\le
\frac{k \lambda \alpha}{\tau \nu (\alpha - 1)} =: \lambda', \\
\frac{\rho^i(n)}{\xi^i_j(n)}
&\ge
\frac{\sum_{j=1}^k \sum_{m=0}^{i-1} \tau \alpha^m \xi^0_j(n)}
     {\alpha^i \xi^0_j(n)} \ge
\frac{\sum_{j=1}^k \sum_{m=0}^{i-1} \tau \alpha^m \nu n}{\alpha^i \lambda n} =
\frac{k \tau \nu (\alpha^i - 1)}{\lambda (\alpha - 1) \alpha^i} \\
&\ge
\frac{k \tau \nu}{\lambda \alpha} =: \nu'.
\end{align*}

\item
The inequalities $n/\alpha + \gamma^- \le n-1 < n+1 \le \alpha n - \gamma^+$
for all $n \ge \chi_1$ follow from hypotheses~({\bf B}) and~({\bf X}).
It suffices to take
\[
\chi_1 \ge
\max\{ (1 + \gamma^+)/(\alpha - 1), (1 + \gamma^-)/(1-1/\alpha) \}.
\]
Set $\gamma^\pm = \gamma^\pm_1$.
The inequalities $\zeta^1_1(n) \le \max\{ \chi_1, n/\alpha + \gamma^- \}$
and $\alpha n - \gamma^+ \le \xi^1_1(n)$ follow directly
by taking $i=0$ in~\eqref{eq:ZetaXiCyclic},
because $\zeta^0_1(n) = n = \xi^0_1(n)$ by definition.

\item
If we take
$n \ge
\max\{ (\chi_1 - \gamma^-)\alpha,
       (N + \gamma^+)/(\alpha - 1),
       (N + \gamma^-)/(1-1/\alpha) \}$,
then
$\chi_1 \le n/\alpha + \gamma^- \le n-N < n+N \le \alpha n - \gamma^+$.
\end{enumerate}
\item
Let us check that the sets $\Xi^i_j(n)$ have no gaps in $\Nset$.
That is, we want to check that $[n^-,n^+] \cap \Nset \subset \Xi^i_j(n)$
for all $n^\pm \in \Xi^i_j(n)$ such that $n^- \le n^+$.

First, we consider the forward case $i \ge 0$.
We prove it by induction on the ordering
\[
\Xi^0_1,\ldots,\Xi^0_k,
\Xi^1_1=\Xi^0_{k+1},\ldots,\Xi^1_k,
\Xi^2_1=\Xi^1_{k+1},\ldots,\Xi^2_k,
\ldots, \Xi^i_1=\Xi^{i-1}_{k+1},\ldots,\Xi^i_j,\Xi^i_{j+1},\ldots.
\]

The base case is trivial: $\Xi^0_1(n) = \{ n \}$.
Let us perform now the inductive step.
We assume that $\Xi^i_j(n)$ has no holes in $\Nset$
for some $i \ge 0$ and $1 \le j \le k$.
The next set is
\[
\Xi^i_{j+1}(n) =
\left\{ n'' \in \Nset :
n'' \ge \chi_{j+1}, \
\exists n' \in \Xi^i_j(n) \mbox{ s. t.
$\alpha_j^- n' + \beta_j^- < n'' < \alpha_j^+ n' - \beta_j^+$}
\right\}.
\]
If $\Xi^i_{j+1}(n)$ has a hole in $\Nset$,
there is $n' \ge \chi_j$ such that
$\alpha_j^+ n' - \beta_j^+ \le \alpha_j^- (n'+1) + \beta_j^-$,
which is impossible by hypotheses~({\bf A}) and~({\bf X}).
It suffices to take
\[
\chi_j > (\alpha_j^- + \beta_j^- + \beta_j^+)/(\alpha_j^+ - \alpha_j^-).
\]
Property $\big[ \max\{\chi_1,n-|i|\}, n+|i| \big] \cap \Nset
\subset \Xi^i_1(n)$ for all $i \in \Zset$ and $n \ge \chi_1$
follows by induction from part~(\ref{item:AdmissibleSymbols4})
of this lemma and the fact that $\Xi^i_1(n)$ has no gaps in $\Nset$.

This ends the proof for the forward case $i \ge 0$.
The backward case $i < 0$ is similar.
\qedhere
\end{enumerate}

\section{Proof of Proposition~\ref{prop:Optimalc}}
\label{app:Optimalc}

Fix any $p \in \Nset$.
We look for the optimal value of $c_\star(p) > 0$ such that
\[
\# \Pi(p,q) \ge
2^{kp} G_q\big( P^{(p)} \big) \ge
c_\star(p) q^n + \Order(q^{n-1})\quad
\mbox{ as $q \to +\infty$}.
\]
Therefore, we want to count as many integer points as possible in
$P^{(p)} \subset \Rset^{n+1}$ whose coordinates sum $q \in \Nset$.
We shall put these points in a 1-to-1 correspondence with the integer
points of a $q$-dilated bounded convex polytope of $\Rset^n$ by means
of a projection.

We shall use a lower bound established by Wills~\cite{Wills1973}.
Let us briefly describe it.
If $t > 0$ and $P \subset \Rset^n$,
then $tP = \{ y\xv : \xv \in P \}$ and $P/t = \{ \xv/t : \xv \in P \}$
are the \emph{$t$-dilation} and \emph{$t$-contraction} of $P$.
The \emph{inradius} $\varrho(K)$ of a proper compact convex set
$K \subset \Rset^n$ is the biggest number $\varrho > 0$ such that
$K$ contains a ball of radius $\varrho$.
Note that $0 < \varrho(K) < \infty$ for any proper compact $K$.

\begin{lemma}
\label{lem:WillsLowerBound}
If $K$ is a proper compact convex subset of $\Rset^n$, then
\[
\# (tK \cap \Zset^n) \ge
\Volume(K) \big(t - \sqrt{n}/2\varrho(K) \big)^n, \qquad
\forall t \ge \sqrt{n}/2\varrho(K).
\]
\end{lemma}

\begin{proof}
The case $t=1$ is proved in~\cite{Wills1973},
assuming that $\varrho(K) \ge \sqrt{n}/2$.
The general case follows directly from this case since
$tK$ is a proper compact convex subset of $\Rset^n$,
$\Volume(tK) = t^n \Volume(K)$,
and $\varrho(tK) = t\varrho(K) \ge 2/\sqrt{n}$
if $t \ge \sqrt{n}/2\varrho(K)$.
\end{proof}

The convex polytope~\eqref{eq:UnboundedPolytope} is not closed,
so the convex polytopes $\tilde{P}^{(p)}_q$ defined
in Proposition~\ref{prop:Optimalc} are not closed either.
However, they are the projection of some convex polytopes contained
in the open simplex
$\Delta_n = \{\xv \in \Rset^{n+1} : \xv > 0, \ x_1 + \cdots + x_n = 1\}$,
which implies that they are bounded.
Hence we need to extend Lemma~\ref{lem:WillsLowerBound} to
proper bounded convex subsets of $\Rset^n$.

\begin{corollary}
\label{cor:WillsLowerBound}
If $P$ is a proper bounded convex subset of $\Rset^n$ and $K = \bar{P}$,
then
\[
\# (tP \cap \Zset^n) \ge
\Volume(K) \big(s - \sqrt{n}/2\varrho(K) \big)^n, \qquad
\forall t > s \ge \sqrt{n}/2\varrho(K).
\]
\end{corollary}

\begin{proof}
The closure $K = \bar{P}$ is compact.
Let $\bar{B}$ be a closed ball of radius $\varrho(K) > 0$
contained in $K$.
Let $B = \Interior \bar{B}$.
Given any point $-\xv \in B$, we have that
$s(\xv + K) \subset t(\xv + P)$ for all $t > s > 0$.
If $t > \sqrt{n}/2\varrho(K)$, then there is a point
$-\xv_t \in B$ such that $t \xv_t \in \Zset^n$.
Then
\begin{align*}
\# (tP \cap \Zset^n) &=
\# \big( (t\xv_t + tP) \cap \Zset^n \big) =
\# \big( t(\xv_t + P) \cap \Zset^n \big) \ge
\# \big( s(\xv_t + K) \cap \Zset^n \big) \\
&\ge
\Volume(\xv_t + K)\big(s- \sqrt{n}/2\varrho(\xv_t + K)  \big)^n =
\Volume(K) \big(s- \sqrt{n}/2\varrho(K)  \big)^n,
\end{align*}
for all $t > s \ge \sqrt{n}/2\varrho(K)$.
\end{proof}

\begin{proof}[Proof of Proposition~\ref{prop:Optimalc}.]
\begin{enumerate}[(a)]
\item
Let $H_{n+1} =
\left\{ \xv \in \Rset^{n+1} : x_1 + \cdots + x_{n+1} = 1 \right\}$.
The cardinality of a finite set is invariant under
$q$-dilations, $q$-contractions, and 1-to-1 projections.
Thus,
\begin{align*}
\# \Pi(p,q) &\ge 2^{kp} G_q\big( P^{(p)} \big) \\
&=
2^{kp} \# \left\{
\xv = (x_1,\ldots,x_{n+1}) \in P^{(p)} \cap \Zset^{n+1} :
x_1 + \cdots + x_{n+1} = q
\right\} \\
&=
2^{kp} \# \big( P^{(p)} \cap \Zset^{n+1} \cap q H_{n+1} \big) \\
&=
2^{kp} \# \big( (P^{(p)}/q) \cap (\Zset^{n+1}/q) \cap H_{n+1} \big) \\
&=
2^{kp} \# \big( P^{(p)}_q \cap (\Zset^{n+1}/q) \big) \\
&=
2^{kp} \# \big( q P^{(p)}_q \cap \Zset^{n+1} \big) \\
&=
2^{kp} \# \big( q \tilde{P}^{(p)}_q \cap \Zset^n \big) \\
&\ge
2^{kp} \Volume\big( \tilde{K}^{(p)}_q \big)
\Big(q - 1 - \sqrt{n}/2\varrho\big( \tilde{K}^{(p)}_q \big) \Big)^n \\
&\ge
2^{kp} V \big( \tilde{K}^{(p)}_{\infty} \big) q^{kp-1} + \Order(q^{kp-2})
\quad \mbox{ as $q \to +\infty$},
\end{align*}
where $\tilde{K}^{(p)}_q$ is the closure of $\tilde{P}^{(p)}_q$.
We have used Corollary~\ref{cor:WillsLowerBound} with
$t = q$ and $s=q-1$ in the second to last inequality.
In the last inequality, we have used estimates
\[
V\big( \tilde{K}^{(p)}_q \big) =
 V\big( \tilde{K}^{(p)}_\infty \big) + \Order(1/q),\qquad
\varrho\big( \tilde{K}^{(p)}_q \big) =
 \varrho\big( \tilde{K}^{(p)}_\infty \big) + \Order(1/q).
\]
These estimates follow from the fact that each facet of
the limit compact polytope $\tilde{K}^{(p)}_\infty$ is at an
$\Order(1/q)$-distance of the corresponding facet of
the polytope $\tilde{K}^{(p)}_q$,
which can be easily seen by comparing the half-space
representation~\eqref{eq:ProjectedLimitPolytope} of $\tilde{K}^{(p)}_\infty$
with the half-space representation
\begin{equation}
\label{eq:ProjectedNoLimitPolytope}
\tilde{K}^{(p)}_q =
\left\{
\tilde{\xv} \in \Rset^n :
\begin{array}{l}
\alpha^-_j x_j + \beta^-_j/q \le x_{j+1} \le \alpha^+_j x_j - \beta^+_j/q,
\quad \forall j=1,\ldots,n-1 \\
\alpha^-_n x_n + \beta^-_n/q \le 1 - \varsigma(\tilde{\xv}) \le
\alpha^+_n x_n - \beta^+_n/q \\
\alpha^-_{n+1} (1 - \varsigma(\tilde{\xv})) + \beta^-_{n+1}/q \le x_1 \le
\alpha^+_{n+1} (1 - \varsigma(\tilde{\xv})) - \beta^+_{n+1}/q\\
x_j \ge \chi_j/q, \quad \forall j=1,\ldots,n \\
\varsigma(\tilde{\xv}) \le 1 - \chi_{n+1}/q
\end{array}
\right\}.
\end{equation}

\item
All convex bodies are Jordan measurable and
\[
\Volume( J ) =
\lim_{q \to +\infty} q^{-n} \# \left( J \cap (\Zset^n/q) \right) =
\lim_{q \to +\infty} q^{-n} \# \left( qJ \cap \Zset^n \right)
\]
for any Jordan measurable set $J \subset \Rset^n$, see~\cite[section 7.2]{Gruber2007}.
Therefore,
\begin{align*}
\lim_{q \to +\infty} q^{-n} G_q\big( P^{(p)} \big) &\le
\lim_{q \to +\infty}
q^{-n} \# \big (q\tilde{K}^{(p)}_q \cap \Zset^n \big) \\
&\le
\lim_{q \to +\infty}
q^{-n} \# \big (q\tilde{K}^{(p)}_\infty \cap \Zset^n \big) =
\Volume\big( \tilde{K}^{(p)}_\infty \big), \\
\lim_{q \to +\infty} q^{-n} G_q\big( P^{(p)} \big) &\ge
\lim_{q \to +\infty}
\Big( V \big( \tilde{K}^{(p)}_{\infty} \big) + \Order(1/q) \Big) =
\Volume\big( \tilde{K}^{(p)}_\infty \big).
\end{align*}
We have used that $\tilde{K}^{(p)}_q \subset \tilde{K}^{(p)}_\infty$
---compare half-space representations~(\ref{eq:ProjectedNoLimitPolytope})
   and~(\ref{eq:ProjectedLimitPolytope})---
in the first line and the lower bound obtained at the beginning of
this proof in the second one.
\item
It is a simple computation using that
$x_{n+1} = 1 - \varsigma(\tilde{\xv})$
when $\xv = (\tilde{\xv},x_{n+1}) \in H_{n+1}$.
\qedhere
\end{enumerate}
\end{proof}

\bibliographystyle{abbrv}

\end{document}